\documentclass[10pt]{article}
\pdfoutput=1
\usepackage[utf8]{inputenc}
\usepackage[T1]{fontenc}
\usepackage{amsfonts}
\usepackage{amsmath}
\usepackage{amssymb}
\usepackage[english,french]{babel}   
\usepackage[all]{xy}
\usepackage{pinlabel}
\usepackage{graphicx}
\usepackage{xcolor}

\usepackage{enumitem}
\usepackage{amsmath}
\usepackage{amsthm}

\usepackage[normalem]{ulem}

\usepackage{geometry}
\usepackage[pdftex]{hyperref}
\hypersetup{linkcolor=blue,colorlinks=true,citecolor=magenta}

\let\oldmarginpar\marginpar
\renewcommand\marginpar[1]{\-\oldmarginpar[\raggedleft\footnotesize #1\\]%
{\raggedright\footnotesize #1}}

\newcommand{\esp}{~~~~~}

\newcommand{\A}{{\bf A}}
\newcommand{\C}{{\bf C}}
\newcommand{\D}{{\bf D}}
\newcommand{\HH}{{\bf H}}
\newcommand{\GL}{{\bf GL}}
\newcommand{\K}{{\bf K}}
\newcommand{\M}{{\bf M}}
\newcommand{\N}{{\bf N}}

\newcommand{\PP}{{\bf P}}
\newcommand{\PSL}{{\bf PSL}}

\newcommand{\R}{{\bf R}}
\newcommand{\SSS}{{\bf S}}
\newcommand{\T}{{\bf T}}
\newcommand{\Z}{{\bf Z}}

\newcommand{\al}{\alpha} \newcommand{\be}{\beta} 
\newcommand{\ga}{\gamma} \newcommand{\Ga}{\Gamma} 
\newcommand{\de}{\delta} \newcommand{\De}{\Delta}
 \newcommand{\eps}{\epsilon} 
 
\newcommand{\la}{\lambda}  
\newcommand{\ro}{\rho}   
\newcommand{\sig}{\sigma} \newcommand{\Sig}{\Sigma} 
 \newcommand{\ph}{\varphi} 
\newcommand{\om}{\omega}
\newcommand{\Om}{\Omega} 

\newcommand{\mcA}{{\mathcal A}}
\newcommand{\mcC}{{\mathcal C}}
\newcommand{\mcD}{{\mathcal D}}
\newcommand{\mcE}{{\mathcal E}}
\newcommand{\mcG}{{\mathcal G}}
\newcommand{\mcK}{{\mathcal K}}
\newcommand{\mcL}{{\mathcal L}}
\newcommand{\mcP}{{\mathcal P}}
\newcommand{\mcQ}{{\mathcal Q}}
\newcommand{\mcR}{{\mathcal R}}
\newcommand{\mcX}{{\mathcal X}}

\newcommand{\mfE}{\mathfrak{E}}
\newcommand{\mfK}{\mathfrak{K}}
\newcommand{\mfX}{\mathfrak{X}}

\newcommand{\msfk}{{\mathsf k}}

\newcommand{\msfl}{{\mathsf l}}

\newcommand{\msfm}{{\mathsf m}}

\newcommand{\msfR}{{\mathsf R}}

\def\makeop#1{
\expandafter\def\csname #1\endcsname{\mathop{\mathrm{#1}}\nolimits}}

\def\Isom{\mathop{\rm Is}\nolimits}
\def\Isomgen{\mathop{\rm Is_{gen}}\nolimits}
\def\Isomsel{\mathop{\rm Is_{sel}}\nolimits}
\def\pizeroIsom{\mathop{\pi_0\rm Is}\nolimits}

\def\gen{{\rm gen}}
\def\carac{\ro}

\makeop{Aut}
\makeop{Der}
\makeop{Hom}
\makeop{Diff}
\makeop{cte}
\makeop{Id}

\makeop{sn}
\makeop{cn}
\makeop{dn}
\makeop{sd}
\makeop{nd}

\def\aut{\mathop{\rm Aut}\nolimits}

\def\ker{\mathop{\rm Ker}\nolimits}

\def\id{\mathop{\rm Id}\nolimits}

\def\card{\mathop{\rm Card}\nolimits}

\theoremstyle{plain}
\newtheorem{theo}{\indent\sc{Théorème}}[section]
\newtheorem*{theo*}{\indent\sc{Théorème}}
\newtheorem{prop}[theo]{\indent\sc{Proposition}}
\newtheorem{coro}[theo]{\indent\sc{Corollaire}}
\newtheorem{lemm}[theo]{\indent\sc{Lemme}}

\newtheorem*{affi*}{\indent\sc{Affirmation}}

\theoremstyle{definition}
\newtheorem{defi}[theo]{\indent\sc{Définition}}
\newtheorem*{defi*}{\indent\sc{Définition}}

\theoremstyle{remark}
\newtheorem{rema}[theo]{\indent Remarque}
\newtheorem{remas}[theo]{\indent Remarques}
\newtheorem{exem}[theo]{\indent Exemple}
\newtheorem{exems}[theo]{\indent Exemples}
\newtheorem*{exem*}{\indent Exemple}
\newtheorem*{exems*}{\indent Exemples}

\newcounter{rmq}

\begin{document}

\title{Extensions maximales et classification des tores lorentziens
munis d'un champ de Killing}
\author{Ch Bavard, P Mounoud}
\date{15 janvier 2016}
\maketitle
\selectlanguage{english}

\begin{abstract}
We study the simply connected inextendable  Lorentzian surfaces 
admitting a Killing
vector field. We construct a natural family of 
such surfaces, that we call ``universal extensions''. They are
characterized by a condition of symmetry, the ``reflexivity'', and a by a
rather weak completeness assumption, the absence of "saddles at infinity".
Considering these surfaces as model spaces, we study their minimal
quotients, divisible open sets and conjugate points.  We show
uniformisation results (by an open subset of one of these universal
extensions, which is uniquely determined) in the following cases: compact
surfaces and analytical surfaces.  It allows us to give a classification of
Lorentzian tori and Klein bottles with a Killing vector field.
\end{abstract}

\selectlanguage{french}
\renewcommand{\proofname}{\indent Preuve.}
\begin{abstract}
Nous étudions les surfaces lorentziennes simplement connexes, 
inextensibles et possédant un champ
de Killing. 
Nous introduisons une famille naturelle de telles surfaces,  désignées par \og
extensions universelles\fg.  Elles sont 
 caractérisées par une condition de
symétrie, la \og réflexivité\fg, et par une condition de complétude assez
faible, l'absence de \og selles à l'infini\fg. Ces surfaces jouent le rôle
d'espaces modèles : nous en étudions les quotients, les ouverts divisibles
et les points conjugués.  Nous établissons des résultats d'uniformisation
(par un ouvert de l'une de nos extensions universelles, uniquement
déterminée) dans les deux cas suivants : surfaces compactes, surfaces
analytiques.  Cela nous permet notamment d'obtenir une classification des
tores lorentziens et des bouteilles de Klein possédant un champ de Killing.
\end{abstract}

\section{Introduction}
L'étude des variétés possédant beaucoup d'isométries est un thème classique
en géométrie riemannienne et pseudo-riemannienne. Nous nous intéressons ici
à la géométrie globale des surfaces lorentziennes admettant un champ de
Killing complet, c'est-à-dire dont le groupe d'isométrie n'est pas
discret. La question de la classification de ces surfaces est abordée en
privilégiant le point de vue des structures géométriques, cadre dans lequel
les objets sont localement modelés par un pseudo-groupe de transformations 
d'un espace modèle.
 Parmi toutes les surfaces dont la
géométrie locale est donnée, en un sens qui sera précisé plus bas, nous
exhibons un \og objet universel \fg, noté pour l'instant $E^u$,
jouant le rôle de modèle et 
généralisant les modèles classiques bien connus en courbure constante. Il
s'agit d'une surface lorentzienne simplement connexe et maximale
c'est-à-dire inextensible, que l'on peut caractériser par des propriétés de
symétrie et de complétude (voir théorème \ref{theo:les_Eu_intro}).  Sa
construction met en évidence un phénomène qui contraste avec le cas
riemannien : l'existence d'une grande variété (géométrique et
topologique) d'exemples de surfaces munies d'un champ de Killing, y
compris parmi les quotients du modèle lui-même.  Dès lors se pose la
question de trouver de \og bonnes\fg\ classes de surfaces {\em
uniformisées} par le modèle, c'est-à-dire dont le revêtement universel est
isométrique à un ouvert du modèle {\it via} une application
développante. Notre surface universelle $E^u$ conduit à des résultats
d'uniformisation dans plusieurs contextes significatifs : surfaces
compactes, surfaces analytiques. Dans le cas compact (tores et bouteilles
de Klein), nous en déduisons une classification assez précise des objets.

Parmi les résultats sur la géométrie des variétés  ayant un \og gros\fg\ groupe
d'isométrie, on peut citer comme prototype le théorème de Ferrand \cite{Ferrand}
en géométrie conforme riemannienne. Le terme \og gros\fg\ est bien sûr à
préciser mais il signifie a minima \og{\em qui n'agit pas proprement sur la
variété}\fg. Notre travail se rapproche notamment de \cite{Monclair}, dans
lequel l'auteur étudie les surfaces lorentziennes globalement hyperboliques, et
de \cite{Piccione-Zeghib} où les auteurs étudient les variétés lorentziennes
compactes munies d'un champ de Killing -- ayant une orbite de type temps -- 
dont le groupe d'isométrie a une infinité de composantes connexes. 
Notre cadre de  travail est un mélange des précédents : 
nous étudions les surfaces munies d'un champ de Killing sans nous limiter
aux surfaces compactes ou globalement hyperboliques. 
Précisons qu'en dehors de  cas élémentaires, l'action du 
flot d'isométries de notre modèle  $E^u$ n'est pas propre.

Les surfaces modèles lorentziennes  de courbure constante ont évidemment
un gros groupe d'isométrie, mais on peut vérifier qu'elles ne possèdent que peu de quotients. À
l'inverse, pour une géométrie locale suffisamment générale, notre surface modèle
$E^u$ admet de nombreux quotients et leur topologie est très variée. Nous
verrons en effet que $E^u$ possède, en plus de son flot d'isométries, une
combinatoire discrète de type hyperbolique, c'est-à-dire admet des sous-groupes
discrets d'isométries dont la dynamique est conjuguée à celle d'un groupe
fuchsien sur le plan.

Notre modèle est obtenu par extension de la géométrie locale.
La notion d'extensibilité joue un rôle fondamental en géométrie 
lorentzienne. Dans un travail précédent \cite{BM},
nous avons utilisé de façon essentielle la possibilité d'étendre le revêtement
universel du tore de Clifton-Pohl en une surface lorentzienne dont toutes les
géodésiques de type lumière sont complètes (on dira L-complète).  On notera
aussi que la construction de l'extension de Kruskal de la métrique de
Schwarzschild se ramène à un problème d'extension d'une surface munie d'un champ
de Killing, \cite{Oneill}.  Nous allons voir que ces deux extensions peuvent
être obtenues par le même procédé général.

Décrivons plus en détails le procédé évoqué ci-dessus, qui repose  sur deux
observations clés \footnote{Ce mode
de construction de surfaces munies d'un champ de Killing a été repris dans
\cite{MS}, il permet de donner des familles de surfaces dont toutes les
géodésiques de type espace sont fermées.}. Premièrement, la présence d'un champ
de Killing entraîne généralement l'existence d'isométries locales
supplémentaires : les réflexions par rapport aux géodésiques de type temps ou
espace perpendiculaires au champ de Killing, que l'on appellera \emph{réflexions
génériques}. Ces réflexions permutent les feuilletages de lumière.  
Dès qu'elles ne sont pas
globalement définies, elle  permettent 
 d'étendre la surface en collant deux
copies de celle-ci. Ensuite, nous cherchons des extensions aussi complètes que
possible ; en particulier toute extension  L-complète sera maximale.
Les orbites de lumière du champ de Killing sont généralement incomplètes
(géodésiquement) et le procédé d'extension par réflexion ne corrige pas ce
défaut. La deuxième observation clé est que l'on peut, dans certaines
conditions, adjoindre des points selles, c'est-à-dire des zéros du champ de
Killing, afin de prolonger ses orbites de lumière incomplètes en géodésiques 
complètes.

Soit $X$ une surface lorentzienne munie d'un champ de Killing $K$, que nous
supposerons toujours non trivial et {\em complet}. Localement, l'un des
feuilletages de lumière est défini par un champ $L$ tel que $\langle K, L
\rangle = 1$. La métrique se met alors sous la forme
\setcounter{equation}{-1}
\begin{equation}
\label{equa:intro}
2dxdy+f(x)dy^2,
\end{equation}
 avec $L=\partial_x$ et $K=\partial_y$ (voir
 lemme~\ref{lemm:carte_adaptee}). La coordonnée locale $x$, bien définie
 modulo translation et changement de signe, sera appelée {\em coordonnée
   transverse}. Ainsi, la fonction~$f$ exprime la norme du champ de Killing
 dans la coordonnée transverse, laquelle norme étant évidemment un
 invariant de la géométrie locale du couple $(X,K)$. Pour le tore de
 Clifton-Pohl, on aura $f(x)=\sin(2x)$ et pour la métrique de Schwarzschild
  $f(x)=1-2M x^{-1}$
(avec $M>0$, coordonnées dites de Eddington-Finkelstein, voir \cite{Penrose}).
La courbure vaut $f''(x)/2$. Dans toute la suite de l'introduction, {\em nous
écartons le cas %
de la courbure constante.} Le champ de Killing est
alors unique à homothétie près (voir la preuve du lemme~\ref{lemm:isom_isomK}).
Les changements de la coordonnée transverse et du champ de Killing se traduisent
par une action -- à droite -- du groupe affine de la droite qui consiste à
remplacer $f(x)$ par $a^{-2}f(ax+b)$ pour $(a,b)\in \R^*\times \R$. La classe de
$f$ modulo cette action sera notée $[[f]]$, celle de l'action du sous-groupe
donné par $a^2=1$ sera notée $[f]$. Les classes $[f]$  et $[[f]]$ sont des
invariants de la géométrie locale de $(X,K)$ et de $X$ respectivement.

Afin de définir une notion plus précise de \og géométrie locale\fg, nous
supposons dans un premier temps que $X$ est simplement connexe. L'espace
$\mcE_X$ des orbites non triviales du champ $K$ est alors une variété connexe de
dimension~1, généralement non séparée ; de plus, toute coordonnée transverse
locale se globalise en une fonction lisse $x\in \mcC^\infty(X,\R)$ qui induit un
difféomorphisme local de $\mcE_X$ sur un intervalle $I$ de $\R$,
proposition~\ref{prop:struct_transv_killing}-(2). Soit $f\in
\mcC^\infty(I,\R)$. Nous
dirons que {\em $(X,K)$ est de classe $[f]$}, ou que {\em $X$ est de classe
$[[f]]$} quand on veut oublier le champ de Killing, si $I=x(X)$ et 
$\langle K, K \rangle = f\circ x$. 
Plus généralement, si~$X$ n'est pas simplement connexe, $(X,K)$ est
dite de classe $[f]$ si son revêtement universel~$\smash{\widetilde{X}}$ 
l'est. Cette condition signifie
que la géométrie locale de la surface est partout déterminée, {\it via}  la
norme du champ,  par une fonction  définie sur un intervalle -- optimal --
de la droite
numérique ; on dira que {\em la géométrie locale est uniforme}. La condition est
satisfaite dans deux cas importants : les surfaces compactes
($\mcE_{\widetilde{X}}$ est
séparé) et les surfaces analytiques. Il faut noter qu'en général la norme du
champ, comme fonction sur l'espace des orbites, ne se factorise pas de
telle sorte.

On fixe  maintenant une géométrie locale (au sens précédent) en se donnant une
fonction $f\in \mcC^\infty(I,\R)$. La surface $R_f= I\times \R$ munie de la
métrique~\eqref{equa:intro} est la surface la plus simple de classe $[f]$. Un
des résultats principaux de cet article est l'existence d'une extension
privilégiée de $R_f$, notée $E^u_f$. Cette surface est munie d'un champ de
Killing, simplement connexe, maximale si $f$ est inextensible (en particulier
L-complète si $f$ est définie sur $\R$) et de classe $[f]$.  Dans le cas
particulier où $f$ ne s'annule pas (on dira que {\em $f$ est élémentaire}), on
prend $E^u_f=R_f$, qui est effectivement
 maximale lorsque~$f$ est inextensible.  Quand $f$ s'annule, les
réflexions génériques ne sont pas globalement définies. La surface $E^u_f$
s'obtient alors à partir de $R_f$ grâce aux deux opérations évoquées plus haut :
{\em extension par réflexion} et {\em adjonction de selles}.

Un  moyen de produire d'autres  extensions de $R_f$ consisterait  à  perturber
$E^u_f$ en dehors de $R_f$, tout en conservant un champ de Killing. Une telle
surface ne serait évidemment plus de classe $[f]$. Cependant, en utilisant des
réflexions non génériques, on peut trouver des extensions de $R_f$ qui
ressemblent beaucoup plus à $E^u_f$.
Ainsi les surfaces de l'exemple~\ref{exem:sf_unif_non_refl} sont maximales, 
recouvertes (à l'exception des
éventuels points selles) par des ouverts isométriques à $R_f$,
mais ne sont pas de classe~$[f]$. Pour une surface
simplement connexe maximale, \^etre à géométrie locale uniforme
équivaut à  une propriété de symétrie : {\em la réflexivité}.
Soit~$X$ simplement connexe munie d'un champ de Killing et soit $Y$ un ouvert de
$X$ isométrique à $R_g$ pour une certaine fonction~$g$. La réflexivité
(définition~\ref{defi:sf_refl_unif}) stipule que les réflexions génériques
de~$Y$ s'étendent en des  isométries locales définies au moins sur $Y$
(ce qui fait défaut aux exemples~\ref{exem:sf_unif_non_refl}). 
On verra que 
les réflexions génériques de $E^u_f$ sont même globalement définies,
lemme~\ref{lemm:Euf_reflexive}.

Supposons que $f$ est inextensible et qu'elle s'annule, de sorte que $E^u_f$ est
une extension de classe $[f]$, simplement connexe, maximale et propre de $R_f$.
On s'aperçoit à nouveau qu'une telle extension n'est pas forcément unique. En
effet, le revêtement universel de $E^u_f$ privé d'un point selle (s'il en
existe) est une surface simplement connexe maximale de classe $[f]$ contenant
des copies de $R_f$ et non isométrique à $E^u_f$.  Pour une fonction $f$
suffisamment générique, on sait même construire, en répétant cette seule
transformation, une infinité non dénombrable d'extensions maximales de $R_f$
deux à deux non isométriques et toutes de classe $[f]$, voir la fin du
\S\ref{subs:uniformisation}. Contrairement à $E^u_f$, ces surfaces possèdent par
construction des géodésiques de lumière incomplètes portées par des trajectoires
du champ de Killing :  on dira abusivement qu'elles ont {\em des selles à
l'infini}. Cette terminologie se justifie par l'incomplétude géodésique des
trajectoires du champ qui tendent vers un point selle. Par exemple, dès que $f$
admet des zéros simples, le bord de $R_f$ dans $E^u_f$ contient des points
selles. L'absence de selles à l'infini (définition~\ref{defi:selle_a_l_infini})
apparaît ainsi comme une version affaiblie de la  L-complétude.  Nous pouvons
maintenant énoncer le résultat d'existence et d'unicité de nos surfaces modèles.

\begin{theo}[proposition~\ref{prop:sf_Euf} et 
théorème~\ref{theo:uniformisation}]
\label{theo:les_Eu_intro}
Pour toute fonction $f\in \mcC^\infty(I,\R)$ inextensible, il existe une unique
(à isométrie près) surface lorentzienne lisse $E^u_f$ munie d'un champ de
Killing complet $K^u$ et caractérisée par les propriétés suivantes :
$(E^u_f,K^u)$ est simplement connexe, maximale  (L-complète si $I=\R$), 
de classe~$[f]$  et sans selles à l'infini.
\end{theo}

Ce résultat est valable y compris en courbure constante ($f''=0$). De plus,
si l'on écarte ce cas, {\em  deux
surfaces $E^u_f$ et $E^u_g$  ($f,g\in
\mcC^\infty(I,\R)$) sont isométriques si et seulement si $[[f]]=[[g]]$}. 

%\vspace{10pt}
\begin{coro}
Toute surface lorentzienne analytique $(X,K)$ (avec $K$ non 
trivial et complet), simplement connexe et L-complète  est
isométrique à une surface $E^u(f)$ pour une certaine fonction analytique 
$f$ définie sur $\R$.
\end{coro}

La surface $E^u_f$ sera considérée comme espace modèle
pour la géométrie locale de classe~$[f]$. 
Nous appellerons également $E^u_f$  {\em extension universelle 
associée à $f$},  car elle a vocation à se substituer au revêtement
universel dans les
questions de classification,  ou encore (comme dans {\cite{BM}) pour analyser la
géométrie globale. À ce stade, l'étude du groupe d'isométrie du modèle $E^u_f$
s'impose.

\begin{theo}[voir théorème~\ref{theo:groupe_isom_Eu} et
proposition~\ref{prop:isom_Eu_elem}]
Le groupe d'isométrie de $E^u_f$ (supposée à courbure non constante) est le
produit semi-direct du sous-groupe distingué engendré par le flot du champ
et d'un sous-groupe discret $G$, non unique.
\end{theo}

Le groupe $G$ se décompose lui-même en produit semi-direct d'un sous-groupe
distingué \og générique\fg\ $G_\gen$, engendré par des réflexions génériques,
avec un sous-groupe (non unique) isomorphe au groupe de symétrie de $f$. La
structure algébrique de $G$ -- en particulier sa taille~-- dépend évidemment de
$f$. Plus précisément, si $Z_f$ désigne l'ensemble des zéros de $f$, le groupe
$G_\gen$ est un groupe de Coxeter dont les générateurs correspondent
bijectivement aux composantes connexes de $I\smallsetminus Z_f$,
proposition~\ref{prop:G_gen_Coxeter}. En particulier, si $f$ n'est pas
périodique, le groupe $G$ est virtuellement de Coxeter. Concernant les aspects
dynamiques, nous établissons, proposition~\ref{prop:action_Gprime}, que $G$ agit
proprement sur $E^u_f$ si et seulement si les composantes de $I\smallsetminus
Z_f$ ne s'accumulent pas dans $I$ (on dira que {\em $f$ est de type fini}). Cela
nous permet de construire,  remarque~\ref{rema:cas_bonus},  une surface dont le
groupe d'isométrie n'agit pas proprement bien qu'il soit de $2$-torsion. À titre
de comparaison, indiquons qu'il est montré dans \cite{Monclair} qu'un groupe
d'isométrie d'une surface lorentzienne globalement hyperbolique Cauchy compacte
n'agissant pas proprement  est toujours isomorphe à un sous-groupe d'un
revêtement fini de $\PSL_2(\R)$. Dans le cas  analytique (à courbure variable), nous montrons  (corollaire \ref{coro:certain}) que   le groupe d'isométrie de toute 
surface simplement connexe $(X,K)$ est le produit semi-direct de sa composante neutre par un sous-groupe agissant proprement (comparer avec \cite{Piccione-Zeghib}).

Lorsque $f$ est de type fini, l'action de $G$ sur $E^u_f$ est
différentiablement conjuguée à celle d'un groupe (virtuellement) fuchsien sur le
demi-plan de Poincaré ; nous décrivons précisément (en fonction de $f$) ce
groupe dont le sous-groupe générique est un groupe de réflexions hyperboliques,
proposition~\ref{prop:refl_hyp}. On peut donc considérer que  la surface $E^u_f$
est alors un objet dont la combinatoire est de nature hyperbolique. Nous
montrons que $G$ possède toujours des sous-groupes sans
torsion d'indice au plus $4$,
 correspondant aux plus petits quotients lisses de $E^u_f$. Nous
explicitons le nombre et la topologie de ces quotients, 
proposition~\ref{prop:top_quo_Eu}, ainsi que l'espace 
des déformations de ceux-ci,
proposition~\ref{prop:composantes_deformations}. La quantité, tant d'un point de
vue topologique que géométrique, de surfaces obtenues n'a rien à voir avec la
situation analogue en géométrie riemannienne : toutes les surfaces non compactes
de type fini apparaissent et les espaces de  déformation sont de dimension
arbitrairement grande (table~\ref{tabl:petits_quotients} et 
proposition~\ref{prop:composantes_deformations}).
Pour une géométrie locale donnée, suffisamment générale, on peut même réaliser
tous les types topologiques finis et orientables, 
proposition~\ref{prop:top_classe_fixee}.

Le rôle d'espace modèle, ou d'extension universelle, joué par la surface
$E^u_f$ est précisé par le résultat suivant (voir les 
théorèmes~\ref{theo:uniformisation}
et~\ref{theo:uniformisation_cas_periodique}).

\begin{theo}[uniformisation]
\label{theo:unif_intro}
Soit $X$ une surface lorentzienne connexe munie d'un champ de Killing~$K$ 
non trivial et complet.
\begin{enumerate}
 \item Soit $f\in \mcC^\infty(I,\R)$. Si $(X,K)$ est de classe $[f]$, alors 
$X$ est modelée sur $E^u_f$, en particulier son revêtement 
 universel est étalé au-dessus de $E^u_f$.
\item La surface $X$ est uniformisée par un ouvert d'une extension universelle (associée 
à une certaine fonction $f$)  dans chacun des cas suivants : 
\begin{enumerate}
\item %
$(X,K)$ est à géométrie locale uniforme et sans selles à l'infini,
\item %
le flot de $K$ est périodique.
\end{enumerate}
\end{enumerate}
En particulier, les surfaces analytiques sans selles à l'infini et les surfaces
compactes sont uniformisées par un ouvert d'une surface $E^u_f$.
\end{theo}

Comme nous l'avons déjà observé, les surfaces analytiques et les surfaces
compactes (tores et bouteilles de Klein) admettent une géométrie locale
uniforme, de type fini dans le cas analytique et périodique dans le cas compact.
Sur les surfaces compactes, les champs de Killing sont toujours périodiques (à
l'exception de certains champs dans le cas plat). Sans être incompatibles, les
conditions (2a) et (2b) sont assez opposées ; ainsi pour
un tore, dès que $f$ possède un zéro simple, le flot du champ est périodique
mais la surface admet des selles à l'infini. Les deux résultats
d'uniformisation ci-dessus sont donc complémentaires.
Quand~$f$ est de type fini, par exemple analytique, l'étude des surfaces
maximales uniformisées par $E^u_f$ se ramène à celle des ouverts maximaux de
discontinuité, décrits par la proposition~\ref{prop:ouv_disc}. L'exemple du
revêtement universel de $E^u_f$ privée de points selles  montre que l'on ne
peut espérer un théorème d'uniformisation général, même dans le cas analytique.

Revenons  pour finir au cas compact. L'étude des espaces $E^u_f$ permet
de mieux comprendre les tores possédant un champ de Killing. En
particulier, on déduit des résultats précédents que le revêtement universel
d'un tel tore se plonge toujours dans une surface simplement connexe et
L-complète. Le phénomène d'extension constaté dans \cite{BM} sur le tore de
Clifton-Pohl n'a donc rien d'exceptionnel. Bien qu'il découle directement
des théorèmes~\ref{theo:les_Eu_intro} et~\ref{theo:unif_intro}, 
le résultat suivant mérite un énoncé.

\begin{theo}[extension universelle d'un tore]
\label{theo:extension_univ_tore}
 Soit $T$ un tore lorentzien lisse  (resp. analytique) de dimension~2 muni d'un
champ de Killing non trivial. Alors le revêtement universel de~$T$ possède une
extension lisse (resp. analytique) munie d'un champ de Killing, 
simplement connexe, L-complète  et réflexive. Une telle extension est 
unique à isométrie près. 
\end{theo}

Dans le cas analytique, on a un résultat plus
fort (corollaire~\ref{coro:ext_unvi_analytique}) : {\em le revêtement universel
d'un tore lorentzien analytique possédant un champ de Killing non trivial admet,
à isométrie près, une unique extension analytique, simplement connexe et
L-complète}.

Quand $T$ n'est pas L-complet, l'extension du revêtement universel doit
évidemment être propre. Le revêtement universel est au contraire inextensible si
$T$ est L-complet, c'est-à-dire (voir Sanchez \cite{Sanchez}) si le champ de
Killing ne change pas de type, $T$ est alors complet. Bien que cette condition
sur la géométrie locale $[f]$ (qui ne change pas de signe) soit ouverte, ce  cas
est assez particulier. Le théorème~\ref{theo:extension_univ_tore}
établit donc que
la complétude est la seule obstruction à l'extensibilité du revêtement
universel, pour les tores admettant un champ de Killing. On pourrait se demander
plus généralement s'il est possible de caractériser les tores lorentziens dont
le revêtement universel est extensible.

Le théorème~\ref{theo:unif_intro} nous conduit à une classification des
tores lorentziens admettant un champ de Killing. Contrairement à son
analogue riemannien, le revêtement universel d'un tel tore n'est pas
déterminé par la norme du champ (sauf si celle-ci ne change pas de signe,
auquel cas le tore est complet), seule son extension universelle l'est. Ce
problème de classification est lié à la description des ouverts divisibles
des surfaces $E^u_f$.  Il apparaît que ces ouverts sont les relevés de
certaines géodésiques de l'espace des feuilles du champ de Killing de
$E^u_f$ et qu'ils peuvent être codés simplement par une donnée finie, mais
arbitrairement grande. Nous pouvons alors donner la classification des
tores et des bouteilles de Klein possédant un champ de Killing, voir
corollaires
\ref{coro:class_surf}, \ref{coro:class_bout1} et \ref{coro:class_bout2}
(comparer avec Sanchez~\cite{Sanchez},  voir aussi
 Matveev~\cite[théorème 4]{Matveev}).
Cette étude nous permet également de
déterminer, pour chaque géométrie locale périodique $[f]$, les composantes
connexes de l'espace des métriques lorentziennes
contenant un tore ou une bouteille de
Klein localement modelé sur l'espace $E^u_f$, voir
section~\ref{sub:compos}.

Finalement, nous retournons  rapidement à l'étude des points conjugués. Nous
montrons que les surfaces $E^u_f$ associées à un tore (c'est-à-dire avec $f$
périodique) vérifient un théorème \og à la Hopf\fg : elles
contiennent toujours des points conjugués à moins qu'elles ne soient
plates, 
proposition~\ref{prop:pts_conjugues_Eu}. 
Autrement dit, {\em si l'extension universelle associée à $(T,K)$ n'a pas
de points conjugués, alors $T$ est plat}.
Nous établissons aussi que les tores
non plats ayant un champ de Killing et aucun point conjugué doivent ressembler
en un sens assez précis à un tore de Clifton-Pohl (qui rappelons-le n'en possède
pas, \cite{BM}), théorème~\ref{theo:comme_CP}.  En particulier, un tel tore
n'est jamais homotope à une métrique plate, ce qui étend un résultat de
Gutierrez, Palomo et Romero \cite{GPR} concernant les métriques globalement
conformément plates.

\renewcommand\contentsname{Contenu de l'article}
\tableofcontents
\section{Géométrie associée à un champ de Killing}
\label{sect:geom_killing}
\subsection{Bandes. Extension par réflexion}
\label{subs:ext_refl}
Tous les objets (surfaces, métriques lorentziennes, champs de vecteurs)
sont supposés de classe $\mcC^\infty$.  Dans la suite, on considère des
couples $(X,K)$ où $X$ est une surface lorentzienne, {\em éventuellement à
  bord}, munie d'un champ de Killing {\em non trivial} $K$.  Une isométrie
entre deux couples $(X,K)$ et $(X',K')$ est une isométrie lorentzienne
$\Phi : X\to X'$ telle que $\Phi_*(K)=K'$.  On dira que $(X,K)$ est {\em
  saturée} si $K$ est complet ; les composantes du bord (s'il est non vide)
sont alors des orbites de $K$. On notera toujours $\mcK$ le feuilletage de
$K$ et $\mcK^\perp$ le feuilletage orthogonal à $K$, feuilletages définis
en dehors des zéros de $K$.

On rappelle qu'une surface simplement connexe, complète et à courbure
constante est isométrique au plan de Minkowski $\text{Mink}_2$,
c'est-à-dire $\R^2$ muni de la métrique $2dxdy$, ou à un multiple du
revêtement universel de $\text {dS}_2$, la surface de de Sitter,
c'est-à-dire le projectifié de $\{(x,y,z)\in \R^3\,| x^2+y^2-z^2=1\}$ muni
de la métrique induite par $dx^2+dy^2-dz^2$. On appelle \og demi-plan de
Minkowski\fg \ toute surface isométrique à un demi-plan ouvert de
$\text{Mink}_2$ dont le bord est une droite de type lumière, c'est-à-dire
toute surface isométrique à $\R^2$ muni de la métrique $2dxdy+xdy^2$. On
appelle \og domaine de de Sitter\fg\ toute surface isométrique au
complémentaire d'une géodésique de lumière dans $\text{dS}_2$.

\begin{lemm}
\label{lemm:isom_isomK}
Soit $X$ une surface lorentzienne connexe munie d'un champ de Killing $K$
non trivial et complet. Le groupe $\Isom^\pm(X,K)$ des isométries de $X$
qui préservent $K$ au signe près coïncide avec $\Isom(X)$, sauf si le
revêtement universel de $X$ est à courbure constante et contient un
demi-plan de Minkowski ou (quitte à multiplier la métrique par un scalaire)
un domaine de de Sitter.
\end{lemm}

\begin{proof}
Supposons que l'algèbre de Lie $\mfK_X$ des champs de Killing sur $X$ soit
de dimension~$1$. Il existe alors $\la\in \R$ tel que $\Phi_*(K)=\la K$. Si
$\langle K,K\rangle$ n'est pas identiquement nulle, on a $\la^2=1$. Sinon
$K$ est de type lumière et donc le revêtement universel de $X$ est
isométrique à un ouvert de $\text{Mink}_2$ dont le bord est constitué de
géodésiques de lumière invariantes par $\Phi$ et $K$. Il en contient donc
au plus une et $X$ contient un demi-plan de Minkowski.

Supposons maintenant que $\mfK_X$ est de dimension au moins 2 (en fait
égale \`a 2, le cas où $\dim \mfK_X>2$ est bien connu,
\cite[p. 372]{wolf}).  Soient $K_1$ et $K_2$ deux champs de Killing et soit
$p\in X$.  Si $K_1(p)$ et $K_2(p)$ sont non colinéaires, alors la courbure
est constante près de $p$.  Sinon, il existe $K_3\in \mfK_X$ non trivial
s'annulant en~$p$. Ce point est nécessairement une selle de $K_3$ (voir la
preuve de la proposition~\ref{prop:carte_exponentielle}): à nouveau la
différentielle de la courbure est nulle en~$p$.  Quitte à multiplier la
métrique par un scalaire, la courbure de $X$ est égale à $0$ ou~$1$ et
donc~$X$ est localement modelé sur $\text{Mink}_2$ ou $\text{dS}_2$. En
comparant les algèbres de Lie des champs de Killing locaux de $X$ et du
modèle, on voit que si $\Isom(X)$ contient un sous-groupe abélien de
dimension $2$, alors $X$ est plate et les champs $K_i$ sont géodésiques et
partout linéairement indépendants.  Par conséquent $X$ est isométrique au
plan de Minkowski (et finalement $\dim \mfK_X = 3$).  Sinon, il contient un
sous-groupe localement isomorphe au groupe affine de la droite, noté
$\text{Aff}$.  Les orbites de dimension $1$ de l'action d'un tel groupe sur
$\text{Mink}_2$ ou $\text{dS}_2$ sont des droites de lumière isolées. Par
conséquent, $X$ contient une orbite ouverte $U$ simplement connexe car son
bord (si $U\neq X$) est constitué de géodésiques de lumière proprement
plongées.  L'ouvert $U$ est donc isométrique au groupe $\text{Aff}$ muni
d'une métrique lorentzienne invariante à gauche.  Une telle surface est
isométrique à un demi-plan de Minkowski ou à un domaine de de Sitter.
\end{proof}

\begin{defi}[rubans, bandes, carrés  et dominos]
\label{defi:rubans_bandes}
Soit $X$ une surface lorentzienne munie d'un champ de Killing $K$ complet
{\em ne s'annulant pas}.  On dit que $(X,K)$ est
\begin{enumerate}
\item \label{defi:ruban}
un {\em ruban} si $X$ est simplement connexe et si l'un des
  feuilletages de lumière de $X$ est partout transverse à $K$ (y compris au
  bord s'il est non vide),
\item  \label{defi:bande_carre}
une {\em bande} (resp. un {\em carré}) si $X$ est homéomorphe à
  $[0,1]\times \R$ (resp.  à \mbox{$[0,1]^2\smallsetminus \{0,1\}^2$}) 
avec $\langle K,K\rangle$ nul au bord  et non nul à l'intérieur de $X$, 
\item  \label{defi:domino}
un {\em domino} si $X$ est simplement connexe 
{\em sans bord} et si son champ $K$ admet une unique orbite de lumière. 
\end{enumerate}
\end{defi}

Un ruban $(X,K)$ dans lequel on a choisi une orbite $c_0$ de $K$ (par exemple
une orbite de lumière s'il en existe) sera appelé {\em ruban marqué} et noté
$(X,K,c_0)$. Les dominos sont automatiquement des rubans, que l'on marquera
toujours par l'unique orbite de lumière de leur champ de Killing.

\begin{lemm}[coordonnées adaptées]
\label{lemm:carte_adaptee}
Soit $(X,K,c_0)$ un ruban lorentzien marqué. Il existe  un
intervalle~$I$ de $\R$ comprenant $0$ et des coordonnées globales $(x,y)\in
I\times \R$ sur $X$ dans lesquelles le champ et la métrique s'écrivent
respectivement~$\partial_y$ et
\begin{equation}
\label{equa:carte_adaptee}
2dxdy + f(x)dy^2 \esp (x,y)\in I\times \R,
\end{equation} 
la feuille $c_0$ étant donnée par $x=0$.  Si le champ $K$ admet une feuille
de lumière, alors ces coordonnées sont uniques à translation près de la
variable $y$ (en particulier~$I$ et $f$ sont uniques).
\end{lemm}

\begin{proof} 
Il existe un champ $L$ sur $X$ (qui est orientable) vérifiant $\langle L,L
\rangle = 0$, $\langle K,L \rangle = 1$ et invariant par $K$, c'est-à-dire
$[K,L]=0$.  Pour toute géodésique maximale $\ga$ tangente à $L$ % on note
$U_\ga$ le saturé de $\ga$ par le flot $\Phi^t$ de $K$. Comme $\ga$ est
maximale, chaque $U_\ga$ est ouvert. Les ouverts $U_\ga$ sont disjoints ou
confondus, $X$ est connexe, par conséquent $X=U_\ga$.  Puisque~$X$ est
simplement connexe, toute géodésique $\ga$ coupe chaque orbite de $K$ au plus
une fois d'après le théorème de Poincaré-Bendixson.  On prend comme coordonnée
$x$ le temps du flot de~$L$ normalisé par $x=0$ sur $c_0$ (ce paramétrage est
géodésique sur chaque orbite de~$L$) et pour $y$ le temps du flot de $K$
normalisé par $y=0$ sur une géodésique $\ga_0$. Si $K$ admet une feuille de
lumière, alors le champ $L$ est unique.  Le seul changement possible de
coordonnées est une translation, d'où l'unicité.
\end{proof}

\begin{rema}
\label{rema:demi_espaces_elem} 
Si $(X,K,c_0)$ est un domino, l'orbite de lumière $c_0$ partage $X$ en deux
demi-espaces ouverts, que l'on notera $X^\pm$, donnés en coordonnées adaptées
par $\pm x>0$.
\end{rema}

Localement, au voisinage d'une orbite non triviale $c_0$ de $K$, il existe
toujours au moins un champ $L$ comme dans la preuve ci-dessus; en
particulier $c_0$ est contenue dans un ruban ouvert.  Un tel champ $L$ est
unique au voisinage des orbites de lumière de $K$, par exemple au voisinage
du bord dans le cas d'une bande.

\begin{prop}[réflexions locales génériques]
\label{prop:refl_loc_generique}
Soit $(U,K)$ une surface lorentzienne saturée homéomorphe au plan et 
telle que $\langle K,K \rangle$ ne s'annule pas. Pour toute feuille
$\gamma$ du feuilletage orthogonal à $K$, il existe une 
isométrie indirecte $\sig_\gamma :U\to U$ fixant $\gamma$ point par point.
Cette isométrie $\sig_\gamma$ sera appelée {\em réflexion locale générique 
d'axe $\gamma$}.
\end{prop}

\begin{proof} 
D'après le lemme~\ref{lemm:carte_adaptee}, la surface $(U,K)$ admet
une carte adaptée de la forme~\eqref{equa:carte_adaptee}.  Si $G$ est
une primitive de $-1/f$ sur l'intervalle (ouvert) $I$, les feuilles du
feuilletage orthogonal à $K$ s'écrivent $y=G(x)+\beta$ ($\beta \in \R$). La
réflexion définie par
\begin{equation}
\label{equa:refl_loc_generique}
x'=x,~~~ y'=2(G(x)+\beta)-y \esp ((x,y)\in I\times \R)
\end{equation}
est clairement une isométrie de $U$, qui inverse le champ $K$.
\end{proof}

\par 
Les rubans jouent un rôle central dans notre étude. Nous donnons
maintenant une classification de ces objets et de leurs isométries.
Certaines d'entre elles proviennent de symétries additionnelles de la
fonction norme du champ $K$. Soit $(X,K)$ un ruban et soit $f\in
\mcC^\infty(I,\R)$ comme dans \eqref{equa:carte_adaptee}. On pose
\begin{equation}
\label{equa:isom_f}
\Isom(f) = \{\ph\in \Isom(I,dx^2) ; f\circ \ph = f\},
\end{equation}
 où $\Isom(I,dx^2)$ désigne le groupe des isométries euclidiennes de
l'intervalle $I$. Le groupe $\Isom(f)$ est trivial sauf si $f$ admet une
symétrie, une période ou les deux ; quand $f$ est non constante, il est
alors isomorphe respectivement à $\Z/2\Z$, $\Z$ ou $D_\infty$ (groupe
diédral infini). Par ailleurs, on note $\Isom_{\rm gen}(X)$ le groupe formé
par le flot de $K$ et, si $\langle K,K\rangle$ ne s'annule pas, par les
réflexions génériques (voir proposition~\ref{prop:refl_loc_generique}).
Les éléments de $\Isom_{\rm gen}(X)$ seront appelés {\em isométries
  génériques} du ruban $(X,K)$.

\begin{prop}[classification et isométries des rubans]
\label{prop:class_isom_rubans}
\hfill
\begin{enumerate}
\item
Les classes d'isométrie de rubans $(X,K)$ tels que $\langle K,K\rangle$
s'annule (resp. ne s'annule pas) correspondent bijectivement aux fonctions
lisses modulo translation (resp. translation et changement de signe) de la
variable.
\item Si $(X,K)$ est un ruban à courbure non constante, alors le groupe
d'isométrie  $\Isom(X)$ est isomorphe à un produit semi-direct 
$\Isom_{\rm  gen}(X) \rtimes \Isom(f)$ du groupe des isométries génériques 
par le groupe de symétrie d'une fonction  $f$ associée à $(X,K)$; de plus
toute géodésique de lumière  interne (maximale) définit une section
$\Isom(f)\to \Isom(X)$.
\end{enumerate}
\end{prop}

\begin{proof} 
La métrique et le champ d'un ruban $(X,K)$ sont déterminés par une carte
adaptée de la forme \eqref{equa:carte_adaptee}.  Dans cette carte, la
fonction $f$ représente la norme $\langle K,K\rangle$ exprimée dans la
coordonnée transverse $x$. Si $\langle K,K\rangle$ s'annule, le champ $L$
est unique et $f$ est bien définie modulo translation de~$x$.  Sinon, il
existe deux champs transverses $L$ et $L'$ et deux cartes adaptées de la
forme $2dx dy + f(x)dy^2$ et $2dx' dy' + f(-x'){dy'}^2$. Le changement de
carte s'obtient en composant $x'=-x, y'=-y$ (qui inverse le champ de
Killing) avec une réflexion de l'une des cartes,
voir~\eqref{equa:refl_loc_generique}.
\par

Toute isométrie $\Phi$ d'un ruban $(X,K)$ qui n'est pas à courbure constante
envoie  $K$ sur $\pm K$ (lemme~\ref{lemm:isom_isomK}).  Posons $(u,v) =
\Phi(x,y)$ où $(x,y) \in I\times\R$ sont des coordonnées adaptées. Si $f$
s'annule, alors par unicité du champ $L$ le couple de champs $(K,L)$ est envoyé
par $\Phi$ sur $(K,L)$ (cas~(a)) ou sur $(-K,-L)$ (cas~(b)).  Dans le cas~(a) on
a $u=x+ a,v=y+t$ ; de plus $a=0$ sauf si $f$ est périodique (et $I=\R$).  Le
cas~(b) n'est possible que si $f(b-x)=f(x)$ pour tout $x\in I$, ce qui impose
une symétrie à la fonction $f$.  Si $f$ ne s'annule pas, il existe deux champs
$L=\partial_x$ et $L'=-\partial_x+ 2/f(x) \partial_y$ vérifiant $\langle
K,L\rangle = \langle K,L'\rangle = 1$.  Aux possibilités (a) et (b) s'ajoutent
les suivantes : $(K,L')$ (cas~(a')) et $(-K,-L')$ (cas~(b')).  D'après la
proposition~\ref{prop:refl_loc_generique}, il existe une isométrie globale
$\sig$ (réflexion générique) vérifiant (b') ; si $\Phi$ satisfait (a') ou (b'),
alors $\sig\circ \Phi$ satisfait (b) et (a) respectivement. On est ramené aux
cas précédents.  Ensuite, le groupe $\Isom_{\rm gen}(X)$ est par définition
distingué et l'assertion (2) se déduit facilement de la description des
isométries. Noter que le stabilisateur d'une géodésique de  lumière interne
(maximale) est isomorphe à $\Isom(f)$ et ne contient aucune isométrie générique
non triviale. 
\end{proof}

\begin{prop}[feuilletage de lumière transverse]
\label{prop:lumiere_transv}
Soit $(X,K)$ une surface lorentzienne connexe, sans bord et 
munie d'un champ complet
 $K$ ne s'annulant pas. On suppose de plus que~$X$ admet un feuilletage
de lumière $\mcL$ partout transverse à $K$. Dans ces conditions 
\begin{enumerate}
\item  le flot de $K$ est transitif sur les feuilles de $\mcL$,
\item  $X$ est homéomorphe au plan, au cylindre ou au tore ;
métriquement, $(X,K)$ est un ruban ou le quotient d'un ruban.  
\end{enumerate}
\end{prop}

\begin{proof} La propriété~(1) de transitivité s'obtient
comme dans la preuve du lemme~\ref{lemm:carte_adaptee}. 
Notons $(\widetilde{X},\widetilde{K})$ le revêtement universel 
de $(X,K)$, qui est
un ruban sans bord.  Tout élément $\ph \in \aut_X \widetilde{X}$ 
doit être
isométrique et respecter les champs $\widetilde{K}$ et $\widetilde{L}$ 
(en particulier $\ph$ préserve les feuilletages de lumière).
  En coordonnées adaptées, on trouve que $\ph$ est un élément
du flot composé éventuellement avec une translation de la variable $x$ (ce
qui impose $I=\R$ et $f$ périodique dans \eqref{equa:carte_adaptee}). Le
groupe $\aut_X\widetilde{X}$ est donc abélien et composé d'isométries
directes, d'où l'assertion~(2).
\end{proof}

\begin{lemm}[types de bandes]
\label{lemm:types_bandes}
Toute bande lorentzienne $(B,K)$ est de l'un des trois types suivants.
\begin{enumerate} 
\item[\textbullet] Type I (ou {\em bande standard}) : $(B,K)$ est un
  ruban. Les feuilletages $\mcK$ et $\mcK^\perp$ sont des  suspensions.
Le champ $L$ est global, rentrant sur un bord et sortant sur l'autre.
\item[\textbullet] Type II : $(B,K)$ n'est pas un ruban, le feuilletage
  $\mcK$ est une suspension et $\mcK^\perp$ est une composante de Reeb.  Le
  champ (local) $L$ est rentrant ou sortant au bord.
\item[\textbullet] Type III (ou type {\em Reeb}) : $(B,K)$ n'est pas un
  ruban, le feuilletage $\mcK$ est une composante de Reeb et $\mcK^\perp$
  est une suspension.  Le champ $L$ (local) est rentrant sur un bord et
  sortant sur l'autre.
\end{enumerate} 
\end{lemm}

\begin{proof} Les feuilletages $\mcK$ et $\mcK^\perp$ coïncident au bord et sont
transverses à l'intérieur. Le champ local $L$ (près du bord) est partout
transverse à $\mcK$ et $\mcK^\perp$. Par définition il n'existe aucune feuille
de lumière de $K$ à l'intérieur de $B$. En conséquence, si l'un des feuilletages
$\mcK$ ou $\mcK^\perp$ n'est pas une suspension (n'admet pas de transversale
globale), alors il se réduit à une composante de Reeb. Considérons le fibré des
directions au-dessus d'un arc $\tau$ qui joint les deux composantes du bord. Il
s'agit d'un cylindre $[0,1]\times \SSS^1$ séparé par les champs de lumière en
deux ouverts (directions positives ou négatives de la métrique au-dessus de
$\tau$). Deux cas se présentent. Si le champ $L$ est global, on voit que
nécessairement les feuilletages $\mcK$ et $\mcK^\perp$ sont des suspensions
(type~I). Sinon, l'un des feuilletages $\mcK$ ou $\mcK^\perp$ est une composante
de Reeb. On peut alors choisir $\tau$ transverse aux deux feuilletages de
lumière et donc à l'autre feuilletage qui est forcément une suspension : la
bande est de type II ou III. Les trois types de bandes sont illustrés plus bas
(voir la figure~\ref{figu:carre_standard} et les commentaires attenants). 
\end{proof}

\begin{defi}
\label{defi:extension}
Soit $(X,K)$ une surface lorentzienne connexe, séparée et  
saturée par un champ de Killing~$K$.
\begin{enumerate}
\item 
Une surface saturée $(\widehat{X},\widehat{K})$ est une {\em extension} de
$(X,K)$ si $\widehat{X}$ est  {\em connexe et séparée} et s'il existe un
plongement isométrique $\varphi : X \to \widehat{X}$ tel que $\widehat{K}$
coïncide avec $\varphi_*(K)$ sur $\varphi(X)$.
\item 
$(X,K)$ est {\em maximale} (ou {\em inextensible}) si elle n'admet pas
d'extension non  triviale. \end{enumerate}
\end{defi}

\begin{exem}
\label{exem:ext_sf_a_bord}
Toute surface saturée $(X,K)$ {\em à bord} est extensible. En effet, soit $c$
une composante de $\partial X$. Quitte à remplacer $X$ par un revêtement
cyclique, on peut supposer que~$c$ est une orbite plongée de $K$. L'existence de
coordonnées adaptées \eqref{equa:carte_adaptee} dans un voisinage saturé de $c$
montre alors que $(X,K)$ s'étend au voisinage de $c$ (analytiquement si $(X,K)$
est analytique).
\end{exem} 

\begin{exem}
\label{exem:L-complete}
Toute surface $(X,K)$ dont les géodésiques de lumière sont complètes (on dira
L-complète) est inextensible, voir la fin de la preuve de la
proposition~\ref{prop:sf_Euf}.
\end{exem}

L'hypothèse de séparation entraîne que les surfaces compactes sont inextensibles
au sens précédent. Il existe néanmoins des plongements isométriques de surfaces
compactes ou complètes dans des surfaces connexes non séparées. Dès qu'une
surface  connexe $(X,K)$ admet un ouvert saturé $U$ non vide et propre (c'est le
cas par exemple d'un tore plat muni d'un champ périodique), le recollement de
deux copies de $(X,K)$ le long de $U$ est une \og extension\fg\ connexe non
séparée de $(X,K)$. On trouvera des exemples plus naturels dans la discussion
qui suit la proposition~\ref{prop:tores_loc_model}.

\begin{rema}
\label{rema:max_champ=max_usuel}
Nous montrerons plus loin (proposition~\ref{prop:max_champ=max_usuel}) que la
notion de maximalité introduite à la définition~\ref{defi:extension} (surfaces
munies d'un champ de Killing) équivaut à la maximalité au sens usuel.
\end{rema}

Nous décrivons maintenant un ingrédient essentiel du processus d'extension,
basé sur l'existence de réflexions locales
(voir proposition~\ref{prop:refl_loc_generique}). 

\begin{lemm}[lemme d'extension]
\label{lemm:extension}
Soit $(X_1,K_1)$ et $(X_2,K_2)$ deux surfaces lorentziennes connexes, séparées
et saturées (avec ou sans bord).  Soit $U_1$ (resp. $U_2$) un ouvert saturé
strict de $X_1$ (resp. de $X_2$) et soit $\sigma : (U_1,K_1)\to (U_2,K_2)$ une
isométrie. On suppose que pour tout $p\in \partial U_1$ la fonction $\sigma(q)$
n'admet aucune valeur d'adhérence quand $q$ tend vers~$p$ ($x\in U_1$).  Alors
les surfaces $(X_1,K_1)$ et $(X_2,K_2)$ sont extensibles.
\end{lemm}
 
\begin{proof}
Soit $X$ la surface définie en recollant $(X_1,K_1)$ et $(X_2,K_2)$ {\it via}
l'isométrie $\sig$.  Plus précisément $X$ est le quotient de $X_1\coprod
X_2$ par l'identification de $q\in U_1\subset X_1$ avec $\sigma(q)\in
U_2\subset X_2$.  Il s'agit d'une surface lorentzienne connexe, munie d'un
champ de Killing $K$ et possédant un ouvert strict isométrique à
$(U_1,K_1)$. Il reste à vérifier que $X$ est séparée. Les éventuels points
de branchement se situent nécessairement au bord de $U_1$. Si $ p_1 \in
\partial U_1\subset X_1$, les points $p_2\in \partial U_2\subset X_2$ non
séparés de~$p_1$ sont par construction les valeurs d'adhérences de $\sig(q)$
quand $q$ tend vers~$p_1$, d'où le résultat.
\end{proof}

\begin{coro}[extension par réflexion]
\label{coro:extension_reflexion}
Toute bande lorentzienne $(B,K)$ de type~$I$ ou $II$ (le champ $K$ est une
suspension) est extensible, de façon unique à isométrie près, par un carré
lorentzien $(\widehat{B},\widehat{K})$
(définition~\ref{defi:rubans_bandes}-(\ref{defi:bande_carre})) avec $B$ dense
dans $\widehat{B}$.
\end{coro}

\begin{proof}
Soit $\sigma$ une réflexion générique définie sur l'intérieur $U$ de $B$
(proposition~\ref{prop:refl_loc_generique}). Notons $L_1,L_2$ les deux champs de
lumière sur $U$ normalisés par $\langle L_i,K\rangle = 1$ ($i=1,2$). La
réflexion $\sigma$ échange $L_1$ et $-L_2$. En un point $p\in \partial B$, l'un
des champs $L_i$ est défini, disons $L_1$, supposé sortant, dont les orbites
locales (au voisinage de~$p$) s'accumulent au bord.  Pour les types I et II, les
orbites de $-L_2$ ne s'accumulent pas au bord (voir
lemme~\ref{lemm:types_bandes}). Par suite $\sigma(q)$ ne peut avoir de valeurs
d'adhérences quand $q$ tend vers~$p$ (mais ce serait possible pour le type III).
L'extension \smash{$(\widehat{B},\widehat{K})$} construite ci-dessus est
clairement homéomorphe $[0,1]^2 \smallsetminus \{0,1\}^2$. De plus $B\simeq
[0,1]\times ]0,1[$ est dense dans \smash{$\widehat{B}$}. L'unicité d'une telle
extension résulte de la classification des carrés et des bandes
(propositions~\ref{prop:class_carres} et~\ref{prop:class_bandes}).
\end{proof}

L'espace des feuilles du champ $K$ d'un carré est une variété riemannienne,
proposition~\ref{prop:struct_transv_killing}. Si $J\subset \R$ est un
intervalle, nous noterons $J^\wedge$ la variété obtenue en identifiant ({\it
via} l'identité) les intérieurs de deux exemplaires de $J$ (intervalle avec bord
double).

\begin{prop}[description et classification des carrés]
\label{prop:class_carres} Si $(X,K)$ est un carré lorentzien, alors
\begin{enumerate}
\item $(X,K)$ est recouvert par deux bandes de type I, échangées par des
isométries indirectes de $X$ (réflexions) et 
dont l'intersection coïncide avec l'intérieur de $X$, 
\item il existe un réel $\msfm>0$ (appelé \og {\em largeur}\fg\ du carré)
et une fonction $f\in \mcC^\infty([0,\msfm],\R)$ telle que
$f^{-1}(0)=\{0,\msfm\}$  qui caractérisent le couple $(X,K)$ à isométrie près.
\end{enumerate}
De plus, l'espace des feuilles de $K$ est isométrique au segment à bord
double $[0,\msfm]^\wedge$.
\end{prop}

\begin{proof}
Dans l'intérieur $U$ de $X$, les deux feuilletages de lumière $\mcL$ et $\mcL'$
sont transverses à $K$. Soit $p\in \partial X$ et soit $\ell$ la feuille de
lumière maximale issue de~$p$ transversalement à $\partial X$. D'après la
proposition~\ref{prop:lumiere_transv}, le saturé de $\ell\cap U$ par le flot de
$K$ coïncide avec $U$.  Tout point intérieur $q\in U$ peut donc être joint à
chacune des composantes du bord par une (unique) géodésique de lumière. Comme
$X$ est simplement connexe, les feuilles de $\mcL$ et $\mcL'$ se coupent au plus
une fois. Cela entraîne que chaque feuille passant par $q \in U$ doit joindre
deux côtés opposés de $X$. Ainsi $X$ est réunion de deux bandes $B$ et $B'$ de
types I telles que $B\cap B'=U$.  L'espace des feuilles de $K$ est donc
isométrique à $[0,\msfm]^\wedge$ pour un certain réel $\msfm>0$.  Notons $L$ et
$L'$ deux champs comme dans la preuve du lemme~\ref{lemm:carte_adaptee},
définissant des coordonnées adaptées $(x,y), (x',y') \in [0,\msfm]\times\R$ sur
$B$ et $B'$ respectivement.  Soit $\sig :U\to U$ une réflexion par rapport à une
géodésique orthogonale à $K$. On doit avoir $\sig_*(K)=-K$ et $\sig_*(L)=-L'$.
En coordonnées adaptées, $\sig$ est donc donnée par $x'=\msfm -x$, $y'=\beta -y$
($\beta \in \R$).  Cette transformation s'étend aux bords des bandes ($x,x'\in
\{0,a\}$).  Par suite $\sig$ s'étend en une isométrie de~$X$ qui échange $B$ et
$B'$.

\par Étant donné un carré $(X,K)$, on note $\msfm>0$ la masse des feuilles
de~$K$. La norme de~$K$ (vue comme fonction sur $[0,\msfm]^\wedge$) induit une
fonction $f=f_{X,K}\in \mcC^\infty([0,\msfm],\R)$ comme dans l'énoncé,
invariante sur la classe d'isométrie de $(X,K)$. Enfin, si
$f_{X',K'}=f_{X,K}$, on construit (en utilisant comme ci-dessus des
coordonnées adaptées) une isométrie entre les carrés $(X,K)$ et $(X,K')$.
\end{proof}

\begin{rema}[feuille du milieu d'un carré]\label{rema:feuille_milieu}
Si $(X,K)$ est un carré de largeur $\msfm$,  le point
$\msfm/2 \in [0,\msfm]^\wedge$ sera appelé {\em feuille du milieu}. 
Il s'agit d'une feuille privilégiée du champ~$K$, invariante par isométrie.
\end{rema}

\begin{prop}[classification des bandes]
\label{prop:class_bandes} Les classes d'isométries de bandes lorentziennes 
$(X,K)$ sont caractérisées par 
\begin{enumerate}
\item le type (I, II ou~III),
\item un élément $\msfm \in ]0,+\infty] $  (largeur de la bande),
$\msfm < +\infty$ pour les types I et II,
\item une fonction $f\in \mcC^\infty(J,\R)$ telle que $f^{-1}(0)=\partial
  J$, avec $J=[0,\msfm]$ pour le type I, $J=[0,\msfm]$ ou
 $[-\msfm,0]$ pour le type II   et $J=[0,\msfm[$ pour le type~III.
\end{enumerate}
L'espace des feuilles de $K$ est isométrique à $J$  pour les
types I et II et à $J^\wedge$ pour le type~III.
\end{prop}

\begin{proof} 
On procède comme pour les carrés, en recouvrant la bande par des cartes adaptées
(une carte pour le type I, deux cartes pour les types II et~III). La largeur
$\msfm$ est finie pour les types I et II d'après le
lemme~\ref{coro:extension_reflexion}. Pour les bandes de type II, $J=[0,\msfm]$
(resp. $J=[-\msfm,0]$) si le champ $L$ est rentrant (resp. sortant).
\end{proof}

\begin{rema}
\label{rema:bande_Reeb_sym}
Toute bande de type~III (Reeb) admet des isométries indirectes qui inversent le
champ et  échangent les deux composantes de bord (on étend au bord les
réflexions de l'intérieur comme dans la preuve de la
proposition~\ref{prop:class_carres}).
\end{rema}

L'interaction entre les bandes et les carrés est illustrée à la
figure~\ref{figu:carre_standard}.  Les réflexions par par rapport aux feuilles
du feuilletage orthogonal $\mcK^\perp$ (axes en pointillés sur la
figure~\ref{figu:carre_standard}) se prolongent au carré. Du point de vue du
champ, depuis l'intérieur du carré, les sommets (l'infini du carré) apparaissent
de deux types : \og selle\fg\ ou source/puits.  Un carré de largeur~$\msfm$
contient~2 bandes de type I (côtés opposés) et deux bandes de type II (côtés
adjacents à une source/puits) toutes de largeur $\msfm$, ainsi que deux familles
de bandes de type~III (côtés adjacents à une selle) de largeur $\msfm'\leq
\msfm$.  Les 2 bandes de types I (resp. de type II) sont échangées par les
réflexions, tandis que les bandes de type~III sont fixes d'après la
remarque~\ref{rema:bande_Reeb_sym}.  À l'intérieur du carré, le 4-tissu formé
par $\mcK$, $\mcK^\perp$ et les deux feuilletages de lumière est invariant par
les réflexions.

\begin{figure}[h!]
\labellist
\small\hair 2pt
\pinlabel $L$ at 280 130
\pinlabel $K$ at -5 165
\endlabellist
\begin{center}
\includegraphics[scale=0.4]{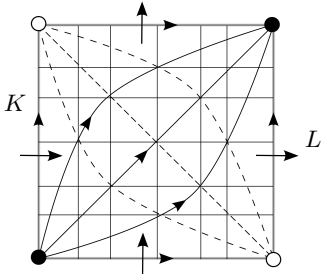}
\caption{Carré lorentzien}
\label{figu:carre_standard}
\end{center}
\end{figure}

\begin{rema}
\label{rema:ouv_sat_carre}
Soit $C$ un carré de largeur $\msfm$ et soit $V$ un ouvert connexe saturé de $C$
qui rencontre le bord $\partial C$. En fonction du nombre $n$ de composantes de
$V\cap \partial C$, l'ouvert $V$ est de la forme suivante : {\em demi-bande} de
largeur $\msfm'\in ]0,\msfm]$ ($n=1$), bande ($n=2$, types I et II de largeur
$\msfm$, ou~III de largeur $\msfm'\in ]0,\msfm]$), {\em quasi-carré} ($n=3$) ou
carré ($n=4$).
\end{rema}

\begin{prop}
\label{prop:class_composantes}
Soit $(X,K)$ une surface lorentzienne lisse homéomorphe au plan $\R^2$ et munie
d'un champ de Killing $K$ non trivial et complet. On note $X_0$ la réunion des
orbites de lumière du champ~$K$ et $\Sigma$ l'ensemble de ses zéros. Soit $U$
une composante connexe de $X\smallsetminus (X_0\cup \Sigma)$.
\begin{enumerate}
\item L'adhérence $U^\vee$ de $U$ dans $X \smallsetminus \Sigma$ est de
  l'un des types suivants (voir remarque~\ref{rema:ouv_sat_carre}) :
  carré, quasi-carré, bande (type I, II ou III), demi-bande, $\emptyset$ ou
  $X$.
\item Si de plus $(X,K)$ est maximale (par exemple L-complète), alors
  $U^\vee$ est de l'un des types suivants : \subitem ~~(i) $U^\vee$ est un
  carré lorentzien, \subitem ~(ii) $U^\vee$ est une bande de type Reeb (de
  largeur infinie si $X$ est L-complète), \subitem (iii) $U^\vee =
  \emptyset$ ou $X$.
\end{enumerate}
\end{prop}

\begin{proof} 
Comme $X$ est homéomorphe au plan, le cône isotrope définit deux feuilletages de
lumière. Toute feuille de lumière est une droite proprement plongée (voir par
exemple \cite{HR1957}) ; par suite les orbites de lumière de $K$ sont des
droites plongées.  Soit $p \in \Sigma$ un zéro de $K$ et soit $\Phi_K^t$ le flot
de $K$.  La différentielle $D_p\Phi_K^t$ définit un groupe à un paramètre
d'isométries du plan tangent en~$p$, groupe conjugué à $\Phi_K^t$ près de~$p$
(par l'exponentielle en~$p$). Par suite, ou bien $K$ est nul près de~$p$, ou
bien~$p$ est un point selle de $K$ adhérent à 4 orbites de lumière de~$K$. Par
connexité de $X$ (vu que $K\neq 0$), l'ensemble des zéros $\Sigma$ est
nécessairement discret ; de plus $X_0 \neq \emptyset$ dès que $\Sigma \neq
\emptyset$.  

\par On peut écarter les cas évidents $X_0 = \emptyset$ ($U^\vee = X$) et $X_0 =
X$ ($U^\vee = \emptyset$).  Soit $\ell$ une feuille de lumière maximale qui
coupe transversalement le bord $\partial U$ de $U$ dans $X\smallsetminus
\Sigma$.  On~a $U = \cup_{t\in \R} \Phi_K^t(\ell \cap U)$ d'après la
proposition~\ref{prop:lumiere_transv}. De plus le  champ $K$ n'a pas de feuilles
fermées et $\ell$ est une droite proprement plongée, donc $U$ est homéomorphe au
plan, et pour toute composante $c$ de $\partial U$, il existe une carte adaptée
\eqref{equa:carte_adaptee} qui recouvre $U\cup c$ ; en particulier tout point
$q\in U$ peut être relié à $c$ par une unique géodésique de lumière $\gamma :
[0,\alpha] \to U^\vee$ telle que $\gamma(0)=q$ et $\gamma(\alpha) \in c$. Le
nombre $n\geq 1 $ de composantes de $\partial U$ ne peut donc excéder le nombre
de rayons de lumière issus d'un point de $U$, c'est-à-dire~$4$.  En examinant
les valeurs possibles de $n$, on conclut grâce aux cartes adaptées que $U^\vee$
est du type souhaité (voir remarque~\ref{rema:ouv_sat_carre}). La maximalité de
$(X,K)$ exclut les composantes vérifiant l'hypothèse du lemme
d'extension~\ref{lemm:extension}, c'est-à-dire les quasi-carrés, les bandes de
type I ou II et les demi-bandes.
\end{proof}

\subsection{Structures géométriques transverses}
\label{subs:struc_transverses}
Nous introduisons ici des structures géométriques transverses pour les
feuilletages précédents : feuilletage $\mcK$ associé au champ de Killing $K$ et
son feuilletage orthogonal $\mcK^\perp$ définis  sur l'ouvert $\{K\neq
0\}$, feuilletages de lumière. Ces structures, adaptées à diverses situations
géométriques sur la surface, ont pour but de mesurer d'une part la \og masse
\fg\ des feuilles de~$\mcK$ et d'autre part le décalage du flot de $K$.
Notamment, la structure projective transverse à $\mcK^\perp$ (proposition
\ref{prop:struct_proj_orth}) permet d'estimer le décalage du flot quand on
tourne autour d'un point selle.

\begin{prop}[structure riemannienne transverse au champ de Killing]
\label{prop:struct_transv_killing}
Soit $(X,K)$ une surface lorentzienne (avec ou sans bord) saturée et connexe,
avec $K$ non trivial.
\begin{enumerate}
\item Le feuilletage $\mcK$ associé au champ de Killing admet une métrique
riemannienne transverse invariante par l'action de $\Isom^\pm(X,K)$.
\item \label{prop:esp_feuilles_etale}
Si $X$ est simplement connexe, alors l'espace $\mcE_{(X,K)}$ des feuilles de
$\mcK$ est une variété riemannienne connexe de dimension~1, avec ou sans bord,
généralement non séparée. De plus, il existe une application lisse $x\in
\mcC^\infty(X,\R)$ qui induit une isométrie locale de $\mcE_{(X,K)}$ dans $\R$ ;
en particulier $\mcE_{(X,K)}$ est étalée au-dessus d'un intervalle de $\R$.
\end{enumerate}
\end{prop}

\begin{proof}
On choisit localement une forme volume lorentzienne $\nu$. Alors la forme
$i_K\nu$ est fermée (formule de Cartan) et définit localement le feuilletage
$\mcK$. L'intégration de $i_K \nu$ sur les transversales à $\mcK$ définit une
structure transverse modelée sur $(\R,\Isom(\R,dx^2))$ ; si $X$ est orientable,
il s'agit d'une structure de translation puisque $\nu$ est globale. La
structure riemannienne transverse à $\mcK$ est invariante par $\Isom^\pm(X,K)$
qui préserve $i_K \nu$ au signe près.
\par

Si $X$ est simplement connexe, la forme $i_K\nu$ est définie sur $X$ et exacte.
Soit $x\in \mcC^\infty(X,\R)$ une primitive de $i_K\nu$ et soit $\Sig$
l'ensemble des zéros de $K$ ; rappelons que $\Sig$ est discret, preuve de la
proposition~\ref{prop:class_composantes}. L'application
$x|_{X\smallsetminus\Sigma}$ est alors une submersion et ses sections locales
définissent une structure de variété sur l'espace des feuilles $\mcE_{(X,K)}$,
et même une structure de translation. L'application de $\mcE_{(X,K)}$ dans $\R$
induite par $x$ est évidemment une isométrie locale puisque la métrique
riemannienne de $\mcE_{(X,K)}$ est donnée par $(i_K\nu)^2=dx^2$.
\end{proof}

\begin{rema}
 \label{rema:esp_feuilles_etale}
 L'assertion~\eqref{prop:esp_feuilles_etale} est valable plus généralement si
$X$ est orientable -- et connexe -- avec $i_K \nu$ exacte. L'application $x\in
\mcC^\infty(X,\R)$ est alors bien définie modulo translation (choix de la
primitive) et changement de signe (choix de $\nu$) ; nous l’appellerons \og
coordonnée transverse\fg\ au champ. À translation et signe près, elle coïncide
avec la coordonnée~$x$ de toute carte adaptée~\eqref{equa:carte_adaptee}.
\end{rema}

\begin{prop}[structures riemanniennes transverses aux feuilletages de lumière]
\label{prop:prop:struct_transv_lumiere}
 Soit $(X,K)$ une surface lorentzienne saturée (avec ou sans bord) dont le champ
ne s'annule pas. Alors tout feuilletage de lumière $\mcL$ transverse à $K$ (s'il
en existe) admet une métrique riemannienne transverse invariante par le groupe
$\Isom^\pm(X,K)$. Si $\Phi^t$ désigne le flot de $K$, cette métrique transverse
est donnée par $dt^2$.
\end{prop}

\begin{proof} 
Localement, il existe un champ~$L$ tel que $\langle L,L \rangle = 0$ et $\langle
L,K \rangle = 1$, et une forme volume lorentzienne $\nu$, normalisée par
$\nu(K,L) = 1$. On voit que $L$ est invariant par~$K$, c'est-à-dire vérifie
$[K,L]=0$, puis que la forme $i_L\nu$ est fermée.   L'intégration de $i_L \nu$
sur les transversales à $\mcL$ définit une structure riemannienne transverse.
Tout élément de $\Isom^\pm(X,K)$ préserve $L$ et $\nu$ au signe près, donc
préserve la structure riemannienne transverse à $\mcL$. Dans une carte adaptée
de la forme~\eqref{equa:carte_adaptee}, on a $\nu= -dx\wedge dy$ et $i_L\nu = -
dy$, d'où la dernière assertion.
\end{proof}

\begin{prop}[structure affine transverse sur le feuilletage orthogonal]
\label{prop:struct_affine_orth}
Soit $(U,K)$ un domino lorentzien (voir définition~\ref{defi:rubans_bandes}). On
suppose que l'unique feuille de lumière de~$K$ est géodésiquement incomplète.
Alors le feuilletage orthogonal à~$K$ admet une structure affine  transverse
invariante par l'action du flot de~$K$. Cette structure est unique.
\end{prop}

\begin{proof} Le point important est l'invariance. Soit $L$ l'unique champ de
vecteurs sur $U$ vérifiant $\langle K,L \rangle = 1$ et $\langle L,L
\rangle = 0$ (dans une carte adaptée $L=\partial_x$) et soit $\omega$ la
1-forme sur $U$ donnée par $\omega=\langle K,\cdot \rangle$.  L'unique
feuille de lumière de $K$, que l'on notera $c_0$, n'étant pas
géodésiquement complète, on a $\lambda = L\cdot \langle K,K \rangle \neq
0$ sur $c_0$. Cela provient de l'existence de coordonnées adaptées et
du lemme suivant, dont la preuve  immédiate est laissée au lecteur.
\begin{lemm} \label{lemm:semi_complet} 
 Soit $I$ un intervalle ouvert comprenant $0$ et soit $f:I\rightarrow \R$
 une fonction lisse s'annulant en $0$. Considérons la métrique $2dxdy
 + f(x)dy^2$ définie pour $(x,y)\in I\times \R$. La droite $\{x=0\}$ est
 géodésiquement complète si et seulement si $f'(0)=0$. Si $f'(0)>
 0$ alors cette droite porte une géodésique semi-complète, le bout complet
 étant du côté $y<0$.
\end{lemm}

Comme $U$ est simplement connexe, toute structure affine transverse sur le
feuilletage orthogonal $\mcK^\perp$ est donnée par une submersion $\xi : U\to
\R$ (que  l'on prend comme paramètre transverse) vérifiant $d\xi = F \omega$
pour une certaine fonction lisse $F : U\to \R$ ne s'annulant pas. On peut
toujours supposer que $\xi=0$ correspond à $c_0$. L'invariance par le flot de
$K$ implique alors l'existence d'un réel $\mu$ tel que $K\cdot \xi = \mu \xi$.
Par suite, vu que $\mcL_K \omega = 0$, on a $K\cdot F = \mu F$ et la condition
$d(F \omega)=0$ s'écrit, en évaluant sur le couple $(K,L)$ :
$$
\langle K,K \rangle L\cdot F  + (L\cdot \langle K,K \rangle -\mu) F =0.
$$ 
Puisque $F$ ne s'annule pas, on a nécessairement $\mu=\la$. En coordonnées
$(x,y)\in I\times \R$ adaptées, sachant que $K=\partial_y$ et $L=\partial_x$, on
voit que la fonction $F$ doit être de la forme 
\begin{equation}
\label{equa:struct_affine_adaptee}
F = \ph(x) \exp(\la y) \esp (x,y)\in I\times \R,
\end{equation}
 où $\ph$ est une solution non nulle sur $I$ de l'équation différentielle
\begin{equation}
\label{equa:ED_struct_affine}
\ph'(x) + \frac{1}{f(x)}(f'(x)-\la)\ph(x)  = 0 \esp (x\in I).
\end{equation}
La linéarité de cette équation assure l'existence de $F$ et la
submersion $\xi$, normalisée par $\xi=0$ sur $c_0$, est finalement donnée
par
\begin{equation}
\label{equa:parametre_affine}
\xi=\frac{F}{\la}  \langle K,K \rangle. 
\end{equation}
Elle est  unique à homothétie  près. 
\end{proof}

\begin{remas}
\label{rema:transv_affine}
(1) Dans chaque demi-espace $U^\pm$ (voir
remarque~\ref{rema:demi_espaces_elem}), toute orbite de~$K$ est une transversale
globale à $\mcK^\perp$ sur laquelle le paramètre affine $\xi$ est proportionnel
à $\exp(\la t)$, où $t$ désigne le temps du flot de $K$ (voir les équations
\eqref{equa:struct_affine_adaptee} et \eqref{equa:parametre_affine}).\\ 
\indent (2) Il en résulte que si l'on prend $F>0$, alors $\xi$ décrit
$]0,\infty[$ sur $U^+$ et $]-\infty,0$[ sur $U^-$.
\end{remas}

\par Soit $(X,K)$ une surface lorentzienne saturée par un champ $K$ sans zéros. 
On dira que les feuilles de lumière de $K$ {\em découpent des bandes} si les
composantes connexes de leur complémentaire dans $X$ sont toutes simplement
connexes et si  l'adhérence dans $X$ d'une telle composante connexe est une {\em
bande} au sens de la définition~\ref{defi:rubans_bandes}. Dans une bande de type
Reeb (type III du lemme~\ref{lemm:types_bandes}), les deux feuilles du bord
n'appartiennent pas au même feuilletage de lumière.  {\em Il n'y a donc pas
d'accumulation de bandes de type Reeb dans la surface}.

\begin{prop}[structure projective transverse sur le feuilletage orthogonal]
\label{prop:struct_proj_orth}
Soit $X$ une surface lorentzienne munie d'un champ de Killing $K$ complet et
sans zéros. On suppose que les feuilles de lumière de $K$ sont géodésiquement
incomplètes et qu'elles découpent $X$ en bandes de type Reeb. Alors le
feuilletage orthogonal au champ $K$ admet une structure projective  transverse
invariante par l'action du flot de $K$. Cette structure est unique.
\end{prop}

\begin{proof} Comme les feuilles de lumière de $K$ sont  géodésiquement
incomplètes, il découle du   lemme~\ref{lemm:semi_complet} qu'elles sont toutes
isolées. Soit $\mcK_0$ l'ensemble de ces feuilles et soit $X_0$ leur réunion.
Pour chaque feuille de lumière $c\in \mcK_0$, la réunion de $c$ et des deux
composantes de $X\smallsetminus X_0$ adhérentes à $c$ définit un voisinage
ouvert $U_c$ de $c$ qui est un domino (au sens de la
définition~\ref{defi:rubans_bandes}) maximal. On note $\la_c$ la valeur de
$L\cdot \langle K,K \rangle$ sur $c$, réel non nul bien défini. Considérons deux
ouverts $U_c$ et $U_d$ ($c,d \in \mcK_0$) contigus, munis de cartes adaptées. 
Avec les notations précédentes, on a $U_c\cap U_d = U_c^\eps = U_d^{-\eps}$
($\eps= 1$ ou $-1$) grâce au type Reeb et à la maximalité. De plus, les champs
$L$ induits sur $U_c\cap U_d$ par $U_c$ et $U_d$ sont distincts : $\partial_x$
et $-\partial_x+ (2/f(x))\partial_y$ (dans une carte adaptée).  Le changement
$\theta : I_c^\eps \times \R \to I_d^{-\eps}\times \R$ de coordonnées adaptées
doit échanger les champ verticaux (qui correspondent à $K$) et les deux champs
$L$ précédents. On trouve ainsi que $-\theta$ correspond nécessairement à l'une
des réflexions de $U_c^\eps$, c'est-à-dire
\begin{equation}
\label{equa:chg_carte_reeb}
-\theta(x,y)= (x, \frac{2}{\la} \ln \frac{\la\xi_0}{f(x)\ph(x)} -y)
\esp (x,y)\in I_c^\eps\times\R,
\end{equation}
où $\xi_0$ est le paramètre de la géodésique fixe par la réflexion (voir
\eqref{equa:parametre_affine}).  En particulier on doit avoir $\la_c+\la_d
= 0$ et $I_d^{-\eps} =I_c^\eps$.
 
\par Pour comparer les structures affines transverses induites par $U_c$ et
$U_d$ sur le feuilletage orthogonal $\mcK^\perp$, on prend une feuille
arbitraire de $K$ dans $U_c \cap U_d$, qui fournit une transversale globale au
feuilletage $\mcK^\perp$ sur $U_c \cap U_d$. Sur une telle transversale, les
paramètres affines $\xi_c$ et $\xi_d$ induits par $U_c$ et $U_d$ sont
proportionnels à $\exp(\la_c t)$ et $\exp(\la_d t)$ respectivement,~$t$ étant le
temps du flot de $K$ (voir remarque~\ref{rema:transv_affine}).  Le produit
$\xi_c \xi_d$ est constant puisque $\la_c+\la_d = 0$. Les différentes structures
affines s'assemblent  donc  bien en une structure projective.

\par 
Considérons $\zeta$ une autre structure projective transverse à $\mcK^\perp$ et
invariante par le flot $\Phi^t$ de $K$. Soit $U$ un ouvert (maximal) de carte
adaptée et soit $\tau$ une transversale globale à $\mcK^\perp$ dans $U$. On a
une carte projective $\tau$ à valeurs dans un ouvert connexe $J$ de la droite
projective. Pour $t\neq 0$, l'action de $\Phi^t$ induit une bijection projective
directe de $J$ dans $J$ ayant un unique point fixe (la feuille de lumière de $K$
dans $U$). Quitte à composer la carte projective, on peut ainsi supposer que
$J=\R$ et que $\Phi^t$ induit une homothétie de rapport $\neq 1$.  La structure
induite $\zeta|U$ contient donc la structure affine (unique) de la
proposition~\ref{prop:struct_affine_orth}. Mais la structure projective
construite ci-dessus est engendrée par l'atlas des structures affines sur les
ouverts de cartes adaptées, d'où l'unicité.
\end{proof}

\begin{exem}[cylindres]
\label{exem:cylindres}
Soit $(U,K,c_0)$ un domino lorentzien avec $c_0$ géodésiquement incomplète et
soit $\sigma^\pm$ une réflexion de chaque demi-espace ouvert $U^\pm$.  On
considère deux copies de $U$ avec des champs opposés : $(U_i,K_i)$ avec $U_i=U$
et $K_i=(-1)^i K$ ($i=0,1$). Cela étant, soit $C$ la surface obtenue en
identifiant $U_0^+$ et $U_1^+$ (resp. $U_0^-$ et $U_1^-$) avec $\sigma^+$ (resp.
 avec $\sigma^-$). Cette surface étend $(U,K)$ et elle est clairement
homéomorphe au cylindre $\SSS^1\times \R$, les fibres $\{x\}\times \R$ ($x\in
\SSS^1$) correspondant au feuilletage orthogonal $\mcK^\perp$ (dont deux
feuilles sont de lumière). Ainsi l'espace des feuilles de $\mcK^\perp$ sur $C$
est un cercle recouvert par deux cartes affines.  Fixons un paramètre affine
$\xi$ pour $\mcK^\perp$ sur $U$, et notons $\xi_0^\pm$ le paramètre de l'axe de
la réflexion $\sig ^\pm$.  Dans $U^\pm$, en prenant une orbite de $K$ comme
transversale, on voit que l'action de $\sig^\pm$ sur $\mcK^\perp$ est donnée par
$\xi\xi'=(\xi_0^\pm)^2$ (remarque~\ref{rema:transv_affine}). Par suite
l'holonomie de la structure projective transverse à $\mcK^\perp$ est
l'homothétie de rapport $(\xi_0^+/\xi_0^-)^2$ sur la droite projective -- les
feuilles de lumière de $K$ correspondent à $0$ et $\infty$. Ce réel
caractérise~$C$ à isométrie près.
\end{exem}

\subsection{Adjonction de  selles}
\label{subs:compl_selle}
Nous nous intéressons maintenant à la L-complétude, et en premier lieu à la
complétude géodésique des feuilles de lumière d'un champ de Killling $K$
(complet).  Une telle feuille $c_0$ est géodésiquement incomplète si et seulement si $\nabla_K
K \neq 0$ sur $c_0$ (voir lemme~\ref{lemm:semi_complet}). Les feuilles de
lumière des champs de Killing sont donc génériquement incomplètes.  Pour lever
cette obstruction, nous allons montrer que nos extensions peuvent être
complétées \og par adjonction de selles\fg\ (voir proposition
\ref{prop:compl_quasi_selle}).

\begin{prop}[extension par une selle symétrique]
\label{prop:extension_selle}
Soit $(U,K)$ un domino lorentzien dont l'unique feuille de lumière de $K$ est
géodésiquement incomplète. La surface $(U,K)$ admet alors une extension
simplement connexe saturée $(\widehat{U}_s,\widehat{K})$ dont le
champ~$\widehat{K}$ possède un unique zéro $p\in \widehat{U}_s$ et dont la
métrique est symétrique par rapport à~$p$. Cette extension est partagée en 4
quadrants ouverts séparés par deux géodésiques de lumière complètes passant
par~$p$ ; chaque demi-plan ouvert est isométrique à $U$.
\end{prop}

\begin{proof} En coordonnées adaptées $(x,y)\in I\times \R$, 
 la métrique est donnée par 
\eqref{equa:carte_adaptee} avec $\la=f'(0)\neq 0$.
Soit $\ph$ une solution non nulle et positive (sur $I$)
de l'équation différentielle linéaire~\eqref{equa:ED_struct_affine}.
La fonction $f$ s'écrit $f(x)=\la x g(x)$
($x\in I$) avec $g\in \mcC^\infty(I,\R)$ vérifiant $g>0$ et
$g(0)=1$. Considérons les changements de coordonnées réciproques
\begin{equation}
\label{equa:xy_uv}
\left\{
\begin{array}{l}
 u = x [\ph(x) g(x)]^{1/2} \exp(\frac{\lambda y}{2})\\
 v = [\ph(x) g(x)]^{-1/2} \exp(-\frac{\lambda y}{2})
\end{array}
\right.
\esp \mathrm{et}\esp
\left\{
\begin{array}{l}
x = uv \\
y = - \frac{1}{\la} \ln[v^2 \ph(uv)g(uv)]
\end{array}
\right.
\end{equation}
avec $(x,y) \in I \times \R$ et $(uv,v)\in I\times ]0,\infty[$. Dans les 
coordonnées $(u,v)$, la métrique de $U$ prend la forme
\begin{equation}
\label{equa:demi-selle}
\frac{1}{\la} [v^2 h(uv)du^2 - 2 (g(uv)+\frac{1}{g(uv)})du dv
+ u^2 h(uv)dv^2] \esp (uv\in I ,v>0),
\end{equation}
avec $h(x)=\frac{1}{x}(g(x)-\frac{1}{g(x)})$ ($x\in I$) et le champ $K$
correspond à $\frac{2}{\la}(u \partial_u - v\partial_v)$. La métrique et le
champ sont clairement définis pour tout $v\in \R$ et sont invariants 
par la symétrie centrée à l'origine, d'où le prolongement cherché.  
\end{proof}

Nous appellerons {\em selle} (voir définition~\ref{defi:selles})
l'extension $(\widehat{U}_s,\widehat{K})$ que nous venons de
cons\-trui\-re. La proposition~\ref{prop:extension_selle} signifie que la
surface $(U,K)$ peut toujours être réalisée comme \og moitié
ouverte\fg\ d'une selle, que l'on peut choisir symétrique. Une extension
par selle symétrique dont une moitié ouverte est isométrique à $(U,K)$ 
sera dite {\em   minimale}.  Nous examinerons l'unicité d'une telle 
extension au \S\ref{subs:unicite_selle}.%

\par
Observons maintenant que le revêtement universel $\widehat{U}_u$ de
$\widehat{U}_s \smallsetminus \{p\}$ est isométrique au revêtement universel des
cylindres $C$ de l'exemple~\ref{exem:cylindres} -- lequel admet un domaine
fondamental ouvert isométrique à $U$. Par conséquent, l'un des quotients de
$\widehat{U}_u$ se complète en la selle $\widehat{U}_s$ par l'adjonction d'un
zéro du champ. Dans ce qui suit, nous établissons un critère de complétion qui
permet notamment de déterminer le cylindre $C$ dont le revêtement double se
complète en une selle $\widehat{U}_s$ (exemple~\ref{exem:cylindre_selle}).

\begin{defi} 
\label{defi:selles}
Soit $X$ une surface lorentzienne munie d'un champ de Killing $K$ complet.
 On dit que $(X,K)$ est 
\begin{enumerate}
\item une {\em selle} si $X$ est simplement connexe et si $K$ admet un 
unique zéro, appelé  {\em point selle},
\item une {\em pseudo-selle} si $X$ est homéomorphe au cylindre, et si $K$
  ne s'annule pas et possède exactement 4 feuilles de lumière
découpant $X$ en 4 bandes de type Reeb,
\item une {\em quasi-selle} si $X$ est une pseudo-selle et si les~4
feuilles de lumière du champ~$K$ sont {\em   géodésiquement incomplètes}.
\end{enumerate}
\end{defi}

\begin{rema} Comme les réflexions d'une bande de type Reeb s'étendent au
  bord (remarque~\ref{rema:bande_Reeb_sym}), les 4 feuilles de lumière du
  champ d'une pseudo-selle sont de même nature, géodésiquement complètes ou
  non.  Elles peuvent être  géodésiquement complètes.  Par
  exemple, la métrique $2 du dv (u^2+v^2)^{-3/2}$ sur
  $\R^2\smallsetminus\{0\}$, invariante par le champ $u^3\partial_u -
  v^3\partial_v$, est une pseudo-selle mais pas une quasi-selle.
\end{rema}

D'après la proposition~\ref{prop:struct_proj_orth}, le feuilletage
orthogonal au champ d'une quasi-selle admet une structure projective
transverse invariante.

\begin{prop}[complétion d'une quasi-selle]
\label{prop:compl_quasi_selle}
Une quasi-selle se complète en une selle (par l'adjonction d'un zéro du
champ) si et seulement si l'holonomie de la structure projective transverse du
feuilletage orthogonal au champ vaut l'identité.
\end{prop}

\begin{proof} 
Soit $X$ une quasi-selle et soit $\mcK^\perp$ le feuilletage orthogonal au champ
de Killing~$K$. On peut recouvrir $X$ par 4 ouverts maximaux $U_i$ ($i$ entier
modulo $4$) munis de cartes adaptées. On suppose que $U_i\cap U_{i+1} \neq
\emptyset$ ($i$ modulo 4). L'holonomie de la structure projective transverse à
$\mcK^\perp$ provient de l'ambiguïté de la coordonnée verticale $y$ des cartes
adaptées : quand on \og développe\fg\ celles-ci autour de la singularité, les
feuilles de $K$ restent fixes mais il se produit un décalage de la coordonnée
$y$, c'est-à-dire du temps du flot de $K$. Soit $\xi_i$ un paramètre affine
définissant $\mcK^\perp$ dans $U_i$ ($i=1,\ldots 4$).  Sur $U_i\cap U_{i+1}$,
les changements de paramètres affines sont donnés par $\xi_{i+1} \xi_i =
\xi_{i+1}(\ga_i) \xi_i (\ga_i)$ sur $U_i\cap U_{i+1}$, $\ga_i$ étant une feuille
(arbitraire) de $\mcK^\perp|_{U_i\cap U_{i+1}}$ ($i$ modulo~4).  Le
développement des cartes adaptées induit sur l'ouvert $U_5=U_1$ d'un paramètre
affine $\xi_5$ et l'holonomie s'écrit
\begin{equation}
\label{equa:holonomie}
\xi_5= \eta \xi_1 \esp
\mathrm{avec} \esp
\eta= \frac{\xi_1(\ga_1)\xi_2(\ga_1)\xi_3(\ga_3)\xi_4(\ga_3)}
{\xi_1(\ga_4)\xi_2(\ga_2)\xi_3(\ga_2)\xi_4(\ga_4)}.
\end{equation}
De plus, le décalage de la coordonnée verticale est donné par 
$y'= y + \frac{1}{\la}\ln \eta$ (voir \eqref{equa:struct_affine_adaptee}).  
La formule \eqref{equa:holonomie} laisse clairement apparaître
l'indépendance de l'holonomie par rapport aux choix des paramètres affines.
On peut ainsi modifier $\xi_2$, $\xi_3$ et $\xi_4$ pour avoir 
$\xi_{i+1} \xi_i = 1$ ($i=1,2,3$) et $\eta^{-1}=\xi_1(\ga_4)\xi_4(\ga_4)
=\xi_1\xi_4$.

\par Soit $X$ une selle. Dans la carte exponentielle \eqref{equa:carte_exp}
de coordonnées $(u,v)$ centrée au point selle, la métrique est invariante
par le champ $u\partial_u -v\partial_v$ et définie sur un voisinage saturé
de l'origine. Le feuilletage $\mcK^\perp$ correspond aux  droites issues
de l'origine, voir \S\ref{subs:unicite_selle}. La structure projective
transverse est donc la structure usuelle sur le cercle, vu comme revêtement
double de la droite projective.  Son holonomie est évidemment triviale.

\par Inversement, considérons une quasi-selle $X$ dont l'holonomie est triviale.
Avec les notations ci-dessus, chaque ouvert $U_i$ admet une carte de la forme
\eqref{equa:demi-selle} dont l'image est une \og demi-selle\fg. D'après ce qui
précède, on peut choisir des paramètres affines $\xi_i$ tels que $\xi_{i+1}
\xi_i = 1$ pour tout $i$ modulo~4. Cela permet d'assembler les 4 cartes en une
carte globale de $X$ à valeurs dans le plan $(u,v)$. Plus précisément, chaque
$U_i$ admet des coordonnées adaptée $(x_i,y_i)\in I_i\times\R$, où la feuille de
lumière de $K$ correspond à  $x_i=0$. On choisit $\xi_1$ tel que $\xi_1/x_1 <0$
sur $U_1$ ; comme $\xi_{i+1} \xi_i >0$, on a  $(-1)^i \xi_i/x_i >0$ pour tout
$i$ modulo 4. Si l'on pose comme dans~\eqref{equa:xy_uv} $$u_i=x_i \sqrt{(-1)^i
\xi_i / x_i}, \esp \mathrm{et} \esp v_i= \sqrt{(-1)^i x_i / \xi_i}, $$ la
condition $\xi_{i+1} \xi_i = 1$ entraîne $(u_{i+1},v_{i+1})=(v_i,-u_i)$ sur
$U_i\cap U_{i+1}$ (sur cet ouvert, on a $x_{i+1} = -x_i > 0$). À partir de la
relation $(\xi_i)'_{y_i}=(-1)^i \lambda \xi_i$, on voit également que le
champ~$K$ s'écrit $(-1)^i \frac{\lambda}{2} (u_i\partial_{u_i} -
v_i\partial_{v_i})$ en coordonnées $(u_i,v_i)$. On obtient ainsi une carte
globale de~$X$ à valeurs dans le plan $(u,v)=(u_1,v_1)$. La métrique est définie
sur un voisinage de $0$ de la forme $a_0< uv <b_0$, excepté en $0$, et elle est
invariante par le champ $u\partial_u -v\partial_v$. Une telle métrique se
prolonge automatiquement à l'origine, proposition~\ref{prop:met_inv_st}. En
développant de cette façon une quasi-selle d'holonomie quelconque $\eta$, on
obtient  sur $U_5=U_1$ des coordonnées $(u_5,v_5)$ et $(u_1,v_1)$  telles que
$(u_5,v_5)=(u_1/\sqrt{\eta},\sqrt{\eta}v_1)$, et on retrouve le décalage du flot
(le champ $K$ s'écrit $-\frac{\lambda}{2} (u_1\partial_{u_1} -
v_1\partial_{v_1})$ dans $U_1$).
\end{proof}

\begin{exem}[cylindres, suite]
\label{exem:cylindre_selle}
Le cylindre $C$ de l'exemple~\ref{exem:cylindres}, défini par deux réflexions
d'axes $\xi_0^\pm$, admet un revêtement double qui est une quasi-selle. Celle-ci
se complète en une selle (symétrique) si et seulement si $|\xi_0^+| = |\xi_0^-|$.
\end{exem}

\subsection{Unicité des selles}
\label{subs:unicite_selle}

\begin{prop}
\label{prop:met_inv_st}
Soit $\Omega^*=\{(u,v)\in \R^2\smallsetminus\{0\}; uv \in I \} $, où $I\subset
\R$ est un intervalle ouvert comprenant $0$.  Toute métrique lorentzienne $g$
lisse (resp. analytique) sur $\Omega^*$ et invariante par le champ
$K=u\partial_u-v\partial_v$ est de la forme
\begin{equation}
\label{equa:met_inv_st}
g=v^2\alpha du^2+2\beta dudv+u^2 \gamma dv^2,
\end{equation}
où $\alpha,\beta$ et $\gamma$ sont  des fonctions lisses (resp. analytiques)
sur $\Omega^*$ et constantes le long du flot de $K$.  De plus $g$ se prolonge à
l'origine en une métrique lisse (resp. analytique). Dans le cas analytique,
les fonctions $\alpha,\beta$ et $\gamma$ ne dépendent que du produit $uv
\in I$ et la métrique est symétrique par rapport à l'origine.
\end{prop}

\begin{proof}
Posons $g=a\, du^2+2 b\, dudv+ c\, dv^2$. La relation $\mcL_K g = 0$ se
traduit par $K\cdot b = 0$, $K\cdot a= -2a$  et $K\cdot c = 2c$.  Par suite
$b$, $v^ {-2}a $ ($v\neq 0$) et $u^{-2}c $ ($u\neq 0)$ sont invariantes par
$K$.  Sur chaque demi-plan ouvert $\pm u >0$, $\pm v >0$, toute fonction
invariante par $K$ est fonction lisse (resp. analytique) du produit $uv$
(par exemple sur $v>0$, on pose $x=uv$, $y=-\ln v$ et le champ correspond à
$\partial_x$). Il existe donc deux fonctions $\ph^\pm \in \mcC^\infty(\R,\R)$
telles que 
$$a(u,v)=v^2 \ph^+(uv) ~~~(v>0) \esp \mathrm{et}\esp a(u,v)=v^2 \ph^-(uv)
~~~(v<0).$$
L'examen du développement de Taylor de $a(u_0,v)$ en $v=0$ ($u_0\neq 0$ fixé)
montre que $\ph^+$ et $\ph^-$ ont le même jet d'ordre infini en $0$.  En posant
$\alpha(u,v)=\ph^\pm(uv)$ pour $\pm v>0$ et $\alpha(u,0)=\ph(0)$, on définit
donc par recollement une fonction $\alpha$ invariante et lisse (resp.
analytique) sur $\Om^*\cup \{0\}$, telle que $a = v^2\alpha$ sur $\Om^*$. Dans
le cas analytique on a $\ph^+=\ph^-$ ; la fonction $\alpha$ ne dépend que du
produit $uv$ et elle est invariante par la symétrie centrée à l'origine. On
vérifierait de même que les coefficients~$b$ et~$c$ sont de la forme voulue et
se prolongent à l'origine. De plus $ac-b^2$ est constant sur $uv=0$, donc non
nul à l'origine.
\end{proof}

\begin{prop}[carte exponentielle]
\label{prop:carte_exponentielle}
Soit $X$ une surface lorentzienne lisse (resp. analytique) possédant un champ de
Killing non trivial et complet $K$ s'annulant en un point~$p$.  Il existe
$\rho>0$ et une carte exponentielle centrée en~$p$, définie sur
$\Omega_\rho=\{|uv|< \rho\}$, dans laquelle la métrique s'écrit
\begin{equation}
\label{equa:carte_exp}
g = 2 du dv + \alpha (udv-vdu)^2  \esp (|uv|<\rho),
\end{equation}
où $\alpha : \Omega_\rho \to \R$ est une fonction lisse (resp. analytique)
constante le long du flot de $u\partial_u-v\partial_v$. La métrique est
déterminée au voisinage du point~$p$ par la fonction $\langle K,K\rangle$.
\end{prop}

 \begin{proof}
L'application exponentielle en~$p$ est une semi-conjugaison entre le flot de $K$
et celui de sa partie linéaire $K_0$. Cela entraîne que~$p$ est un point selle
du champ. On peut choisir des coordonnées $(u,v)$ du plan tangent en~$p$ de
sorte que la métrique à l’origine s'écrive $g_0 = 2 du dv$. On a alors
$K_0=C(u\partial_u-v\partial_v)$ ($C\neq0$ constante). L'exponentielle est un
difféomorphisme local près de l'origine, donc sur un voisinage saturé par $K_0$,
de la forme $\Omega_\rho$ ; l'injectivité sur $\Omega_\rho$ se déduit du lemme
de Clairaut et Noether : les rayons géodésiques issus de~$p$ étant orthogonaux
à~$K$, il ne peuvent se recouper qu'en un zéro de~$K$.

\vspace{-\parskip}
Comparons maintenant l'expression $g$ de la métrique en coordonnées $(u,v)$
avec $g_0$. Comme la carte est exponentielle, le champ radial est
orthogonal à $K_0$ pour $g$ et $g_0$, et de même norme pour ces deux
métriques. De plus ce champ engendre le noyau de la forme $\omega_0=
udv-vdu$. Par conséquent, en dehors des axes $uv=0$, le tenseur $g-g_0$ est
de rang au plus 1 et proportionnel à $\omega_0^2$ (en tant que forme
quadratique). Puisque cette condition est fermée, il existe par continuité
une fonction $\al$ définie sur $\Omega_\rho$ et vérifiant
\eqref{equa:carte_exp}. En évaluant localement sur un champ de vecteurs qui
n'annule pas $\omega_0$, on voit que $\al$ est lisse (resp. analytique). De
plus $\al$ est invariante par $K_0$ car toutes les autres quantités de
\eqref{equa:carte_exp} le sont. Enfin, si $N$ désigne l'expression de
$\langle K,K \rangle $ dans notre carte, la fonction $\alpha$ est
déterminée par la relation
$$4 C^2 u^2v^2 \alpha(u,v) = N(u,v)  + 2 C^2 uv
\esp  (|uv|< \rho).$$ 
\end{proof}

\begin{rema} 
\label{rema:germes_de_selles}
Dans chaque composante du complément du cône isotrope en~$p$ on choisit
arbitrairement un rayon géodésique $r_i$ issu de~$p$
($i=1,\ldots,4$). L'étude des fonctions invariantes par $u\partial_u
-v\partial_v$ (preuve de la proposition~\ref{prop:met_inv_st}) montre que
la donnée de $\langle K,K\rangle$ comme fonction de la longueur d'arc sur
$r_i$ ($i=1,\ldots,4$) suffit à déterminer la métrique. Les germes de
selles lisses sont donc paramétrés par des quadruplets de germes lisses
$f_i:[0,\delta[ \to \R$ ($i=1,\ldots,4$) ayant le même jet infini en
$0$. Dans le cas analytique, les germes de selles sont paramétrés par
des germes de fonctions analytiques en $0\in\R$.
\end{rema}

\begin{prop}[unicité de l'extension par selle symétrique]
\label{prop:unicite_selle_sym}
Soit $(U,K)$ un domino lorentzien  dont l'unique feuille de lumière de
$K$ est géodésiquement incomplète. L'extension de $(U,K)$ par une selle
symétrique minimale (proposition~\ref{prop:extension_selle}) est unique à
isométrie près.
\end{prop}

\begin{proof} C'est une conséquence immédiate de la 
proposition~\ref{prop:carte_exponentielle} car l'extension est symétrique 
(voir également la remarque~\ref{rema:germes_de_selles}).
\end{proof}

On remarquera que dans le cas lisse, la surface $(U,K)$ admet une infinité
d'extensions par des selles non symétriques, dont les germes au point selle sont
paramétrés par un couple de germes de fonctions $[0,\delta[\to \R$. Dans le cas
analytique, le germe au point selle de l'extension est unique à isométrie près.
La surface $(U,K)$ admet également une infinité d'extensions par des
quasi-selles symétriques, exemple~\ref{exem:cylindre_selle}. Elles s'obtiennent
à partir de la selle symétrique $(\widehat{U}_s,\widehat{K})$ par découpage et
recollement avec un décalage non nul le long d'une feuille de lumière de
$\widehat{K}$. On sait qu'une telle métrique n'est pas isométrique à
$\widehat{U}_s\smallsetminus\{p\}$ et ne se prolonge jamais en~$p$,
voir \S\ref{subs:compl_selle}. Ce phénomène est cohérent avec l'unicité de la
proposition~\ref{prop:carte_exponentielle}.

\begin{prop}
\label{prop:ext_isom_demi_selle}
Soient $(X_1,K_1)$ et $(X_2,K_2)$ deux selles symétriques de points selles $p_1$
et $p_2$.  Toute isométrie d'une demi-selle (ouverte) de $X_1$ dans une
demi-selle de $X_2$ s'étend de fa\c con unique en une isométrie entre $X_1$ et
$X_2$. Toute isométrie de $X_1\smallsetminus \{p_1\}$ dans $X_2\smallsetminus
\{p_2\}$ s'étend de fa\c con unique en une isométrie entre $X_1$ et $X_2$.
\end{prop}

\begin{proof} 
Dans les deux cas, les selles $X_1$ et $X_2$ sont isométriques d'après la
proposition~\ref{prop:unicite_selle_sym}. On peut donc supposer que $X_1=X_2$.
De plus, il est bien connu que l'extension cherchée, si elle existe, est unique.
 Soit $(U,K)$ une demi-selle de $X$. Si $U$ est à courbure constante non nulle
alors elle contient des géodésiques de lumière maximale qui ne sont
semi-complètes ni dans un sens ni dans l'autre, ce qui entraîne qu'elle ne
contient pas de domaine de de Sitter. Si $U$ contient un demi-plan de Minkowski
alors $X$ est égal au plan de Minkowski et l'énoncé est vérifié. Ainsi, grâce au
lemme~\ref{lemm:isom_isomK}, on peut supposer que toute isométrie $\Phi$ de $U$
envoie $K$ sur $\pm K$.  Alors $\Phi$ préserve l'unique feuille de lumière de
$K$ sur~$U$.  Quitte à composer par le flot de $K$ on peut donc supposer que
$\Phi$ fixe un point sur cette feuille. Sa différentielle en ce point ne peut
être que $\pm \Id$. Si $d\Phi=-\Id$ alors $\Phi$ échange $U^+$ et $U^-$, ce qui
est impossible vu que $\langle K,K\rangle$ change de signe. Par conséquent
$d\Phi=\Id$ et donc $\Phi=\Id$.
\par

Les réflexions génériques des bandes d'une selle symétrique s'étendent à toute
la selle. En coordonnées normales, il s'agit des applications $\sigma(u,v)=\pm 
(v,u)$ et de leurs conjuguées par le flot de $K$. Si $\Psi$ est une isométrie de
$X\smallsetminus\{p\}$, on peut donc supposer que $\Psi$ préserve une demi-selle
et refaire le raisonnement ci-dessus.
\end{proof}

\section{Extensions maximales}
\label{sect:ext_max}

\subsection{Construction de surfaces maximales}
\label{subs:const_sf_max}
Soit $I$ un intervalle ouvert non vide de $\R$ et soit $f\in
\mcC^\infty(I,\R)$ (resp. $f \in \mcC^\omega(I,\R)$).  On appelle {\em
  ruban associé à $f$}, noté $R_f=(\msfR,\msfk)$, la surface 
$\msfR=I\times \R$ munie de la métrique et du champ définis par
\begin{equation}
\label{equa:rub_f}
2dxdy + f(x)dy^2 ~~~\mathrm{et} ~~~ \msfk = \partial_y \esp (x,y)\in 
\msfR = I \times \R.
\end{equation} 
On dira que {\em $f$ est inextensible} si elle ne se prolonge pas (avec la
même régularité, lisse ou analytique) à un intervalle ouvert contenant
strictement $I$.  Enfin, on note $f^\vee$ la fonction \og miroir \fg\ de
$f$, définie pour $-x\in I$ par $f^\vee(x)=f(-x)$.

\par La notion de maximalité introduite à la définition~\ref{defi:extension}
est spécifique aux surfaces munies d'un champ de Killing. Il convient de la
comparer à la notion usuelle. Les surfaces à bord n'étant jamais maximales
au sens usuel ni au sens de la définition~\ref{defi:extension}
(exemple~\ref{exem:ext_sf_a_bord}), nous supposerons dans cette discussion
que les surfaces sont sans bord.

\begin{prop}
\label{prop:max_champ=max_usuel}
Une surface lorentzienne saturée $(X,K)$ est maximale (parmi les surfaces
munies d'un champ de Killing complet) si et seulement si la surface $X$ est maximale au
sens usuel.
\end{prop}

\begin{proof}
 Soit $X'$ une surface lorentzienne (sans bord) et soit $X$ un ouvert 
propre de~$X'$ muni d'un champ de Killing complet $K$.  Il faut établir que
$(X,K)$ admet une extension propre saturée $(X'',K'')$.  Comme deux points
arbitraires de $X'$ sont toujours reliés par une géodésique de lumière
brisée, il existe une feuille de lumière $\ell'$ de~$X'$ qui rencontre~$X$
et $X'\smallsetminus X$.  Supposons dans un premier temps que $\ell' \cap
X$ possède une composante $\ell$ transverse à~$K$.  Grâce à cette hypothèse
de transversalité, la feuille $\ell'$ admet un paramétrage géodésique
$\gamma : J \to X'$, avec $J$ intervalle ouvert de $\R$, tel que $\langle
K,\ga'(x)\rangle =1$ sur $\ell$ et cette composante~$\ell$ est contenue
dans une carte adaptée $U$ de la forme \eqref{equa:rub_f}. La fonction
courbure au point $\gamma(x)$, $x\in J$,  vaut $f''(x)/2$ sur $\ell$. Par
suite la fonction~$f$ s'étend en une fonction~$g$ définie sur~$J$. On
définit une extension propre $(X'',K'')$ de $(X,K)$ en lui recollant
$R_g$ le long de~$U$.

\par On peut désormais supposer que pour toute feuille de lumière
$\ell'$ de $X'$ 
rencontrant~$X$ et $X'\smallsetminus X$, les composantes de $\ell'\cap X$ sont
des orbites de lumière du champ $K$, nécessairement semi-complètes
géodésiquement et isolées (lemme~\ref{lemm:semi_complet}); par suite
$\ell'\cap X$ admet au plus deux composantes et toute feuille de lumière de
$X'$ voisine de $\ell'$ est entièrement contenue dans~$X$. Fixons une telle
feuille de lumière $\ell'_0$. L'intersection $F=\ell'_0\cap
(X'\smallsetminus X)$ est un intervalle fermé non vide de
$\ell'_0$. L'existence d'un point $q$ intérieur à $F$ dans $\ell'_0$ est
exclue par ce qui précède (considérer la deuxième feuille de lumière de
$X'$ passant par $q$).  Finalement $F$ se réduit à un point $p\in X'$ isolé
dans $\partial X$.  Il reste à montrer que le
champ $K$ se prolonge au point $p$ en un champ de Killing.

\par Il existe une carte centrée en $p$ de coordonnées $(u,v)\in\R^2$
dans laquelle les feuilletages de lumière de $X'$ au voisinage de $p$ sont
donnés par $u=\cte$ et $v=\cte$.  Le flot local de~$K$ dans cette carte
préserve les feuilletages de lumière.  Son expression est donc de la forme
$(\phi^t(u),\psi^t(v))$, où $\phi^t$ et $\psi^t$ sont les flots
locaux de champs de vecteurs lisses $\al(u)\partial_u$ et
$\be(v)\partial_v$ définis sur $\R$. En particulier, le champ
$\al(u)\partial_u + \be(v)\partial_v$ est défini pour $(u,v)=(0,0)$ et
permet de prolonger $K$ au point $p$. Par continuité, ce prolongement est
un champ de Killing sur $X\cup\{p\}$.
\end{proof}

\begin{defi}
\label{defi:selle_a_l_infini}
Soit $(X,K)$ une surface lorentzienne saturée. On dit que $(X,K)$ {\em n'a
  pas de selles à l'infini} si (dans le revêtement universel) toute orbite
de lumière du champ est incluse dans une géodésique de lumière {\em
  complète}.
\end{defi}

\begin{prop}
\label{prop:sf_Euf}
Pour toute fonction lisse (resp. analytique) $f :I \to \R$ il existe une
surface lorentzienne $(E^u_f,K^u)$ lisse (resp. analytique), saturée,
homéomorphe au plan~$\R^2$ et telle que
\begin{enumerate}
\item[(i)] tout ruban maximal contenu dans $(E^u_f,K^u)$ est isométrique 
à $R_f$ ou à  $R_{f^\vee}$,
\item[(ii)] $(E^u_f,K^u)$ n'a pas de selles à l'infini. 
\end{enumerate}
De  plus, $E^u_f$ est maximale si et seulement si $f$ est inextensible, et  
$E^u_f$ est  L-complète si et seulement si $f$ est définie sur $\R$.
\end{prop}

\begin{proof} 
Il convient d'observer que dès que $f$ s'annule sans être constante,
$R_f$ ne satisfait pas la condition~(i) : en effet $R_f$ contient
alors une bande ou une demi-bande, donc un ruban maximal $R_{g^\vee}$ où
$g$ est une restriction non triviale de $f$. Si $f=0$ ou si $f$ ne s'annule
pas, on pose $E^u_f=R_f$, dont les rubans maximaux sont $R_f$ (unique
si $f=0$) ou $R_f$ et $R_{f^\vee}$, et qui satisfait évidemment la
condition~(ii).  Dorénavant nous supposons que la fonction $f$ s'annule et
n'est pas constante.

\par Soit $Z_0\subset I $ le lieu des zéros de $f$ (par hypothèse non vide
et distinct de $I$), correspondant aux orbites de lumière du champ
$\msfk$. D'après le lemme~\ref{lemm:semi_complet}, celles qui sont
géodésiquement complètes sont paramétrées par $Z_{00}=Z_0\cap
\{f'=0\}$. Les points de $Z_0\smallsetminus Z_{00}$ seront appelés {\em
  zéros simples de $f$} ; s'il en existe, ils sont associés aux orbites de
lumière semi-complètes et sont isolés dans $Z_0$.  La réunion des orbites
de lumière $Z_0\times \R$ découpe sur $\msfR=I\times \R$ des bandes
standards de type I (adhérences dans $\msfR$ des composantes de
$(I \smallsetminus Z_0)\times \R$), en nombre au plus dénombrable.
 
\par Nous construisons une première extension (non simplement connexe) en
recollant des copies de $(\msfR,\msfk) = R_f $ et de $(\msfR,-\msfk)
\simeq R_{f^\vee}$.  Soit $A$ un groupe abélien dénombrable,
suffisamment gros pour que l'on puisse indexer les composantes connexes de
$\msfR\smallsetminus \msfR_0$ par une partie $S\subset A$, pour l'instant
arbitraire.  Ces composantes seront appelées \og carrés ouverts\fg, même si
potentiellement deux d'entre elles -- aux extrémités du ruban -- pourraient
être des bandes de type III ouvertes.  Notons $(C_\alpha)_{\alpha\in S}$
(resp. $(C_\alpha^\vee)_{\alpha\in S}$) les carrés ouverts de
$(\msfR,\msfk)$ (resp.  de $(\msfR,-\msfk)$) et, pour chaque $\alpha\in S$,
fixons une réflexion $\sigma_\alpha$ qui échange $C_\alpha$ et
$C_\alpha^\vee$, champs de Killing compris.  Considérons maintenant deux
familles de rubans indexées par le groupe $A$, définies par
$H_i = (\msfR,\msfk)$ ($i\in A$) et $V_j =  (\msfR,-\msfk)$ ($j\in
A$). 
Les copies respectives dans $H_i$ et $V_j$ des $C_\alpha$ et $C_\alpha^\vee$
($\al\in S$) sont notées $C^H_{i,\alpha}$ et $C^V_{j,\alpha}$.  Soit enfin
$(Y,\msfk_Y)$ la surface lorentzienne (sans bord, munie d'un champ de
Killing $\msfk_Y$) définie comme quotient de la surface $(\coprod_{i\in A}
H_i) \amalg ((\coprod_{j\in A} V_j)$
par la relation  d'équivalence $\mcR$ engendrée par
\begin{equation}
\label{equa:sf_Y}
 p\mcR q ~~ \mathrm{si}~ p=q ~~\mathrm{ou}~~ (p,q)\in
C^H_{i,\alpha}\times C^V_{j,\alpha}, ~ \alpha=i+j \in S ~\mathrm{et} ~
q=\sigma_\alpha(p). 
\end{equation} 
Autrement dit, les surfaces lorentziennes $X^H=\coprod_{i\in A} H_i$ et
$X^V= \coprod_{j\in A} V_j$ sont recollées au moyen de l'isométrie
partielle (involutive) $\sigma$ définie par~\eqref{equa:sf_Y} grâce aux
$\sig_\al$. En particulier, pour tout $(i,j)\in A^2$, le ruban \og
horizontal\fg\ $H_i$ est recollé au ruban \og vertical \fg\ $V_j$ le long
d'au plus un carré $C_\alpha$ (si $\alpha=i+j\in S$) {\it via} la réflexion
$\sigma_\alpha$.
On vérifie que $\mcR$ est une relation
ouverte. La surface $(Y,\msfk_Y)$ contient des images isométriques des deux
surfaces $X^H$ et $X^V$ dont les rubans maximaux, encore 
notées $H_i$ ($i\in A$) et $V_j$ ($j\in A$), forment un atlas dénombrable
de $Y$ ; noter que $Y$ admet également un atlas à deux cartes de domaines
$X^H$ et $X^V$. 

\par
L'isométrie de recollement $\sigma$ est définie sur $U^H = \amalg_{i\in A}
\cup_{\al\in S} C^H_{i,\al}$.  Soit $(p_n)$ une suite de points de $U^H$ qui
converge vers $p\in \partial U^H$. Il existe $i\in A$ et une suite $(\al_n) \in
S^\N$ tels que $p_n\in C^H_{i,\al_n}$ et $p\in H_i$. La suite $\sig(p_n) \in
C^V_{\al_n-i}$ ne peut avoir de valeurs d'adhérence que si $(\al_n)$ est
stationnaire ; dans ce cas, il existe $\al\in S$ tel que $p_n \in  C^H_{i,\al}$
($n$ assez grand) et $\sig(p_n) = \sig_\al(p_n)$ n'a pas de valeurs d'adhérence
(voir la preuve du corollaire~\ref{coro:extension_reflexion}). D'après le
lemme~\ref{lemm:extension}, la surface $Y$ est séparée. Pour étudier la
connexité de $Y$, observons que deux rubans horizontaux (resp. verticaux) sont
reliés par un ruban vertical (resp. horizontal) si et seulement si la
différence de leurs indices appartient à $S-S$. On suppose de plus que {\em
$S-S$ engendre le groupe $A$} : la surface $Y$ est donc connexe. Par
construction, elle vérifie la condition~(i).

\par Les géodésiques de lumière de $Y$ se répartissent en deux types : internes
aux rubans, \mbox{c.-à-d.} transverses au champ $\msfk_Y$ (et complètes quand
$I=\R$), ou associées aux zéros de~$f$. Ces dernières sont complètes ou
semi-complètes selon qu'elles correspondent à $Z_{00}$ ou à $Z_0\smallsetminus
Z_{00}$. Observons maintenant que la surface $Y$ se rétracte par déformation sur
la réunion $\Delta$ des âmes des rubans $(H_i)_{i\in A}$ et $(V_j)_{j\in A}$. Ce
rétract $\Delta$ forme une \og grille \fg\ dont les sommets correspondent aux
carrés $(C^H_{i,\alpha})_{(i,\alpha)\in A\times S}$. On s'intéresse au cycle
d'ordre~4 de~$\Delta$ (ou de~$Y$) défini par deux indices horizontaux $i\neq i'$
et deux indices verticaux $j\neq j'$ tels que $i+j,j+i',i'+j'$ et $j'+i$
appartiennent à $S$: ses sommets $C^H_{i,\alpha}, C^H_{i,\beta},
C^H_{i',\alpha'}$ et $C^H_{i',\beta'}$ sont déterminés par $i,i',j,j'$ ($\alpha
= i+j, \beta = i+j'$, etc). On suppose désormais que {\em le groupe~$A$ est de
2-torsion (par exemple $A\subset (\Z/2\Z)^{(\N)}$) et que la partie $S$ comprend
$0$ et engendre~$A$}. Cette hypothèse additionnelle assure la connexité de $Y$
(voir plus haut) et va permettre, quitte à modifier les réflexions
$(\sig_\alpha)_{\alpha\in S}$, de compléter $Y$ par adjonction de selles.

\begin{figure}[t]
\labellist
\small\hair 2pt
\pinlabel $i$ at 50 50
\pinlabel $i'$ at 50 145
\pinlabel $j$ at 80 180
\pinlabel $j'$ at 175 180
\pinlabel $\alpha$ at 75 55
\pinlabel $\beta$ at 87 150
\pinlabel $\alpha$ at 170 150
\pinlabel $\beta$ at 182 55
\endlabellist
\begin{center}
\includegraphics[scale=0.8]{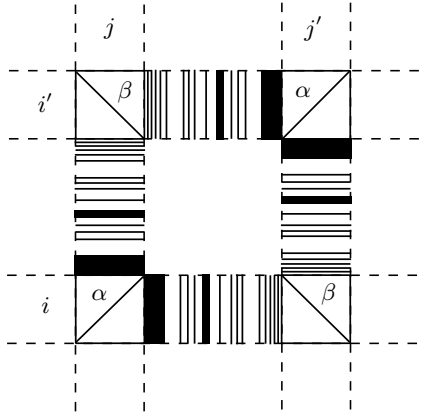}
\caption{Surface $Y$ (groupe $A$ de 2-torsion)}
\label{figu:sf_Y}
\end{center}
\end{figure}

\par Nous remplissons maintenant la condition~(ii), sur $Y$.  Étant donnés
$i\in A$ et $(\alpha,\beta)\in S^2$ avec $\alpha\neq \beta$, en posant $i'=
i + \alpha -\beta = i + \beta -\alpha $, $j=\alpha-i$ et $j'=\beta -i$, on
définit un cycle d'ordre~4 dont les sommets sont $C^H_{i,\alpha},
C^H_{i,\beta}, C^H_{i',\alpha}$ et $C^H_{i',\beta}$ (c'est-à-dire
$\alpha'=\alpha$ et $\beta'=\beta$, voir figure~\ref{figu:sf_Y}).  Toute
orbite de lumière semi-complète du champ doit border deux carrés contigus,
c.-à-d. séparés par un zéro {\em simple} de $f$,
et l'argument précédent 
montre que ces carrés font partie d'une quasi-selle
symétrique (figure~\ref{figu:sf_Y} avec $\al$ et $\be$ contigus, les
diagonales étant les axes des réflexions de $\sig_\al$ et de $\sig_\be$).
Nous définissons un {\em graphe de contiguïté} $\mcG_f=(S,\mfE)$,
\label{graphe_contiguite} 
associé à~$f$, ayant pour sommets les indices $\al\in S$ et pour arêtes les
paires $\{\al,\be\}$ d'indices contigus.  Toutes les quasi-selles de $Y$
associées à une telle paire sont isométriques.  Le graphe $\mcG_f$ est
connexe (resp. de dimension~$0$) quand tous les zéros de $f$ sont simples
(resp. non simples). Si $\al$ et $\be$ sont contigus, on peut toujours
modifier l'une des réflexions $\sigma_\al$ ou $\sigma_\be$ de sorte que les
quasi-selles associées se complètent en selles ; on dira alors que
$\sigma_\al$ et $\sigma_\be$ sont compatibles. Pour toute composante
connexe $\mcC$ de $\mcG_f$, on peut choisir un sommet $\al \in \mcC$ et par
récurrence (sur la distance combinatoire à $\al$) modifier les réflexions
$\sig_\be$ ($\be \in \mcC \smallsetminus\{\al\}$) afin de rendre
compatibles toutes les paires de réflexions associées aux arêtes de $\mcC$.
Après cette opération, la surface~$Y$ se complète par adjonction de selles
(si $\mfE \neq \emptyset$) 
en une surface notée $(Y^s,\msfk_Y^s)$, connexe, paracompacte (en
particulier séparée) et  vérifiant la condition~(ii). Si  $\mfE =
\emptyset$, on pose simplement $(Y^s,\msfk_Y^s)=(Y,\msfk_Y)$.
Soit
enfin $(E^u_f,K^u)$ le revêtement universel de $(Y^s,\msfk_Y^s)$. Il
s'agit évidemment d'une surface homéomorphe au plan~$\R^2$ et 
satisfaisant les
conditions~(i) et~(ii), préservées par passage à un revêtement.

\par Si $E^u_f$ n'est pas maximale, la surface $(E^u_f,K^u)$ admet une
extension non triviale saturée $(E,K)$ (proposition
\ref{prop:max_champ=max_usuel}). Il existe une feuille de lumière maximale
$\ell$ de $E$ qui rencontre $E^u_f$ et $E\smallsetminus E^u_f$.  Par la
condition (ii), $\ell\cap E$ est transverse à $K^u$. Par suite $\ell$ est
transverse à $K$ et admet un paramétrage géodésique $\ga:J\to E$ ($J$
intervalle contenant $I$) tel que $\langle K, \ga'(x) \rangle = 1$. La
fonction norme de $K$ en $\ga(x)$ étend $f$.  Inversement, si $f$ admet une
extension $g$, on obtient une extension de $E^u_f$ en étendant $R_f$,
proprement plongé dans $E^u_f$, par $R_g$ (voir aussi la 
remarque~\ref{rema:Eug_Euf}).  Enfin l'assertion sur la
L-complétude est claire.
\end{proof}

\begin{rema}
\label{rema:variante_Y}
Voici une variante adaptée au cas où les carrés $(C_\al)_{\al \in S}$ ne
s'accumulent pas dans $R$ (par exemple si $f$ est analytique). On prend
$\smash{X'}^H= (\msfR,\msfk)$ et $\smash{X'}^V= (\msfR,-\msfk)$, que l'on
recolle avec
les réflexions $\sig_\al : C_\al \to C_\al^\vee$ $(\al\in S)$ en une
surface $X'$ connexe et séparée. Si $f$ n'a pas de zéros simples, 
on pose ${Y'}^s = X'$. Sinon, on considère un revêtement double non trivial
au voisinage des bouts de $X'$ associés à $\mfE$, que l'on peut compléter
comme plus haut (par adjonction de selles) en une surface ${Y'}^s$.
Cette surface  ${Y'}^s$   est en un certain sens
(voir~\S\ref{subs:quotients_mini}) la plus petite extension de $R_{f}$ par
une surface sans bord et sans selles à l'infini. Aux orbites de lumière du
champ près, ${Y'}^s$  a la même taille que $R_f$ si $f$ n'a pas
de zéros simples, deux fois la taille de $R_f$ sinon. La topologie de
$Y^s$ dépend de $f$ et, si ${Y'}^s \neq  X'$, du choix du revêtement double
de $X'$, voir~\S\ref{subs:quotients_mini}. Dans le cas  où les 
$(C_\al)_{\al \in S}$ s'accumulent, cette construction fonctionne encore, 
mais elle produit des surfaces $X'$ et ${Y'}^s$ non séparées.
\end{rema}

\begin{exems}
Si $f$ est une fonction affine non constante alors la surface $E^u_f$
ci-dessus est le plan de Minkowski muni d'un champ de Killing ayant un
zéro. De m\^eme, si $f$ est une fonction polynomiale de degré $2$ définie
sur $\R$, alors $E^u_f$ est isométrique à un
facteur près au rev\^etement universel de l'espace de de Sitter. Le champ
de Killing associé à la construction est elliptique, parabolique ou
hyperbolique selon que $f$ ait $0$, $1$ ou $2$ racines. Les termes
elliptique, parabolique ou hyperbolique font référence aux sous-groupes de
$\PSL_2(\R)$ qui leurs sont associés (voir \cite{MS} pour plus de détails).
\end{exems}

\subsection{Géométrie locale uniforme}
\label{subs:sf_reflexives}

Soit $I$ un intervalle ouvert de $\R$ et soit $f\in \mcC^\infty(I,\R)$.  On
rappelle (voir l'introduction) que l'orbite de $f$ d\'efinie sous l'action
à droite du groupe affine (resp. affine euclidien) est not\'ee $[[f]]$
(resp. $[f]$).

\begin{defi}
\label{defi:classe_[[f]]}
Soit $X$ une surface lorentzienne  simplement connexe et soit 
\mbox{$f\in \mcC^\infty(I,\R)$.}
\begin{enumerate}
\item Soit $K$ un champ de Killing non trivial et complet sur $E$. On dit que
$(X,K)$ est {\em de classe  $[f]$}  s'il existe une coordonn\'ee transverse
$x\in\mcC^\infty(X,\R)$ (voir proposition~\ref{prop:struct_transv_killing}-(2))
telle que $I=x(X)$ et $\langle K,K\rangle=f\circ x$.
\item  On dit que $X$ est de classe $[[f]]$ si elle possède un champ de Killing
$K'$ tel que $(X,K')$ est de classe $[f]$.
 \end{enumerate}
Par extension, une surface $(X,K)$ (resp. $X$) non simplement connexe est de
classe $[f]$ (resp. de classe $[[f]]$) si son revêtement universel l'est.
\end{defi}

Une surface $(X,K)$ (resp. $X$) sera dite {\em à géométrie locale uniforme}
si elle est de classe $[f]$ (resp. $[[f]]$) pour une certaine fonction
$f$. Si on change de coordonn\'ee transverse, la classe $[f]$ est
inchangée. Si de plus $K$ est remplacé par $\lambda K$ ($\lambda \in
\R^*$), la classe $[[f]]$ est inchangée. Cette classe est unique quand~$X$
est à courbure non constante. Par contre, en courbure
constante, une surface  peut admettre plusieurs classes.  Ainsi le
plan de Minkowski est de classe $[[0]]$, $[[-1]]$, $[[1]]$ et $[[x]]$.

\par Considérons une surface lorentzienne saturée $(E,K)$, simplement
connexe.  Toute feuille de lumière~$\ell$ de~$E$ ne portant pas d'orbite de
$K$ (donc pas de zéro de $K$) doit être partout transverse à $K$.  Par
suite~$\ell$ est contenu dans un {\em ruban maximal} où la métrique
s'exprime sous la forme~\eqref{equa:carte_adaptee}.  Les rubans maximaux
recouvrent~$E$ à l'exception des zéros du champ~$K$.  Notons $E_0$ la
réunion des orbites de lumière de $K$ et supposons que $K$ n'est pas de
lumière. Toute composante connexe~$C$ de $E\smallsetminus\overline{E}_0$
(voir proposition~\ref{prop:class_composantes}) est incluse dans deux
rubans maximaux, dont la réunion sera appelée {\em croix} associée
à~$C$. De plus, on rappelle que~$C$ admet des réflexions locales {\em
  génériques} (proposition~\ref{prop:refl_loc_generique}).  Certains
carrés particuliers, dit {\em symétriques}, peuvent également admettre une
réflexion exceptionnelle (ou non générique) par rapport à une feuille du champ
$K$.
Les parties à courbure constante de~$E$ possèdent également des 
réflexions locales, non génériques.

\begin{defi} 
\label{defi:sf_refl_unif} Soit $(E,K)$ une surface lorentzienne saturée et
simplement connexe. On dit que $(E,K)$ est 
\begin{enumerate}
\item {\em réflexive} si toute réflexion locale générique de $(E,K)$ s'étend
isométriquement à la croix correspondante,
\item {\em R-homogène} si le groupe $\Isom(E)$ agit transitivement sur
  l'ensemble des rubans maximaux de~$E$.
\end{enumerate}
\end{defi}

Ces deux propriétés sont invariantes par isométrie usuelle.  Dans le cas très
particulier où~$K$ est de lumière, $(E,K)$ est réduite à un ruban plat :
la surface $(E,K)$ est alors réflexive car 
il n'y a aucune condition à satisfaire ;
elle est évidemment R-homogène.  Les surfaces $(E^u_f,K^u)$ du
\S\ref{subs:const_sf_max} sont par construction de classe $[f]$ et réflexives
(voir lemme~\ref{lemm:Euf_reflexive}).  Celles de
l'exemple~\ref{exem:sf_unif_non_refl} ne  le sont pas. Nous verrons que le fait
d'être à géométrie locale uniforme induit un comportement analytique pour les
surfaces lisses. Cette propri\'et\'e et d'ailleurs automatique dans le cas
analytique.
\begin{prop}
\label{prop:analy_refl}
Soit $(X,K)$ une surface lorentzienne saturée et connexe.
\begin{enumerate}
\item Dans chacun des cas suivants, la surface $(X,K)$ est à géométrie locale
uniforme :
\begin{enumerate}
\item le flot de $K$ est périodique,
\item $X$ est analytique.
\end{enumerate}
\item On suppose que $X$ est simplement connexe et maximale. Alors la surface 
 $(X,K)$ est à géométrie locale uniforme si et
seulement si elle est réflexive.
\end{enumerate}
\end{prop}

\begin{proof} Supposons d'abord que le flot de $K$ est p\'eriodique. En relevant
les cylindres ou les rubans de Möbius feuilletés par $K$ au revêtement universel
$\widetilde{X}$, on voit que l'espace des orbites $\mathcal E_{\widetilde{X}}$
est s\'epar\'e. Cet espace est alors un intervalle réel et la surface $(X,K)$
est clairement \`a g\'eom\'etrie locale uniforme, d'où (1-a).

Pour la suite de la preuve, on peut supposer que $X$ est simplement connexe. On
se donne alors une coordonn\'ee transverse $x$ sur $X$. Soit $\sig$ une
réflexion d'une composante $C$ de $E\smallsetminus\overline{E}_0$. Notons $R$ et
$R'$ les deux rubans maximaux d'intersection $C$. Soient
$g\in\mcC^\infty(I,\R)$ et $g'\in\mcC^\infty(I',\R)$ les fonctions telles que
$g\circ x$ et $g'\circ x$ sont les restrictions de $\langle K,K\rangle$ \`a $R$
et $R'$ respectivement. La réflexion $\sig$ induit la relation $g'(u)=g(u)$ pour
tout $u\in x(C)$. Si $(X,K)$ est suppos\'ee analytique alors~$g$ et~$g'$ le sont
aussi et cette relation est vraie pour tout $u\in I\cap I'$. Les fonctions~$g$
et~$g'$  sont donc des restrictions d'une m\^eme fonction analytique
inextensible $f$. En raisonnant de proche en proche, on voit qu'il en est de
m\^eme pour chaque ruban. Quitte à restreindre $f$, on conclut que $(X,K)$
est de classe $[f]$, d'où (1-b).

Si $(X,K)$ est suppos\'ee maximale et de classe $[f]$, alors $g=f|_I$ et
$g'=f|_{I'}$ sont inextensibles. Par suite $I=I'$ et $\sigma$ s'étend à la croix
associée à $C$. Enfin, comme les r\'eflexions g\'en\'eriques pr\'eservent
clairement les coordonn\'ees transverses sur une croix (voir remarque 
\ref{rema:preservation_indices}), toute surface réflexive
est à géométrie locale uniforme. 
\end{proof}

Nous montrerons plus bas que toute surface simplement connexe $(E,K)$ réflexive
et sans selles à l'infini est R-homogène (voir
théorème~\ref{theo:uniformisation}). Mais il existe des surfaces réflexives non
R-homogènes (par exemple le revêtement universel d'une surface $E^u_f$ privée
d'un point selle, avec $f$ suffisamment générale) et inversement des surfaces
R-homogènes non réflexives.

\begin{exem}[surfaces R-homogènes non réflexives]
\label{exem:sf_unif_non_refl}
Soit $R_f= (\msfR,\msfk)$ le ruban (maximal) associé à une fonction $f \in
\mcC^\infty(\R,\R)$ qui ne s'annule qu'en deux points, disons $-1$ et $1$,
et qui ne change pas de signe. Les deux feuilles de lumière du champ de
Killing délimitent~3 composantes notées simplement~1,~2 et~3, dans l'ordre
défini par le champ de lumière 
$\msfl$ tel que $\langle \msfk,\msfl \rangle = 1$.  Les
composantes de $(\msfR,-\msfk) \simeq R_{f^\vee} $ sont dans l'ordre 3,~2,~1.
On suppose de plus que les composantes externes~1 et~3 ne sont pas
isométriques et que $f$ est symétrique entre les zéros, c'est-à-dire
$f(x)=f(-x)$ pour $x\in [-1,1]$, de sorte que la composante centrale~2
possède une réflexion exceptionnelle.  Cela étant, on considère deux
familles de rubans $(H_i)_{i\in\Z}$ (\og horizontaux\fg) et
$(V_j)_{j\in\Z}$ (\og verticaux\fg), avec $H_{2k}=V_{2k}=(\msfR,\msfk)$ et
$H_{2k+1}=V_{2k+1}=(\msfR,-\msfk)$, assemblés comme sur la
figure~\ref{figu:sf_unif_non_refl}. Les composantes impaires sont
identifiées {\it via} une réflexion générique tandis que les composantes
paires sont identifiées {\it via} {\em la réflexion exceptionnelle}.  La
surface $Y$ ainsi définie est L-complète et porte un champ de Killing
(figure~\ref{figu:sf_unif_non_refl}, axes des réflexions exceptionnelles en
pointillés).  Les réflexions génériques des composantes impaires s'étendent
à $Y$, mais pas celles des composantes paires (c'est la réflexion
exceptionnelle qui s'étend). Le revêtement universel $(E,K)$ est clairement
R-homogène mais pas réflexif.

\begin{figure}[h!]
\labellist
\small\hair 2pt
\pinlabel $2$ at 19 123
\pinlabel $1$ at 51 123
\pinlabel $1$ at 19 91
\pinlabel $2$ at 51 91
\pinlabel $3$ at 83 91
\pinlabel $3$ at 51 59
\pinlabel $2$ at 83 59
\pinlabel $1$ at 115 59
\pinlabel $1$ at 83 27
\pinlabel $2$ at 115 27

\pinlabel $2$ at 211 123
\pinlabel $1$ at 243 123
\pinlabel $4$ at 275 123
\pinlabel $3$ at 307 123
\pinlabel $1$ at 211 91
\pinlabel $2$ at 243 91
\pinlabel $3$ at 275 91
\pinlabel $4$ at 307 91
\pinlabel $4$ at 211 59
\pinlabel $3$ at 243 59
\pinlabel $2$ at 275 59
\pinlabel $1$ at 307 59
\pinlabel $3$ at 211 27
\pinlabel $4$ at 243 27
\pinlabel $1$ at 275 27
\pinlabel $2$ at 307 27
\endlabellist
\begin{center}
\includegraphics[scale=0.9]{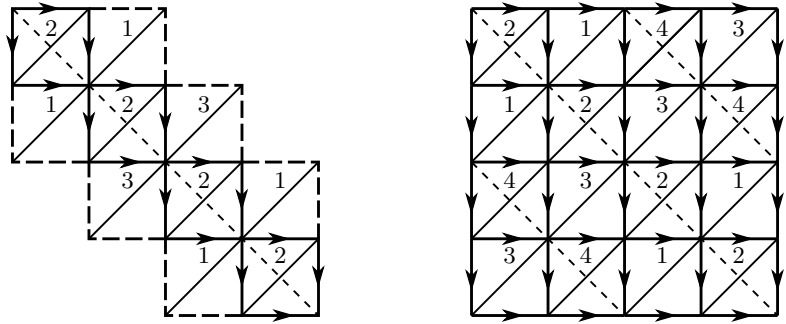}
\caption{Surfaces R-homogènes non réflexives}\label{figu:sf_unif_non_refl}
\end{center}

\end{figure}

Il est facile de construire des variantes possédant des ouverts divisibles.
On part avec un ruban périodique $R$ dont la période de base comprend~4
composantes $1,\ldots,4$. On suppose que les composantes~2 et~4 admettent
une réflexion exceptionnelle, puis on considère deux familles $(H_i)$ et
$(V_j)$ (indexées par $\Z$) de copies de $R$ et $R^\vee$ comme ci-dessus.
Ces familles sont assemblées périodiquement selon la
figure~\ref{figu:sf_unif_non_refl}, en collant les composantes impaires par
des réflexions génériques et les composantes paires par des réflexions
exceptionnelles.  Après passage au revêtement universel on trouve comme
plus haut une surface $(E',K')$ R-homogène non réflexive. Chaque copie de $R$
revêt un tore lorentzien~$T$. L'extension $E'$ de $\widetilde{T}=R$ est
R-homogène puisque les translatés de $\widetilde{T}$ par le
groupe $\Isom(E')$ recouvrent la surface, mais ses réflexions génériques ne
s'étendent pas toutes. Elle n'est donc pas isométrique à l'extension
$E^u_f$ de \S\ref{subs:const_sf_max}, voir lemme~\ref{lemm:Euf_reflexive}. 
\end{exem}

Nous considérons maintenant la surface $E^u_f$ comme espace modèle pour
des structures géométriques. Rappelons qu'une surface est {\em (localement)
  modelée sur $E^u_f$}, ou {\em admet une $E^u_f$-structure}, si elle
possède un atlas à valeur dans $E^u_f$ dont les changements 
de cartes sont des restrictions d'isométries globales de $E^u_f$.

\begin{lemm}
\label{lemm:Euf_reflexive}
Toute réflexion locale générique de la surface
$E^u_f$ définie au~\S\ref{subs:const_sf_max} s'étend à~$E^u_f$. 
Le sous-groupe engendré par ces transformations (appelé sous-groupe
{\em générique}) 
agit transitivement sur les rubans maximaux de $E^u_f$.
En particulier   $E^u_f$ est réflexive et  R-homogène.
\end{lemm}

\begin{proof} 
On peut supposer que $f$ s'annule sans être constante.  Rappelons que
$E^u_f$ est construite comme revêtement universel d'une surface~$Y^s$,
elle-même complétée d'une surface~$Y$ par adjonction de selles. Comme toute
isométrie de~$Y$ s'étend par continuité à $Y^s$, proposition
\ref{prop:ext_isom_demi_selle}, il suffit de prouver la propriété
d'extension pour $Y$.  De plus, il suffit d'étendre une réflexion
particulière pour chaque carré de $Y$, puis de conjuguer par le flot
de~$Y$. On reprend les notations de la page~\pageref{equa:sf_Y}.  Si
$(\al,i_0)\in S\times A$, considérons la réflexion du carré
$C_{\al,i_0}^H=C_{\al,\al-i_0}^V$, notée $\overline{\sig}_\al$, induite par
l'application identique de $C_\al$ sur $C_\al^\vee$.
Les points de $H_i = R$ (resp. de $V_j$) sont notés simplement 
$p_i$, avec $p\in R$
(resp. $q_j$, avec $q\in R$). Soit $\ph_\al$ l'isométrie involutive de 
la surface $X=(\coprod_{i\in A} H_i) \amalg ((\coprod_{j\in A} V_j)$
définie par l'application identique de $H_i$ sur $V_{\al-i}$ et 
de $V_j$ sur $H_{\al-j}$, autrement dit
\begin{equation}
\label{equa:Euf_reflexive}
\left\{
\begin{array}{lll}
\ph_\al(p_i) = p_{\al-i} \in V_{\al-i}  &  \mathrm{si} &
p_i \in H_i~~(i\in A),\\
\ph_\al(q_j) = q_{\al-j} \in H_{\al-j}  &  \mathrm{si} &
 q_j \in V_j~~(j\in A).
\end{array}
\right.
\end{equation}
Cette application $\ph_\al$ induit une isométrie $\widehat{\sig}_\alpha$
qui prolonge $\overline{\sig}_\al$ sur le quotient $Y=X/\mcR$,
voir~\eqref{equa:sf_Y}. En effet, supposons que $p_i\mcR q_j$, c.-à-d.
$\beta = i+j \in S$ et $q=\sigma_\beta(p)$ ; on a alors $(\al-j) + (\al-i)=
\beta$ (le groupe $A$ étant de 2-torsion) et $p = \sigma_\beta(q)$ puisque
$\sig_\be$ est involutive, d'où $\ph_\al(q_j) \mcR \ph_\al(p_i)$.

\par
 Les extrémités d'une chaîne $(R_i)_{1\leq i\leq N}$ de rubans maximaux
de $E^u_f$ ($R_i\cap R_{i+1}\neq \emptyset$, $1\leq i \leq N-1$) sont
toujours échangées par un produit de réflexions génériques, d'où
la deuxième assertion.
\end{proof}

\begin{rema}[invariance des indices par transformation générique]
\label{rema:preservation_indices}
Fixons un système de coordonnées adaptées $(x,y)\in I\times \R$ pour le ruban
$R_f$, donc pour les $H_i$ et les $V_j$. Par construction de $E^u_f$, voir
notamment~\eqref{equa:sf_Y}, les coordonnées $x$ des rubans se recollent en une
fonction encore notée $x\in \mcC^\infty(E^u,\R)$ ; les carrés des rubans sont
ainsi indexés de façon cohérente par un ensemble fixé $S$. Si $\nu^u$ est une
forme volume lorentzienne sur $E^u_f$, on~a $dx = i_{K^u}\nu^u$, voir la
remarque~\ref{rema:esp_feuilles_etale}. 
Une réflexion générique préserve $dx$ et fixe une feuille de $K$,
elle préserve donc la fonction $x$; en particulier elle préserve les
indices des carrés.
\end{rema}

\begin{rema}
\label{rema:refl_commutent}
La composée $\widehat{\sig}_\alpha \widehat{\sig}_\beta$ ($\al,\be\in
S$) de deux réflexions de $Y^s$ induites par \eqref{equa:Euf_reflexive}
préserve l'ensemble des rubans horizontaux (resp. verticaux) de $Y^s$ en
décalant leurs indices de $\al-\be$. Comme $A$ est de 2-torsion, on a que
$\widehat{\sig}_\alpha$ et $\widehat{\sig}_\beta$ commutent
(c.-à-d. $\widehat{\sig}_\alpha \widehat{\sig}_\beta$ est
d'ordre~2). Le groupe d'isométrie de $Y^s$ contient donc un sous-groupe
isomorphe à $(\Z/2\Z)^{(S)}$.
\end{rema}

Voici encore quelques propriétés préliminaires utiles pour l'étude des
surfaces réflexives. Soit $(X,K)$ une surface lorentzienne saturée.  Nous
appellerons {\em cycle de rubans} toute suite finie $(R_i)_{1\leq i\leq
  N+1}$ de rubans de $(X,K)$ avec $N\geq 2$, $R_i\cap R_{i+1}\neq
\emptyset$ ($1\leq i \leq N$), $R_{N+1}=R_1$ et $R_i \neq R_j$ (comme
ouverts de $X$) si $|i-j|=1$. Un cycle de rubans $(R_i)_{1\leq i \leq N+1}$
sera dit {\em simple} s'il existe une courbe polygonale fermée simple
constituée de segments géodésiques de lumière $\ell_1, \ldots,\ell_N$,
transverses au flot avec $\ell_i \subset R_i$ ($i=1,\ldots,N)$.

\begin{lemm}
\label{lemm:cycl_rubans}
Soit $(E,K)$ une surface lorentzienne saturée et simplement connexe. Alors
tout cycle simple de rubans de $(E,K)$ est constitué de 4 rubans autour
d'un point selle.
\end{lemm}

\begin{proof}
Soit $(R_i)_{1\leq i\leq N+1}$ un cycle simple de rubans et soit $\ell =
\ell_1 \ldots \ell_N$ une courbe polygonale fermée simple comme ci-dessus.
Par le théorème de Jordan, $\ell$ borde un disque~$D$. 
Toute orbite de lumière de $K$ qui coupe $D$ s'accumule sur un point selle,
intérieur à~$D$. 
Si elle est non
vide, l'intersection de $D$ avec l'adhérence d'une bande est donc
homéomorphe à un disque dont le bord est composé de quatre segments
géodésiques de lumière: deux transverses à $K$ et deux tangents à
$K$ s'intersectant en un point selle.  Par conséquent, $D$ est contenu
dans une selle.
\end{proof}

\begin{lemm}[relèvement des géodésiques de lumière brisées]
\label{lemm:relevement_chemins}
Soient $X$ et $Y$ des surfaces lorentziennes connexes et soit $\mcD:X\to Y$ une
isométrie locale. 
\begin{enumerate}
 \item Si toute  géodésique de lumière brisée incluse dans $\mcD(X)$ se relève
par~$\mcD$, alors $\mcD$ induit un revêtement de $X$ sur $\mcD(X)$.
 \item Si toute  géodésique de lumière brisée de $Y$ 
se relève
par~$\mcD$, alors $\mcD(X)= Y$ et $\mcD$ est un revêtement.
\end{enumerate}
\end{lemm}

\begin{proof} Pour établir (1), il suffit de vérifier que tout chemin
$c:[0,1]\to \mcD(X)$ se relève par $\mcD$. On peut approximer $c$ (au sens
$\mcC^0$) dans $\mcD(X)$ par une suite $(\ga_n)_n$ de géodésiques de lumière
brisées avec $\ga_n(0)=c(0)$ et $\ga_n(1)= c(1)$.  Comme $\mcD$ est un
difféomorphisme local, les relevés $(\tilde{\ga}_n)$ des $(\ga_n)$ convergent au
sens $\mcC^0$ vers un chemin $\tilde{c}$ qui relève $c$. 
\par 
Fixons $p_0\in X$. Par connexité, il existe pour tout $q\in Y$ une géodésique de
 lumière brisée $\ga:[0,1]\to Y$ telle que $\ga(0)= \mcD(p_0)$ et $\ga(1)=q$.
Sous l'hypothèse de (2), on peut relever $\ga$ en $\tilde{\ga}$, en particulier
$q=\mcD(\tilde{\ga}(1))$ ; par suite $\mcD$ est surjective. Pour conclure, il
suffit d'appliquer (1).
\end{proof}

\begin{lemm}%
\label{lemm:refl_modelee}
Si  $(E,K)$ est  une surface lorentzienne saturée, simplement connexe et de
classe $[f]$, alors la surface $E$ admet  une $E^u_f$-structure.
\end{lemm}

\begin{proof} 
On peut écarter le cas évident où $(E,K)$ est réduite à un ruban. Soit
$\mcR$ l'ensemble des rubans maximaux de $(E,K)$. Pour tout $R\in \mcR$,
notons $U_R$ la réunion de $R$ avec toutes les selles maximales qui
rencontrent $R$ (et de même dans toute surface) ; nous allons construire
des cartes définies sur les $U_R$ à valeurs dans la surface $E^u_f$. Fixons
un ruban maximal $R_0$ de $E^u_f$ ainsi qu'une coordonnée transverse $x$
(resp. $x^u$) sur $(E,K)$ (resp.  sur $E^u_f$) induisant $f$.  Ce choix
permet de définir, {\it via} les coordonnées adaptées, un plongement
isom\'etrique $\varphi_R$ de $(R,K)$ dans $(R_0,K^u)$, unique modulo action
du flot de $K^u$ et vérifiant $x^u\circ \varphi_R=x$ sur $R$.  De plus, si
$R'\in \mcR$ est tel que $R\cap R'\neq \emptyset$ et $R\neq R'$ alors
$\varphi_R$ et $\varphi_{R'}$ diffèrent par une réflexion générique.
Les selles de $(E,K)$ ne sont pas forcément symétriques mais elles
contiennent toujours une selle symétrique. C'est suffisant pour pouvoir
déduire de la proposition~\ref{prop:ext_isom_demi_selle} que le plongement
$\ph_R$ se prolonge de façon unique en un plongement isom\'etrique $\psi_R$
de $U_R$ dans $U_{R_0}\subset E^u_f$.  D'après la discussion ci-dessus, les
changements de cartes de l'atlas $\{(U_R,\psi_R); R\in \mcR\}$ sont des
réflexions locales génériques ou des éléments du flot.  Toutes ces
transformations s'étendent à $E^u_f$ (lemme~\ref{lemm:Euf_reflexive}).
\end{proof}

\begin{rema}
\label{rema:developpante}
Soit $\Isom^\pm(E^u_f,K^u)$ le groupe des isométries de $E^u$ qui conserve
le champ $K^u$ au signe près (ce groupe coïncide avec $\Isom(E^u)$ sauf en
courbure constante). L'atlas ci-dessus définit une
$(\Isom^\pm(E^u_f,K^u),E^u_f)$-structure. Quitte à composer par une
réflexion du modèle, on voit que cette structure possède toujours une
développante (voir par exemple \cite[p.176]{Goldman} pour cette notion) qui
est une isométrie locale de $(\widetilde{X},\widetilde{K})$ dans
$(E^u_f,K^u)$).
\end{rema}

\begin{lemm}
\label{lemm:ext_isom_rub}
 Soit $E^u_f$ la surface associée à une fonction $f$ 
 (voir~\S\ref{subs:const_sf_max}). Si $R_0$ est  un
ruban maximal de $E^u_f$, alors toute isométrie (au sens usuel) de 
$R_0$ est la restriction d'une isométrie de $E^u_f$.
\end{lemm}

\begin{proof} On utilise les notations et les arguments développés dans la
preuve du lemme~\ref{lemm:refl_modelee}. Soit $\ph\in \Isom(R_0)$ et soit
$U_0=U_{R_0}$. D'après la proposition~\ref{prop:ext_isom_demi_selle} toute
isométrie entre demi-selles s'étend en une isométrie entre selles symétriques,
par conséquent~$\ph$ s'étend en une isométrie $ \ph_0$ de $U_0$. On choisit
ensuite comme en~\ref{lemm:refl_modelee} un atlas $\{(U_R,\psi_R); R\in \mcR\}$
sur $E=E^u_f$, à valeurs dans $U_0$ et avec $\psi_{U_0}=\id_{|U_0}$. La
composition de chaque carte par $ \ph_0$ définit un nouvel atlas de
$E^u_f$-structure sur $E=E^u_f$. En effet, les changements de cartes
$\psi_{R'} \circ \smash{\psi_R^{-1}}$ 
sont des restrictions d'éléments du flot ou des
réflexions locales génériques ; même si $\ph_0$ n'est {\it a priori} que locale,
leur conjugués par $\ph_0$ sont de même nature, donc 
se prolongent à $E^u_f$.
Comme $E=E^u_f$ est L-complète,
la développante de cette structure à partir de la carte $\ph_0 : U_0\to U_0$
satisfait l'hypothèse de l'assertion (2) du lemme~\ref{lemm:relevement_chemins}
: c'est donc une isométrie de $E^u_f$ qui par construction prolonge~$\ph$.
\end{proof}

\begin{lemm}[extension des isométries locales, cas analytique]
\label{lemm:ext_isom_analytique}
Si $f$ est analytique et inextensible, alors toute isométrie locale de $E^u_f$ s'étend
en une isométrie de $E^u_f$.
\end{lemm}

\begin{proof} 	Soit $\Phi$ une isom\'etrie locale de $E^u_f$ d\'efinie sur un ouvert $U$. On peut supposer que $U$ est satur\'e et, quitte \`a restreindre $\Phi$ et \`a composer par une isom\'etrie de $E^u_f$, que $U$ et $\Phi(U)$ sont contenus dans un ruban $R$. Comme $f$ est analytique et inextensible, toute sym\'etrie locale de $f$ est globale. Par suite $\Phi$ s'\'etend \`a $R$, puis \`a $E^u_f$ d'apr\`es la proposition \ref{lemm:ext_isom_rub}.
\end{proof}

\begin{prop}\label{prop:refl_modelee} Si $(X,K)$ est une surface lorentzienne
connexe, saturée et de classe $[f]$, alors la surface $X$ admet une
$E^u_f$-structure. 
\end{prop}

\begin{proof} Le revêtement universel $\widetilde X$ de $X$ possède 
une $E^u_f$-structure (lemme~\ref{lemm:refl_modelee}) qui permet de le
munir d'une isométrie locale  $\mcD : \smash{(\widetilde{X},\widetilde{K})}
 \to (E^u_f,K^u)$
(remarque~\ref{rema:developpante}). On définit des cartes sur $X$ en composant
des sections locales du revêtement $\smash{(\widetilde{X},\widetilde{K})} \to
(X,K)$ avec $\mcD$. De plus, on suppose que les domaines de cartes sont deux à
deux d'intersection connexe. L'analyse des changements de cartes revient alors
à comparer $\mcD$ et $\mcD\circ \ga$, où $\ga$ est un automorphisme du
revêtement universel. 
Il s'agit de montrer que
ces deux applications diffèrent par une isométrie globale de $E^u_f$.

En utilisant les coordonnées adaptées, on voit que $\mcD$ est injective sur tout
ruban maximal de $\widetilde{X}$.  Soit $R$ un tel ruban et soit $R'=\ga(R)$.
Les images $\mcD(R)$ et $\mcD(R')$ sont incluses dans des rubans maximaux $R_0$
et $R'_0$ de $E^u_f$. Fixons une coordonnée transverse $x^u$ sur $E^u_f$.
Alors $x=x^u\circ \mcD$ est une coordonnée transverse sur 
$\smash{\widetilde{X}}$ et il
existe $\ph\in \Isom (f)$ tel que $\varphi \circ x=x\circ \gamma$ (action sur la
coordonnée transverse, voir proposition~\ref{prop:isom_gen}). Ainsi, on a
$\mcD\circ \ga_{|R} = \Phi\circ \mcD_{|R}$, où $\Phi$ est une isométrie de
$\mcD(R)$ sur $\mcD(R')$ déterminée par l'image d'une géodésique interne à
$\mcD(R)$ et par l'action de $\ph$ sur la coordonnée $x^u$. Comme $\ph \in \Isom
(f)$, l'isométrie $\Phi$ s'étend en une isométrie de $R_0$ sur $R'_0$ ; sachant
que $E^u_f$ est R-homogène (lemme~\ref{lemm:Euf_reflexive}), on conclut grâce
au lemme~\ref{lemm:ext_isom_rub} que $\Phi$ s'étend en une isométrie de
$E^u_f$.
\end{proof}

\subsection{Uniformisation}
\label{subs:uniformisation}
Nous dirons qu'une surface $X$ modelée sur $E^u_f$ {\em est uniformisée par
un ouvert de $E^u_f$} si son revêtement universel est plongé dans le modèle
$E^u_f$ par une développante de la $E^u_f$-structure. La représentation
d'holonomie associée réalise alors $X$ comme quotient d'un ouvert de $E^u_f$
par l'action d'un sous-groupe d'isométries de $E^u_f$. Nous établissons ici
l'uniformisation de deux classes de surfaces : les surfaces $(X,K)$ à
géométrie locale uniforme et sans selles à l'infini d'une part et les
surfaces $(X,K)$ compactes (et plus généralement dont le flot de $K$ est
périodique) d'autre part.

\begin{lemm}[injectivité de la développante]
\label{lemm:inj_dev}
Soit $f$ une fonction lisse et soit $(X,K)$ une surface lorentzienne lisse
connexe et saturée. Dans chacun des cas suivants, toute isométrie locale 
 $\mcD : \smash{(\widetilde{X},\widetilde{K})}\to(E^u_f,K^u)$ est injective :
\begin{enumerate}
 \item $(X,K)$ n'a pas de selles à l'infini,
 \item le flot de $K$ est périodique.
\end{enumerate}
\end{lemm}

\begin{proof} Il suffit d'établir le résultat pour le revêtement universel
$(\widetilde{X},\widetilde{K})$ de 
$(X,K)$. Soit $\mcD : (\widetilde{X},\widetilde{K}) \to (E^u_f,K^u)$ une
isométrie locale.

(1) Grâce aux coordonnées adaptées, on sait que $\mcD$ est injective sur tout
ruban (maximal)~$R$ de $\widetilde{X}$.  Par cons\'equent si~$\mathcal D$ n'est
pas injective il existe deux bandes distinctes de $\widetilde{X}$ dont l'image
par $\mathcal D$ est contenue dans un m\^eme ruban de $E^u_f$.  Il existe donc
une suite finie $(R_i)_{1\leq i\leq N+1}$ de rubans de 
$\smash{\widetilde{X}}$ telle que
$R_1\neq R_{N+1}$ et telle que la suite $(\mathcal D(R_i))_{1\leq i\leq N+1}$
est un cycle simple de rubans.  D'apr\`es le lemme~\ref{lemm:cycl_rubans}, les
$\mathcal D(R_i)$ sont les 4 rubans qui bordent un point selle. Ce point selle
ne peut pas avoir d'antécédent : on en d\'eduit que $\smash{\widetilde{X}}$
 a une \og selle \`a l'infini \fg.

(2) L'hypothèse de périodicité entraîne la séparation de l'espace des orbites du
champ~$\widetilde{K}$ (voir la preuve de la proposition~\ref{prop:analy_refl}),
en particulier deux orbites arbitraires peuvent être reliées par une
transversale lisse. Si $\mcD$ n'était pas injective, il existerait une courbe
fermée lisse transverse à $K^u$. Sur une telle courbe, les coordonnées
transverses de $E^u$ seraient strictement monotones, ce qui est absurde.
\end{proof}

\begin{theo}[uniformisation, surfaces sans selles à l'infini]
\label{theo:uniformisation}
Soit $f$ une fonction lisse. Alors toute surface lorentzienne $(X,K)$ lisse,
connexe,  saturée, de classe $[f]$ et {\em sans selles à l'infini} est
uniformisée par un ouvert de $E^u_f$. Si de plus $X$ est simplement connexe et
réflexive (par exemple maximale), alors $(X,K)$ est isométrique à
$(E^u_f,K^u)$.
\end{theo}

\begin{proof}
On sait que $X$ est modelée sur $E^u_f$ et que l'on peut choisir une
développante $\mcD$ qui respecte les champs (voir remarque
\ref{rema:developpante}).  D'après le lemme~\ref{lemm:inj_dev}, $\mcD$ est
injective, d'où la première assertion.

Si $X$ est simplement connexe et réflexive, alors toute géodésique de lumière
brisée de $E^u_f$ se relève par $\mcD : X\to E^u$.  En effet, d'une part les
géodésiques de lumière internes aux rubans se relèvent car les rubans de
$E^u_f$ et ceux de $X$ sont isométriques ; d'autre part, tous les autres
segments de lumière de $X$ sont inclus dans une géodésique de lumière complète
car $K$ n'a pas de selles à l'infini.  Par suite,
lemme~\ref{lemm:relevement_chemins}-(2), $\mcD$ est surjective, ce qui prouve
la deuxième assertion.
\end{proof}

\begin{coro}[classification des extensions universelles]
\label{coro:class_Eu}
Soient $f$ et $g$  deux fonctions lisses et inextensibles d\'efinies sur des
intervalles r\'eels.
\begin{enumerate} 
\item Les surfaces $(E^u_f,K^u_f)$ et $(E^u_g,K^u_g)$ sont isométriques 
si et seulement si $[f]=[g]$.
\item  Les surfaces $E^u_f$ et $E^u_g$ sont isométriques (au sens usuel) 
si et seulement si : ou bien  $E^u_f$ et $E^u_g$ sont à même courbure constante
($f''=g''$ et $f'''=g'''=0$), ou bien  $[[f]]=[[g]]$.
\end{enumerate} 
\end{coro} 

\begin{proof}
Elle résulte du théorème~\ref{theo:uniformisation} et de la proposition
\ref{prop:refl_modelee}.
\end{proof}

\begin{rema}
\label{rema:Eug_Euf}
Si $(I,f)$ est un prolongement de $(J,g)$, la preuve du
lemme~\ref{lemm:refl_modelee} montre que toute surface $(E,K)$ simplement
connexe et de classe $[g]$ admet une $E^u_f$-structure (l'hypothèse de
simple connexité est ici essentielle).  On déduit du
lemme~\ref{lemm:inj_dev} que $(E^u_g,K_g^u)$ se plonge isométriquement dans
$(E^u_f,K_f^u)$.
\end{rema}

\begin{theo}[uniformisation, flot périodique]
\label{theo:uniformisation_cas_periodique}
Soit $(X,K)$ une surface lorentzienne  connexe. Si le flot de $K$ est
périodique, alors $X$ est uniformisée par un ouvert d'une surface $E^u_f$ pour
une certaine fonction $f$.
\end{theo}

\begin{proof}
Le résultat découle des propositions~\ref{prop:analy_refl} (géométrie locale
uniforme),~\ref{prop:refl_modelee} (existence de $E^u_f$-structure) et du
lemme~\ref{lemm:inj_dev} (injectivité de la développante).
\end{proof}

\begin{theo}\label{theo:uni_tore}
Soit $(T,K)$ un  tore lorentzien muni d'un champ de Killing $K$ non trivial.
Le revêtement universel $(\widetilde{T},\widetilde{K})$ de $(T,K)$ 
admet une extension $(E,K_E)$ simplement connexe, L-complète et réflexive.
Une telle extension est unique à isométrie (usuelle) près.
\end{theo}

\begin{proof} %
On peut supposer que $T$ n'est pas plat (sinon le le
th\'eor\`eme~\ref{theo:uni_tore} est classique). 
Observons que $K$ ne s'annule pas d'après
le théorème de Poincaré-Hopf, car les seuls zéros possibles sont des points
selles. De plus, puisque $T$ n'est pas plat, le champ~$K$ admet une orbite
fermée : il existe $p\in X$ et $t_0\neq 0$ tels que $\Phi_K^{t_0}(p)=p$. Cette
relation implique $d\Phi_K^{t_0}(p)=\Id$, puis  $\Phi_K^{t_0}=\Id$ : le flot de
$K$ est donc périodique. Par suite
(théorème~\ref{theo:uniformisation_cas_periodique} et
remarque~\ref{rema:developpante}), $(\widetilde{T},\widetilde{K})$ se plonge
isométriquement dans $(E^u_f,K^u)$, où $[f]$ est la classe de $(T,K)$. De plus,
le champ $K$ admet une transversale lisse, fermée et simple. En relevant une
telle courbe à $\widetilde{T}$, on voit que $f$ est périodique. Ainsi $E^u_f$
est L-complète, d'où l'existence de l'extension cherchée. Si $(E,K)$ est une
extension simplement connexe, L-complète et réflexive de
$(\widetilde{T},\widetilde{K})$ alors $(E,K)$ est forc\'ement de classe $[f]$
(puisque~$f$ est définie sur $\R$),  donc isom\'etrique \`a $(E^u_f,K^u)$
d'apr\`es le théorème~\ref{theo:uniformisation}.
\end{proof}

\begin{coro}
\label{coro:ext_unvi_analytique}
Soit $T$ un  tore lorentzien analytique muni d'un champ de Killing non
trivial. À isométrie près, le revêtement universel de $T$ 
admet une unique extension analytique,  simplement connexe et 
L-complète.
\end{coro}

\begin{proof}
Le seul point à vérifier est l'unicité. Soit $E$ une extension 
comme dans l'énoncé. Par le théorème d'extension
des champs de Killing de Nomizu (voir~\cite{Nomizu}), le relevé du
champ de~$T$ au revêtement universel 
se prolonge en un champ de Killing analytique $K_E$ sur $E$.
D'après la proposition~\ref{prop:analy_refl}-(2), la surface $(E,K_E)$ est
réflexive. 
On conclut grâce au théorème~\ref{theo:uni_tore}.
\end{proof}

\begin{exem}
 Si $(T,K)$ est un tore de Clifton-Pohl dont le 
champ de Killing $K$ est normalisé par $\max \langle K,K\rangle=1$,
la fonction $f$  associée à $(T,K)$ est  $\sin(2x)$. En effet,
 dans ce cas la courbure est égale à $-2f$ (voir \cite{BM} page 472) et
 donc $f$ vérifie $f''+4f=0$. L'espace $E^u_f$ est isométrique au
 rev\^etement universel de la surface $\widehat \Sigma$ 
apparaissant  dans~\cite{BM}.
\end{exem}

L'absence de  \og selles à l'infini\fg\   dans l'énoncé du 
théorème~\ref{theo:uniformisation} signifie que la surface $(E,K)$ contient 
\og le maximum\fg\ de selles possible. À l'opposé, on trouve la classe
des surfaces  $(E,K)$ dont le champ $K$ ne s'annule pas : on dira
que $(E,K)$ est {\em sans selles} (toutes les selles potentielles
sont à l'infini). Parmi ces surfaces, le revêtement universel 
de $E^u_f$ privé de ses points selles,  noté 
$(F^u_f,K^u_F)$,  joue un rôle particulier.
Bien que n'étant généralement
pas \mbox{L-complète}, cette surface
 satisfait une caractérisation analogue à celle de
$(E^u_f,K^u)$ et permet d'uniformiser les surfaces sans selles de classe $[f]$.

\begin{theo}[uniformisation, surfaces sans selles]
\label{theo:uniformisation_Fu}
Soit $f$ une fonction lisse. Alors toute surface lorentzienne $(X,K)$ lisse,
connexe, saturée, de classe $[f]$ et {\em sans selles} est uniformisée par un
ouvert de $F^u_f$. Si de plus $X$ est simplement connexe et réflexive (par
exemple maximale), alors $(X,K)$ est isométrique à $(F^u_f,K^u_F)$.
\end{theo}

\begin{proof} 
La preuve est analogue à celle du théorème~\ref{theo:uniformisation}.
L'existence d'une $F^u_f$-structure s'obtient par une adaptation immédiate des
arguments développés dans le cas des surfaces sans selles à l'infini.
L'injectivité de la développante découle du même argument que dans le
lemme~\ref{lemm:inj_dev}, en concluant par le fait que $F^u_f$ ne contient pas
de points selles, donc pas de cycles de rubans. Lorsque $X$ est simplement
connexe et réflexive, l'argument de relèvement des géodésiques de lumière
brisées, pour les arcs non internes aux rubans, est légèrement différent.  Comme
les surfaces sont sans selles, les orbites de lumière des champs de Killing
correspondent simultanément, dans $X$ et dans $F^u_f$, à des géodésiques de
lumière complètes ou semi-complètes. On peut donc relever les arcs
correspondants, d'où la surjectivité de la développante.
\end{proof}

\begin{rema} 
Le théorème~\ref{theo:uniformisation_Fu} s'applique notamment aux surfaces
compactes. Il donne alors une variante du théorème~\ref{theo:uni_tore} : {\em le
revêtement universel d'un tore lorentzien $(T,K)$ admet une extension $(F,K_F)$
simplement connexe, maximale, réflexive et sans selles ; une telle extension est
unique à isométrie (usuelle) près.}
\end{rema}

\begin{exem}\label{exem:ext}
On a $E^u_f=F^u_f$ dans les cas suivants : 
$f$ ne s'annule pas ($E^u_f$ est un ruban maximal), 
  $f$ n'admet que des zéros non simples. Si $f$ n'admet
qu'un seul zéro, simple, alors $E^u_f$ est une selle maximale et $F^u_f$
le revêtement universel de cette selle épointée.
\end{exem}

En dehors des cas particuliers de l'exemple précédent, il existe de
nombreuses surfaces maximales simplement
connexes et réflexives de classe fixée $[f]$. 
Elles sont toutes étalées
au-dessus d'un ouvert de $E^u_f$ {\it via} une développante, qui
généralement n'est pas un revêtement sur son image à cause de la présence
de selles à la fois dans la surface et à l'infini.  Ces surfaces sont \og
intermédiaires\fg\ entre $E^u_f$ et $F^u_f$, lesquelles apparaissent
respectivement comme {\em la plus petite} et {\em la plus grande} surface
simplement connexe réflexive de classe $[f]$. 
Il peut exister une infinité non dénombrable d'extensions maximales
intermédiaires.
Supposons par exemple que $f$ admette une suite de zéros $(x_i)$ telle que
$f'(x_i)\neq f'(x_j)\neq 0$ pour tout $i\neq j$. Soit~$R$ un ruban de
$E^u_f$ et soit $p_i\in \overline R$ le point selle correspondant à
$x_i$. Chacune de ces selles (munie du champ) est caractérisée par le réel
$f'(x_i)$. \smash{À} toute partie $P$ de $\N$ on associe $E_P$ le rev\^etement
universel de $E^u_f\smallsetminus \{p_i,i\in P\}$. Ces surfaces sont
clairement maximales et  $E_P$ et $E_{P'}$  sont  isométriques si et
seulement si $P=P'$.

\section{Isométries et quotients}
\label{sect:G-structures}

\subsection{Espace des feuilles. Groupe d'isométrie}

Dans toute la suite, $f\in \mcC^\infty(I,\R)$ désigne une fonction
inextensible. Pour alléger les notations, la surface lorentzienne maximale
$E^u_f$ associée à $f$ sera notée simplement $E^u$. De plus, nous supposerons
que $f^{(3)} \neq 0$ afin d'écarter le cas où $E^u$ est à courbure constante. 
Dans le cas particulier où $f$ ne s'annule pas, nous dirons que $E^u$ (ou  $f$)
est {\em élémentaire}. 

\par

Rappelons que $E^u$ est munie d'un champ de Killing complet $K^u$. L'espace des
orbites non triviales de $K^u$ est une variété riemannienne de dimension~$1$,
notée $\mcE^u$  ; elle n'est pas séparée sauf si $E^u$ est élémentaire. Il
existe une coordonnée transverse $x\in \mcC^\infty(E^u,I)$ vérifiant $\langle
K^u, K^u \rangle = f\circ x$ et induisant une isométrie locale $\overline{x}\in
\mcC^\infty(\mcE^u,I)$, proposition
\ref{prop:struct_transv_killing}-\eqref{prop:esp_feuilles_etale}. La norme du
champ comme fonction sur $\mcE^u$ est donnée par $\overline{f}=f\circ
\overline{x}$. Les points de branchement de $\mcE^u$ correspondent aux bords des
composantes de $\{\overline{f}\neq 0\}$, c'est à dire aux orbites de lumière de
$K^u$ qui bordent les carrés. Métriquement, $\mcE^u$ se décompose en segments à
bord double (associés aux carrés, voir proposition~\ref{prop:class_carres}), en
intervalles à bord double de la forme $[0,\msfm[^\wedge$ ($\msfm\in
]0,+\infty]$, associés aux bandes voisines de l'infini) et en segments
ordinaires (associés aux composantes de l'intérieur de $\{\overline{f} = 0\}$),
raccordés en leurs extrémités.  Les bandes de $E^u$ se projettent dans~$\mcE^u$
sur des ouverts sans points de branchement, maximaux pour cette propriété. De
même que toute variété riemannienne, $\mcE^u$ possède des géodésiques (non
uniquement déterminées par leur vitesse initiale) ; au paramétrage près, il
s'agit des courbes lisses injectives. Les segments (à bord simple ou double)
peuvent s'accumuler. Toutefois, sur une géodésique donnée, l'ensemble des points
de branchement est discret.

\vspace{8pt}
\begin{figure}[h]
\labellist
\small\hair 2pt
\pinlabel $(A)$ at 15 55
\pinlabel $(B)$ at 15 0
\pinlabel $(C)$ at 385 55
\pinlabel $(A1)$ at 90 80
\pinlabel $(A2)$ at 240 80
\endlabellist
\begin{center}
\includegraphics[scale=0.6]{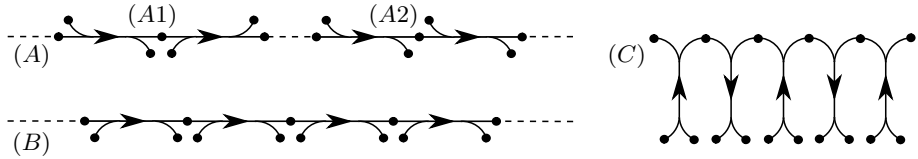}
\caption{Raccordement de carrés dans l'espace des  feuilles}
\label{figu:rac_carres}
\end{center}
\end{figure}

L'espace des feuilles $\mcE^u$ est orienté par $dx$ et 
également transversalement orienté par $K^u$.  Le
raccordement de deux carrés adjacents (zéro isolé de $f$) est indiqué sur
la figure~\ref{figu:rac_carres} : (A1) si~$f$ change de signe, (A2) sinon.
Une géodésique orientée associée à un ruban est caractérisée par la façon
dont elle traverse un carré : si elle entre rive droite, elle doit sortir
rive gauche et vice-versa (gauche et droite étant définis par une
orientation transverse), voir figure~\ref{figu:rac_carres}-(A). D'autres
situations sont illustrées sur cette figure : source/puits à l'infini
(zéros isolés avec changements de signe) en~(B), revêtement infini cyclique
d'une pseudo-selle en~(C).  \par

\begin{figure}[h!]
\begin{center}
\includegraphics[scale=0.6]{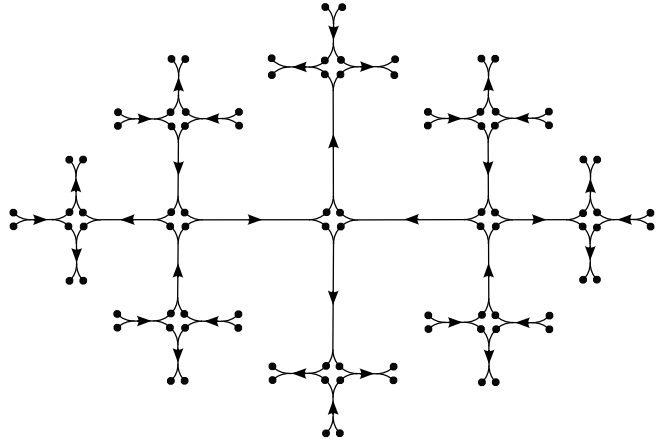}
\caption{Espace des feuilles  de $K^u$ --- cas générique}\label{figu:esp_feu}
\end{center}
\end{figure}

Aux selles près (cycles de 4 points de branchement), l'espace $\mcE^u$ est
un arbre réel (voir figure~\ref{figu:esp_feu}, cas où tous les zéros de
$f$ sont simples). Il est localement quasi-compact si et seulement si les segments à bord
double ne s'accumulent pas.
\par

Soit $\Isom(E^u)$ le groupe d'isométrie de $E^u$. On note $\Isom^0(E^u)$ la
composante neutre (paramétrée par le flot du champ $K^u$), $\Isom^+(E^u)$ (resp.
$\Isom^K(E^u)$) le sous-groupe constitué des isométries {\em directes} (resp. 
préservant le champ) et on pose $\Isom^{+,K}(E^u)= \Isom^+(E^u)\cap
\Isom^K(E^u)$.  On rappelle que $\Isom(f)$ désigne le groupe de symétrie de $f$,
voir~\eqref{equa:isom_f}.  Grâce au lemme~\ref{lemm:Euf_reflexive} et au lemme 
\ref{lemm:ext_isom_rub}, on connaît déjà un certain nombre d'isométries
de~$E^u$. Soit $\Isomgen(E^u)$ le sous-groupe (distingué) de $\Isom(E^u)$
engendré par les réflexions génériques. En considérant les réflexions génériques
associées à un même carré, on voit que $\Isomgen(E^u)$ contient la composante
neutre $\Isom^0(E^u)$.

\begin{prop}[action sur la coordonnée transverse]
\label{prop:isom_gen}
L'action de $\Isom(E^u)$ sur une coordonnée transverse de $E^u$ induisant
$f$ définit une suite exacte scindée 
\begin{equation}
\label{equa:se_isom_gen}
\xymatrix{1 \ar[r]& \Isomgen(E^u) \ar[r] & \Isom(E^u) \ar[r] &
  \Isom(f) \ar[r] & 1},
\end{equation}
le groupe $\Isom(f)$ étant trivial sauf si la fonction $f$ possède des symétries
additionnelles, auquel cas il est isomorphe à $\Z/2\Z,\Z$ ou $D_\infty$.  Le
groupe $\Isomgen(E^u)$ agit simplement et transitivement sur l'ensemble des
géodésiques de lumière internes aux rubans maximaux ; les stabilisateurs de
cette action définissent des sections de \eqref{equa:se_isom_gen}. De plus, si
$R$ est un ruban maximal de $E^u$, le groupe $\Isomgen(E^u)$ est engendré par
$\Isom^0(E^u)$ et par la donnée d'une réflexion par carré de~$R$.
\end{prop}

\begin{proof}
Soit $\nu^u$ la forme volume lorentzienne de $E^u$ vérifiant $dx=i_{K^u}\nu^u$.
Toute isométrie $\ga\in\Isom(E^u)$ préserve $\nu^u$ et $K^u$ au signe près (on
rappelle que la courbure n'est pas constante), donc induit un élément
$\overline{\ga}\in \Isom(I,dx^2)$ tel que $x\circ \ga = \overline{\ga}\circ x$.
Comme la norme de $K^u$ est invariante par isométrie, on a $f\circ
\overline{\ga}=f$ ; ainsi $\overline{\ga}\in\Isom(f)$ et l'action sur la
coordonnée transverse $x$ définit un morphisme $\theta :\Isom(E^u)\to
\Isom(f)$. Le noyau de $\theta$ contient $\Isomgen(E^u)$ car les réflexions
génériques préservent  $dx$ et fixent une feuille de $K^u$, donc préservent $x$.
Soit~$R$ un ruban maximal et soit $\ell$ une géodésique de lumière interne à
$R$. Le groupe $\Isomgen(E^u)$ agit transitivement sur les géodésiques de
lumière internes aux rubans maximaux (action des réflexions génériques sur les
rubans, puis action de $\Isom^0(E^u)$) ; comme $\Isomgen(E^u) \subset \ker
\theta$, cette action est libre. Le stabilisateur de $\ell$ dans $\Isom(R)$ est
isomorphe à $\Isom(f)$, proposition~\ref{prop:class_isom_rubans}. Puisque toute
isométrie de $R$ s'étend à $E^u$ (lemme~\ref{lemm:ext_isom_rub}), le
stabilisateur $H= \Isom(E^u)_\ell$ est isomorphe à $\Isom(f)$. Par conséquent 
on a une décomposition en produit semi-direct \mbox{$\Isom(E^u) = \Isomgen(E^u)H$}:
la suite \eqref{equa:se_isom_gen} est exacte et scindée par $H$.

\par 
Soient $R_1$ et $R_2$ deux rubans maximaux s'intersectant en un carré et
soit $\sigma_0$ une réflexion générique associée à ce carré. Il est clair
que $\sigma_0$ conjugue les réflexions génériques dont l'axe coupe $R_1$
avec les réflexions génériques dont l'axe coupe $R_2$. Par conséquent le
groupe engendré par la donnée d'une réflexion générique dans chaque carré
de $R$ contient au moins une réflexion générique par carré de $E^u$.
\end{proof}

\par L'action isométrique de $\Isom(E^u)$ sur l'espace des feuilles
$\mcE^u$ définit un morphisme $\pi$ de $\Isom(E^u)$ dans le groupe
d'isométrie $\Isom(\mcE^u)$. Notons $\Pi$, $\Pi_\gen$ et $\Pi^{+,K}$
les projections par $\pi$ des groupes $\Isom(E^u)$, $\Isomgen(E^u)$ et
$\Isom^{+,K}(E^u)$.  Soit $H\simeq \Isom(f)$ un sous-groupe de $\Isom(E^u)$
comme dans la preuve de la proposition~\ref{prop:isom_gen}. Le groupe $\Pi$
est alors engendré par le sous-groupe distingué  $\Pi_\gen$ 
et par 
$\pi(H)\simeq \Isom(f)$. On a donc une suite exacte scindée
\begin{equation}
\label{equa:se_pi}
\xymatrix{1 \ar[r]& \Pi_\gen  \ar[r] & \Pi  \ar[r] &
  \Isom(f)  \ar[r] & 1}.
\end{equation}

\begin{rema}
Si $E^u$ n'est pas élémentaire, le groupe $\Isom(\mcE^u)$ est beaucoup plus gros
que~$\Pi$ car il contient des éléments qui échangent les deux branches associées
à un couple arbitraire de points non séparés.
\end{rema}

\begin{prop}
\label{prop:isom_Eu_elem} 
Soit $E^u$ élémentaire à courbure non constante.
\begin{enumerate}
 \item\label{ass:kerpi} On a $\ker \pi= \Isomgen(E^u)$ et $\Pi$ est
   isomorphe à $\Isom(f)$.
 \item Le quotient $\Isom(E^u)/\Isom^0(E^u)$ est isomorphe à $\Z/2\Z \times
   \Isom(f)$
 et la projection de $\Isom(E^u)$ sur  $\Isom(E^u)/\Isom^0(E^u)$ est
 scindée.
\end{enumerate}
\end{prop}

\begin{proof}
Ici $E^u$ est un ruban associé à une fonction $f$ qui ne s'annule pas. Les
transformations génériques agissent trivialement sur les feuilles de $K^u$, d'où
l'assertion ({\ref{ass:kerpi}}). D'après la
proposition~\ref{prop:class_isom_rubans}-(2), le groupe $\Isom(E^u)$ est
isomorphe à un produit semi-direct $\Isomgen(E^u)\rtimes\Isom(f)$, où
$\Isomgen(E^u)$ est engendré par $\Isom^0(E^u)$ et une réflexion générique
$\sig$ globale. Les symétries éventuelles de $f$ (relevées par une section de
$\Isom(E^u) \to \Isom(f)$) commutent avec $\sig$ modulo $\Isom^0(E^u)$, d'où
$\Isom(E^u)/\Isom^0(E^u)\simeq \Z/2\Z\times \Isom(f)$. Enfin, $\Isom^0(E^u)$
agit simplement et transitivement sur les feuilles du feuilletage orthogonal au
champ $K^u$. Par suite, le stabilisateur d'une de ces feuilles définit une
section de la projection modulo $\Isom^0(E^u)$.
\end{proof}

\begin{theo}[action sur l'espace des feuilles]
\label{theo:groupe_isom_Eu} 
Soit $E^u$ non élémentaire et à courbure non constante.
Alors on a $\ker \pi= \Isom^0(E^u)$ et la suite exacte 
\begin{equation}
\label{equa:pi0_isom}
\xymatrix{0 \ar[r]& \Isom^0(E^u) \ar[r] & \Isom(E^u) \ar[r]^{~\pi} &
  \Pi \ar[r] & 1}
\end{equation}
est scindée. De plus,   $\Isom^{+,K}(E^u)$ est isomorphe au produit direct
$ \Isom^0(E^u)\times \Pi^{+,K}$.
\end{theo}

\begin{proof} Fixons un ruban maximal $R$
de $E^u$. On peut choisir un ensemble de réflexions génériques 
$\Sigma=\{\sig_\al; \al\in S\}$, bijectivement associé aux carrés de $R$, de sorte
que deux réflexions contiguës aient des axes orthogonaux (procéder par
récurrence dans chaque composante du graphe de contiguïté
$\mcG_f$). Notons $G_\gen$ le sous-groupe engendré par
$\Sigma$ et montrons que $ G_\gen\cap\Isom^0(E^u)$ est trivial. 
Soit $\ga=\sig_{\al_k}\ldots\sig_{\al_1} \in  G_\gen\cap\Isom^0(E^u)$
avec $\al_i\neq \al_{i+1}$ ($1\leq i \leq k-1$) et $k$ minimal.
Posons $R_0=R$, $R_i=\sig_{\al_i} (R_{i-1})$ ($1\leq i \leq k$). Le cycle 
de rubans $(R_i)_{i=0}^k$ est simple par minimalité de $k$.  
D'après le 
lemme~\ref{lemm:cycl_rubans}, on a $k=4$ et $\ga=(\sig_\al\sig_\be)^2$
avec $\sig_\al$ et $\sig_\be$ contiguës (donc d'axes orthogonaux),
d'où $\ga = 1$. 
D'après la fin de la preuve de la proposition~\ref{prop:isom_gen}, tout
carré de $E^u$ est stabilisé par une réflexion appartenant à
$G_\gen$, unique par ce qui précède.  Ainsi $\pi$ induit
un isomorphisme de $G_\gen$ sur $\Pi_\gen$.

\par Soit $H\simeq \Isom(f)$ l'image d'une section de
\eqref{equa:se_isom_gen}, $H$ stabilisant le ruban $R$ (preuve de la
proposition~\ref{prop:isom_gen}). Notons $G$ le sous-groupe de $\Isom(E^u)$
engendré par $G_\gen$ et $H$, de sorte que $\pi(G)=\Pi$. La restriction
$\pi_{|G}$ sera injective si et seulement si $H$ normalise $G_\gen$, ce qui équivaut à
dire que $H$ stabilise la réunion $\mcA_S$ des axes des réflexions
$(\sig_\al)_{\al\in S}$.  On rappelle que deux axes contigus  se
coupent orthogonalement en un point selle adhérent au ruban~$R$.  Si $H$
est trivial, alors $G=G_\gen$ définit une section de~$\pi$. Supposons que
$H$ n'est pas trivial ; nous allons modifier $H$ et $\Sigma$ afin que
$H(\mcA_S)=\mcA_S$.  Suivant les cas, $H$ est engendré par une involution,
un élément d'ordre infini, ou deux involutions, les involutions ayant un
seul point fixe. Noter qu'en composant indépendamment ces générateurs par
des éléments de $\Isom^0(E^u)$, on obtient encore un relevé de $\Isom(f)$
dans le stabilisateur de $R$.  Le groupe $\Isom(f)$ agit sur l'ensemble
$\pi_0(\mcG_f)$ des composantes connexes du graphe de
contiguïté~$\mathcal G_f$. À toute composante $T\in \pi_0(\mcG_f)$
correspond une composante $\mcA_T$ de $\mcA_S$, formée par la réunion des
axes des réflexions $(\sig_s)$ associées aux sommets de $T$.  On se donne
une tranche $\mathcal T$ de l'action de $\Isom(f)$ sur $\pi_0(\mcG_f)$,
c.-à-d. une composante par orbite, puis on note $S_0\subset S$ l'ensemble
des sommets des éléments de~$\mathcal T$ et on pose $\Sigma_0 =\{\sig_\al;
\al\in S_0\}$.  Si l'action est libre, il suffit de propager les éléments
de $\Sigma_0$ à toutes les bandes de $R$ en conjuguant par~$H$.  S'il
existe un stabilisateur infini, le graphe $\mcG_f$ est alors connexe ainsi
que l'ensemble $\mcA_S$, et $H$ est isomorphe à~$\Z$ ou à $D_\infty$. On
modifie des générateurs de $H$ comme indiqué plus haut pour avoir
$H(\mcA_S)=\mcA_S$.  Reste le cas où tous les stabilisateurs non triviaux
sont d'ordre 2. Ils se répartissent alors suivant une orbite (si
$H\simeq\Z/2\Z$) ou deux orbites (si $H\simeq D_\infty$), et les
involutions fixant les éléments $T \in \mathcal T$ associés engendrent
$H$. On modifie ces involutions par des éléments de $\Isom^0(E^u)$ pour
qu'elles fixent $\mcA_T$, puis on conclut comme dans le cas où l'action est
libre en propageant $\Sigma_0$ par les conjugaisons de $H$. Enfin, les
éléments de $\Isom^{+,K}(E^u)$ commutent avec $\Isom^0(E^u)$ (ils
préservent le champ), d'où la dernière assertion.
\end{proof}

\begin{rema}
Les symétries centrales autour des points selles
seront appelées {\em éléments elliptiques génériques}.  Le sous-groupe
$\Isomsel(E^u)$ qu'elles engendrent est distingué dans $\Isom(E^u)$, inclus
dans $\Isomgen(E^u)$ et trivial si $E^u$ n'admet aucune selle ; à
l'inverse, si $f$ n'a que des zéros simples, $\Isomsel(E^u)$ est d'indice 2
dans l'image de toute section au-dessus de $\Pi_\gen$.
\end{rema}

\subsection{Sous-groupes discontinus}
\label{subs:sous_gr_discontinus} 

Nous supposons ici que $E^u$ n'est ni à courbure constante, ni élémentaire.
Nous poursuivons l'étude dynamique de $\Isom(E^u)$ en négligeant, dans un
premier temps, l'action du flot de $K^u$.  Soit $G$ l'image d'une section
de \eqref{equa:pi0_isom}, fixée dans tout
le~\S\ref{subs:sous_gr_discontinus}. Pour tout sous-groupe $J$ de~$G$, on
pose 
\begin{equation}
\label{equa:gen_plus_K}
J_\gen= J\cap \Isom_\gen(E^u), ~ J^+=J \cap \Isom^+(E^u), ~
J^K=J \cap \Isom^K(E^u)~
\mathrm{et}~ J^{+,K}=J^+\cap J^K.
\end{equation} 
L'objectif de cette partie est de
mettre en évidence l'existence de quotients lisses et séparés de $E^u$,
associés à certains sous-groupes de $G^K$. 
Soit $R$ un ruban maximal de $E^u$, fixé dans toute la suite.
Soit
$\mcG_f=(S,\mfE)$ le graphe de contiguïté de~$f$ ; rappelons que $S$ est
l'ensemble des composantes connexes de $\{f\neq 0\}$ et que 
$\mfE$ est l'ensemble des paires $(\al,\be)\in S^2$ séparées par un zéro simple
de $f$, voir page~\pageref{graphe_contiguite}. 
On note comme plus haut $\Sig=\{\sig_\al
; \al\in S\}$ l'ensemble des réflexions de $G_\gen$ dont l'axe coupe $R$ et
on rappelle que $\Sig$ engendre $G_\gen$. Commençons par préciser la
structure algébrique de $G_\gen$.

\begin{prop}
\label{prop:G_gen_Coxeter} L'ensemble générateur 
$\Sig=\{\sig_\al ; \al\in S\}$ et les relations $\sig_\al^2=1$ ($\al\in
S$), $(\sig_\al\sig_\be)^2 = 1$
($\{\al,\be\} \in \mfE$) forment une présentation de $G_\gen$. Autrement
dit, $(G_\gen,\Sig)$ est un système de Coxeter.
\end{prop}

\begin{proof}
Si $\{\al,\be\} \in \mfE$, la relation $(\sig_\al\sig_\be)^2 = 1$ est
évidemment vérifiée puisque les axes de $\sig_\al$ et $\sig_\be$ sont
orthogonaux.  Soit $w$ une relation réduite entre des éléments de
$\Sig$. En utilisant des conjugués des $\sig_\al^{2k}$ ($\al\in S,k\in \Z$),
on se ramène au cas où $w= \sig_1\sig_2\ldots\sig_N$ avec $\sig_i \in \Sig$
$(i=1,\ldots, N)$ et $\sig_{i+1}\neq \sig_i$.  Considérons la suite
$R_k=(\prod_{i=1}^k \sig_i) (R)$, $k=0,\ldots,N$. On a $R_N=R_0$ et comme
$R_{k}$ et $R_{k+1}$ sont les images de $R$ et $\sig_{k+1}(R)$ par une même
isométrie, $(R_0,\ldots,R_N)$ forme un cycle de rubans. Il existe une
courbe polygonale fermée $c=\ell_0\ldots\ell_{N-1}$ constituée de segments
de lumière $\ell_i\subset R_i$ ($i=0,\ldots,N-1$) transverses au
flot. L'entier $N$ est nécessairement pair et supérieur à 4. Nous allons
prouver par récurrence sur $N/2\geq 2$ que $w$ est produit de conjugués de
relations associées à $\mfE$.

Rappelons que deux rubans $R'$ et $R''$ qui se coupent sont échangés par
l'unique réflexion de $G_\gen$ associé au carré $R'\cap R''$. De plus, tout
cycle de 4 rubans est nécessairement de la forme
\begin{equation}
\label{equa:cycle_4_rubans}
(R',\sig(R'),\sig\sig'(R'),\sig'(R'))
\esp
\mathrm{avec}\esp(\sig\sig')^2=1,
\end{equation}
pour une unique paire $\{\sig,\sig'\}$ de réflexions de $G_\gen$
(lemme~\ref{lemm:cycl_rubans}). Posons $\sig_i=\sig_{\al_i}$ 
($\al_i\in S$). Si $N=4$, la
courbe polygonale $c$ est
simple et le cycle $(R_i)_{i=0,\ldots,3}$ est de la
forme~\eqref{equa:cycle_4_rubans} ; par suite $\sig_1$ et $\sig_2$
commutent,
$\sig_3= \sig_1$ et $w=(\sig_1\sig_2)^2$ avec $\{\al_1,\al_2\}\in \mfE$. Si
$N\geq 6$, la courbe $c$ n'est
plus simple (lemme~\ref{lemm:cycl_rubans}), mais elle reste immergée : on
peut en extraire une courbe simple $c'$ paramétrée par une restriction de
$c$. Nécessairement (lemme~\ref{lemm:cycl_rubans}, à nouveau), la courbe
$c'$ est associée à~4 rubans successifs $(R_{k-2},\ldots,R_{k+1})$ du cycle
$(R_i)$, de la forme~\eqref{equa:cycle_4_rubans}. En posant
$\ga=\sig_1\ldots \sig_{k-2}$ et $\ga'=\sig_{k+2}\ldots\sig_N$, on a
$R'=\ga(R)$, les réflexions $\sig, \sig\sig'$ et $\sig'$ étant les
conjuguées par $\ga$ de $\sig_{k-1}$, $\sig_{k-1}\sig_k$ et
$\sig_{k-1}\sig_k\sig_{k+1}$ respectivement. D'où $\{\al_{k-1},\al_k\}\in
\mfE$, $\sig_{k+1}=\sig_{k-1}$ et enfin $w=(\ga\sig_{k}^{-1})
(\sig_{k-1}\sig_{k})^2 (\sig_{k}\ga^{-1}) w'$ avec
$w'=\ga\sig_{k}^{-1}\ga'$. Le mot $w'$ est une relation de longueur $N-2$,
ce qui achève la récurrence.
\end{proof}

\begin{rema}
\label{rema:abelianise_Ggen}
Il résulte de la proposition~\ref{prop:G_gen_Coxeter} que la famille 
des degrés modulo~2 par rapport aux générateurs $(\sig_\al)_{\al \in S}$
est bien définie et induit un isomorphisme de  l'abélianisé
$G_\gen^{ab}$ sur $\Z/2\Z^{(S)}$.
\end{rema}

\begin{defi}
\label{defi:type_fini}
Une partie $S'$ de l'ensemble $S$ des composantes connexes de $\{f\neq 0\}$
est {\em localement finie} si les éléments de $S'$ ne s'accumulent pas dans
$\R$. On dira que la fonction~$f$ (ou que la surface $E^u$) est {\em de
  type fini} si $S$ est localement fini.
\end{defi}

On rappelle  que la suite exacte~\eqref{equa:se_isom_gen} 
se restreint en une suite exacte scindée 
\begin{equation}
\label{equa:se_G}
\xymatrix{1 \ar[r]& G_\gen  \ar[r] & G  \ar[r] &
  \Isom(f)  \ar[r] & 1},
\end{equation}
isomorphe à~\eqref{equa:se_pi}. Comme $E^u$ n'est pas élémentaire, le
stabilisateur de $R$ dans $G_\gen$ est trivial,
proposition~\ref{prop:class_isom_rubans}.  Par suite $G_\gen$ agit
simplement et transitivement sur les rubans maximaux de $E^u$ et le
stabilisateur~$H$ de~$R$ dans $G$ définit une section de~\eqref{equa:se_G},
en particulier $H$ est isomorphe à $\Isom(f)$. Soit $S'$ une partie de $S$
invariante par $\Isom(f)$. On pose $\Sig'=\{\sig_\al;\al\in S'\}$ 
et on note  $G'$ %
le sous-groupe %
de $G$ engendré par $\Sig'$ et $H$.

\begin{prop}
\label{prop:action_Gprime}
Soit $E^u$ non élémentaire et soit $S'\subset S$ invariante
par $\Isom(f)$.  
\begin{enumerate}
\item L'action de $G'$ sur $E^u$ est propre et discontinue\footnote{Pour tout
compact $K$ de $E^u$, l'ensemble $\{\ga\in G';\ga(K)\cap K \neq \emptyset \}$
est fini.} si et seulement si $S'$ est  localement finie.
\item L'action de $G^+$ sur $E^u$ est discontinue\footnote{Tout point
de $E^u$ admet un voisinage $U$ tel que 
$\{\ga\in G^+;\ga(U)\cap U \neq \emptyset \}$ soit fini.}.
\end{enumerate}
En particulier quand $f$ est de type fini (par exemple analytique),
l'action de $G$ sur $E^u$ est propre et discontinue.
\end{prop}

\begin{proof} Quitte à raisonner sur un sous-groupe d'indice fini de  $G'$, on
peut supposer que $\Isom(f)$ est trivial ou isomorphe à $\Z$.
Si l'action de $G'$ n'est pas proprement discontinue, alors il existe une
suite convergente $(p_n)$ de points de $E^u$ et une suite injective $(g_n)$
d'éléments de~$G'$ telle que la suite $(g_n(p_n))$ converge.  Comme $S'$
est invariante par $\Isom(f)$, on a $h\Sig' h^{-1}= \Sig'$ pour tout $h\in
H$.  Pour tout $n\in \N$, il existe donc $\gamma_n\in
G'_{\text{gen}}=G'\cap G_{\text{gen}}$ et $h_n\in H$ tels que
$g_n=\gamma_n\,h_n$. Puisque $H$ définit une section de \eqref{equa:se_G},
une telle écriture est unique dans $G=G_\gen H$; par suite le groupe
$G'_\gen$ est engendré par $\Sig'$.  Soit $x$ une coordonnée transverse à
$K^u$ comme dans la proposition~\ref{prop:struct_transv_killing}.  Pour
tout $g\in \Isom(E^u)$, on désigne par $\overline g$ l'élément de
$\Isom(f)\subset \Isom(I,dx^2)$ vérifiant $x\circ g=\overline g \circ
x$. On désigne par $x_n$ la coordonnée transverse de~$p_n$.  Pour tout
$n\in \N$, il existe $k_n\in \Z$ tels que ${\overline h}_n(x_n)={\overline
  g}_n(x_n)=x_n+ k_nT$, où $T$ est soit nul, soit une période minimale de
$f$.  Les suites $(x_n)$ et $({\overline g}_n(x_n))$ étant convergentes, on
voit que la suite $(h_n)$ est forcément stationnaire.  En posant
$q_n=h_n(p_n)$, on obtient une suite convergente $(q_n)$ dans $E^u$ telle
que $(\gamma_n(q_n))$ converge. Comme les points selles forment un ensemble
discret et fermé et sont de stabilisateur fini, on peut supposer que les
suites $(q_n)$ et $(\gamma_n(q_n))$ sont chacune contenues dans un
ruban. Ces rubans étant forcément distincts, il existe un élément $w\in
G_\gen$ et des réflexions génériques $\sigma_{\alpha_n}$ avec $\alpha_n\in
S$ tels que $\gamma_n=w\sigma_{\alpha_0}\sigma_{\alpha_n}w^{-1}$.  Le degré
modulo $2$ de $\gamma_n$ en $\sigma_{\alpha_n}$ est donc impair, voir
remarque~\ref{rema:abelianise_Ggen}, et $\alpha_n\in S'$. Par conséquent
$S'$ n'est pas localement finie. La réciproque est évidente, ce qui prouve
l'assertion (1).

Soit $p$ un point de $E^u$. On veut lui trouver un voisinage $U$ tel que
$\{\ga\in G^+;\ga(U)\cap U \neq \emptyset \}$ soit fini. Si $p$ est un point
selle, on prend pour $U$ la selle maximale qui le contient.  Sinon $p$ possède
un voisinage $U_0$ contenu dans un ruban $R_0$, comme l'action de $G^+$
préserve les feuilletages de lumière, tout $ \ga\in G^+$ tel que $\ga(U_0)\cap
U_0 \neq \emptyset$ appartient au stabilisateur de $R_0$. On trouve donc
facilement un voisinage $U$ convenable.
\end{proof}

\begin{coro}\label{coro:certain}
 Soit  $X$ une surface lorentzienne analytique, simplement connexe et munie d'un
champ de Killing non trivial et complet. Si $X$ n'est pas à courbure constante
alors $\Isom (X)$ s'\'ecrit comme le produit semi-direct de $\Isom^0(X)\simeq
\R$ et d'un sous-groupe discret $G_X$ agissant proprement sur $X$.
\end{coro}
\begin{proof}
D'apr\`es les propositions \ref{prop:analy_refl}, \ref{prop:refl_modelee} et
la remarque~\ref{rema:developpante}, il existe une application analytique $f$ et
une développante $\mcD:X\rightarrow E^u_f$ préservant les champs de Killing.
Soit $R$ un ruban maximal de $X$. Rappelons que $\mcD$ est injective sur les
rubans de $X$. 
 Pour toute isométrie $\Phi\in \Isom(X)$, on a donc une isométrie locale 
 $\mcD(R) \to \mcD\circ \Phi(R)$ définie par $\mcD\circ \Phi\circ \mcD^{-1}$.
Comme~$f$ est analytique, cette isométrie se prolonge en
$\theta(\Phi) \in \Isom(E^u_f)$ (lemme~\ref{lemm:ext_isom_analytique}) et
$\theta(\Phi)$
est l'unique élément $\Psi\in \Isom(E^u_f)$ vérifiant  $\mcD\circ
\Phi=\Psi \circ \mcD$. Par unicité, on obtient un morphisme de groupe 
$\theta: \Isom(X)\to \Isom(E^u_f)$ qui envoie le flot $\Phi^t$ de $X$ sur
le flot $\Psi^t$  de $E^u_f$. Si $G$ est l'image d'une section de
\eqref{equa:pi0_isom}, le sous-groupe
$G_X= \theta^{-1}(G)$ est clairement transverse à $\Isom^0(X)$. 
De plus, si $\Phi\in \Isom(X)$, il existe $t\in \R$ tel que 
$\Psi^t \circ \theta(\Phi)= \theta(\Phi^t \circ \Phi) \in G$. Par suite
$G_X$ scinde $\Isom(X)$.

Comme l'action de $G$ sur $E^u_f$ est propre et discontinue, il suffit pour
établir la dernière assertion de
vérifier que l'action de $\Ga=\ker \theta$ est propre et
discontinue.  Si \mbox{$s_1, s_2 : \Omega \to X$} sont deux sections
locales de $D$ définies sur un ouvert connexe $\Omega$, alors
 ($D$ étant localement injective) on a $s_1=s_2$ ou $s_1(\Omega)\cap
s_2(\Omega)=\emptyset$. En particulier, si $U$ est un ouvert connexe
avec $D_{|U}$ injective, les ouverts $\Phi(U)$ ($\Phi \in \Ga$) sont
mutuellement disjoints. Soit $(x_n,\Phi_n)_n $ une suite de $X\times \Ga$
telle que $(x_n,\Phi_n(x_n))_n$ 
converge vers un point $(x,y)\in X^2$, et soit $U$, $V$ des voisinages ouverts
connexes de $x,y$ avec  $D_{|U}$, $D_{|V}$ injectives et $D(U)=D(V)$. Pour $n$
grand,
$\Phi_n(U)$ coupe $V$ : on a $\Phi_n(U) = V$ d'après ce qui précède et 
la suite $(\Phi_n)$ est stationnaire.
\end{proof}

\begin{rema}
Il existe deux autres familles de sous-groupes de $G_\gen$ agissant proprement
discontinûment sur $E^u$: le sous-groupe normal engendré par $\Sigma'$ et le
sous-groupe dérivé de $G_\gen$. Dans le premier cas la preuve est similaire à
celle ci-dessus. La deuxième affirmation s'obtient en remarquant que  le groupe
fondamental de la surface $Y^s$ construite dans la preuve de la proposition
\ref{prop:sf_Euf}  contient un groupe égal à $[G_\gen,G_\gen]$ modulo
$\Isom^0(E^u)$.
\end{rema}

\begin{rema}\label{rema:cas_bonus}
Il découle de la proposition~\ref{prop:action_Gprime} que si $f$ n'est pas de
type fini, alors le groupe d'isométrie de la  surface $Y^s$ ci-dessus contient
un sous-groupe de $2$-torsion, voir remarque~\ref{rema:refl_commutent}, qui agit
non proprement. Il est de plus facile de perturber la métrique afin que ce
sous-groupe soit en fait tout le groupe d'isométrie. On obtient alors une
surface lorentzienne dont le groupe d'isométrie est de torsion mais dont la
dynamique n'est pas triviale, au sens où son action n'est pas propre.
\end{rema}

Notons $\Pi'$ et ${\Pi'}_\gen$ les projections dans $\Pi\simeq G$ des groupes
$G'$ et $G'_\gen$.  Si $\mcG_f=(S,\mfE)$ désigne le graphe de contiguïté de~$f$,
on pose $\mfE' = \mfE \cap S'\times S'$
et $\mcG' = (S',\mfE')$ (sous-graphe de contiguïté associé à $S'$). Voici
d'abord une description de l'action du groupe $G'$ sur $E^u$. Quand $f$ est de
type fini (par exemple analytique), l'énoncé suivant vaut pour $S'=S$,
c'est-à-dire pour $G$ et $\Pi$.

\begin{prop}
\label{prop:refl_hyp}
Soit $E^u$ non élémentaire et soit $S'$ une partie de $S$ invariante par
$\Isom(f)$ et localement finie. L'action de $G'_\gen$ est différentiablement
conjuguée à celle d'un groupe de réflexions du plan hyperbolique, associé à un
polygone $\mcP_0$ dont les côtés sont en bijection avec~$S'$. En particulier, si
$(x_\al)_{\al\in S'}$ désignent les réflexions associées à $\mcP_0$, le groupe
${\Pi'}_\gen\simeq G'_\gen$ admet la présentation
\begin{equation}
\label{equa:pres_Gprime_gen}
\langle x_\al ~(\al \in S') ~|~ x_\al ^2=1 ~(\al \in S')~;~ (x_\al
x_\be)^2= 1~(\{\al,\be\}\in \mfE') \rangle.
\end{equation}
Quand $\Isom(f)$ n'est pas
trivial, le groupe ${G'}\simeq {\Pi'}$ est également différentiablement
conjugué à un groupe discret d'isométries du plan hyperbolique, engendré
par les réflexions génériques $(x_\al)_{\al\in S'}$ et par les éléments non
génériques suivants, fixant $\mcP_0$ :
\begin{enumerate}
\item\label{ass:1}  si $\Isom(f) \simeq \Z/2\Z$, un élément elliptique 
d'ordre 2  (\theenumi a) ou  une réflexion  (\theenumi b),
\item\label{ass:2} si $\Isom(f) \simeq \Z$,  un élément d'ordre infini, 
direct (hyperbolique ou parabolique) ou indirect,
\item\label{ass:3} si $\Isom(f) \simeq D_\infty$, deux  elliptiques  d'ordre 2
 (\theenumi a), ou un  elliptique  d'ordre 2 et une réflexion  (\theenumi b), 
ou  deux réflexions (\theenumi c),
\end{enumerate}
où les elliptiques (resp. les réflexions) correspondent aux involutions 
de $\Isom(f)$ sans point fixe (resp. avec point fixe) sur $S'$.
\end{prop}

\begin{proof} 
La propreté de l'action de ${G'}$ implique l'existence d'une métrique
riemannienne sur $E^u$ invariante par ${G'}$ (\cite[thm. 4.3.1]{Palais1961}). On
munit $E^u$ de la structure de surface de Riemann définie par des coordonnées
locales isothermes. Il existe alors un biholomorphisme de $E^u$ sur~$X$, $X=\C$
ou $\HH$ (demi-plan de Poincaré), biholomorphisme qui conjugue ${G'}$ en un
groupe $G'_X$ de transformations holomorphes ou anti-holomorphes de $X$.
Supposons que  $X=\C$. Le groupe $G'_X$ doit alors  préserver la  métrique
euclidienne usuelle, sans quoi il admettrait des stabilisateurs infinis. De plus
$X/G'_X$ ne peut être compact (voir le début de la preuve du théorème
\ref{theo:uni_tore}). Par suite, ou bien $G'_X$ est fini et $E^u$ est une selle,
ou bien  $X/G'_X$ est revêtue par une cylindre ; dans les deux cas, l'orbifolde
$X/G'_X$ admet aussi une structure hyperbolique. On peut donc supposer que
$X=\HH$ dans tous les cas. Par suite  ${G'}$ est conjugué par un difféomorphisme
$\ph$ à un sous-groupe discret ${G'}^\ph$ d'isométries de  $\HH$.

\par 
On a vu dans la preuve de la proposition~\ref{prop:action_Gprime} que $G'_\gen
=G'\cap G_\gen$ est engendré par $\Sig'$.  De plus $G'/G'_\gen$ est isomorphe à
$H\simeq \Isom(f)$. Soit $\Om'$ l'ensemble des réflexions $\ga' \sig_\al
\ga'\null^{-1}$ pour $\sig_\al \in \Sig'$ et $\ga'\in G'_\gen$, et soit $\mcA'$
la réunion des axes des éléments de $\Om'$.  L'ensemble~${\mcA'}$ est stable
par~$G'$, ainsi que son complémentaire dans $E^u$, noté ${\mcA'}^c$. D'après la
remarque~\ref{rema:abelianise_Ggen}, deux réflexions distinctes dont les axes
coupent $R$ ne peuvent pas être conjuguées dans $G_\gen$ ; les seuls éléments de
$\Om'$ dont l'axe coupe $R$ sont donc ceux de $\Sig'$. Par suite, les
composantes de ${\mcA'}^c$ ne s'accumulent pas ($S'$ est localement finie) et
tout élément de $\mcA'$ est adhérent à l'une d'elle.  Si $\mcD_0$ et $\mcD_1$
sont des composantes de ${\mcA'}^c$, il existe un chemin $c$ joignant $\mcD_0$
et $\mcD_1$, évitant les points selles et tel que $c([0,1])\cap {\mcA'}$ soit
fini. Le produit des réflexions par rapport aux axes correspondants échange
$\mcD_0$ et $\mcD_1$ ; le groupe $G'_\gen$ agit donc transitivement sur les
composantes de ${\mcA'}^c$. Soit $\Sig''$ l'ensemble des réflexions dont les
axes bordent $\mcD_0$.  Si $\sig\in \Sig''$, les réflexions associées au domaine
voisin $\sig(\mcD_0)$ sont les conjuguées des éléments de $\Sig''$ par~$\sig$ ;
par suite, tout élément de $\Om'$ est produit d'éléments de $\Sig''$ et
finalement $G'_\gen$ est engendré par~$\Sig''$ (ou par les réflexions associées
à n'importe quelle composante de ${\mcA'}^c$).

\par Notons $\mcP_0$ l'adhérence de $\ph(\mcD_0)$ dans $\HH$. Il s'agit d'un
polygone géodésique, dont les sommets ont des angles droits et
correspondent à des points selles.  Le conjugué $\ph G'_\gen\ph^{-1}$
coïncide avec le groupe de réflexions associé à $\mcP_0$ : la présentation
cherchée résulte du théorème de Poincaré, à l'indexation des générateurs
près. De plus, on voit que l'action de $G'_\gen$ sur les composantes de
${\mcA'}^c$ est libre. D'après ce qui précède, toute réflexion $\sig\in
\Sig''$ est conjuguée dans $G'_\gen$ à une unique réflexion $\sig_\al\in
\Sig'$, que nous notons $\xi(\sig)$.  Pour tout $\al\in S'$, chaque
composante de ${\mcA'}^c$ bordée par l'axe de $\sig_\al$ est congruente à
$\mcD_0$ par un élément de $G'_\gen$. L'application $\xi : \Sig''\to \Sig'$
est donc surjective.  Elle est également injective ; en effet si $\ga_i
\sig_i \ga_i^{-1}= \sig_\al\in \Sig'$ ($\ga_i\in G'_\gen$, $i=1,2$), on
peut supposer que $\ga_1(\mcD_0)=\ga_2(\mcD_0)$, d'où $\ga_1=\ga_2$ et
$\sig_1=\sig_2$.  De plus, si les axes de $\sig , \tau \in \Sig''$ se
coupent, les indices de $\xi(\sig)$ et de $\xi(\tau)$ doivent être
contigus, et inversement.  Ainsi, le graphe d'incidence des côtés de $\mcD_0$
(ou de $\mcP_0$) est simplicialement isomorphe au graphe $\mcG'$ et on a
bien la présentation annoncée \eqref{equa:pres_Gprime_gen} pour
$G'_\gen\simeq \Pi'_\gen$.

\par Considérons le stabilisateur $H_0$ de $\mcD_0$ dans $G'$. Comme l'action de
$G'_\gen$ sur les composantes de ${\mcA'}^c$ est libre et transitive, le
stabilisateur $H_0$ est isomorphe à $G'/G'_\gen\simeq \Isom(f)$, supposé non
trivial.  Si $H_0$ est infini cyclique engendré par $h$, alors $h$ n'a pas de
point fixe, c'est le cas ({\ref{ass:2}}) de l'énoncé. Sinon, $\Isom(f)$ est
engendré par des involutions et on trouve les cas ({\ref{ass:1}}) et
({\ref{ass:3}}) de l'énoncé. Toute involution de $\Isom(f)$ se relève dans $H$
(stabilisateur de $R$ dans $G$) en un elliptique non générique $\tau$ centré
dans $R$. En conjuguant par $G'_\gen$, on obtient un elliptique non générique
$\tau_0\in \Isom(E^u)$ dont le centre appartient à $\mcD_0$ si $\tau$ ne fixe
aucune composante $s\in S'$ ou à $\partial \mcD_0$ sinon.  Dans le premier cas
$\tau_0\in H_0$, dans le second c'est une réflexion non générique $\tau_0\sig$
(pour un $\sig\in \Sig''$) qui fixe $\mcD_0$.
\end{proof}

\begin{rema}
D'après la proposition~\ref{prop:G_gen_Coxeter}, on sait que
$(G'_\gen,\Sig')$ est un système de Coxeter, \cite[Thm.~2,
  p.~20]{Bourbaki_LieIV}. On remarquera que la présentation de la
proposition~\ref{prop:refl_hyp} est analogue, mais avec un système
générateur différent.
\end{rema}

\begin{rema}
Si $f$ est de type fini, les domaines fondamentaux explicites donnés plus
bas pour $S'=S$ fournissent, {\it via} le théorème de Poincaré, une
présentation du groupe $\Pi$ dans tous les cas.
\end{rema}

\subsection{Quotients minimaux}
\label{subs:quotients_mini} 
Nous nous intéressons maintenant aux quotients lisses de $E^u$. Par
exemple, quand $f$ est de type fini, les sous-groupes sans torsion $\Gamma$
de $\Pi$ (dont on connaît maintenant la structure algébrique) définissent
des quotients lisses {\it via} une section de~\eqref{equa:pi0_isom}.  Ces
surfaces portent automatiquement un \og feuilletage de Killing
\fg\ (lemme~\ref{lemm:isom_isomK}) et même un champ de Killing quand
$\Gamma$ est inclus dans $\Pi^K$. Nous chercherons $\Ga$ d'indice aussi
petit que possible ; dans cette étude, les morphismes $\carac : \Pi^K \to
\mu_2= \{\pm 1\}$ (caractères) et $\om : \Pi \to \mu_2^2$ vont jouer un
rôle clé. On identifie comme dans la
présentation~\eqref{equa:pres_Gprime_gen} le graphe $\mcG_f$ avec le graphe
d'incidence des côtés de $\mcD_0$ (ou de~$\mcP_0$) ; les réflexions associées à ces côtés
seront notées simplement $\ga$ ($\ga \in S$) dans la suite.  Si $\ker \carac$ est sans torsion, la condition
suivante est évidemment vérifiée:
\begin{equation}
\label{equa:ker_chi_st}
\carac(\al \be)= -1 \esp 
(\{\al,\be\}\in \mfE).
\end{equation}
De même, si $\ker\om$ est sans torsion, les valeurs  $\om(\ga)$ pour $\ga \in
S$ et $\om(\al\be)$ pour $\{\al,\be\}\in \mfE$ sont non  triviales, et 
les valeurs de $\om$ sur 
$\al$, $\be$ et $\al\be$ ($\{\al,\be\}\in \mfE$) doivent être 
mutuellement distinctes, c'est-à-dire
\begin{equation}
\label{equa:ker_om_st}
\left\{
\begin{array}{ll}
\om(\ga) \neq  (1,1) &  (\ga \in S)\\
\om(\al) \neq \om(\be) & (\{\al,\be\}\in \mfE).
\end{array}
\right.
\end{equation}
L'apparition de morphismes $\om : \Pi \to \mu_2^2$ dans notre problème
 provient du phénomène  algébrique suivant.

\begin{lemm}
\label{lemm:ind4_normal}
Soit $\Pi''$ un groupe contenant un sous-groupe $K''$ isomorphe à
$\mu_2^2$ et soit $\Ga''$ un sous-groupe de $\Pi''$. On suppose
que
\begin{itemize}
\item[(i)] $\Ga''$ est d'indice 4  et n'a pas de 2-torsion,
\item[(ii)] $\Ga''$ est engendré par des produits de la forme $x\delta'$ avec
$x\in K''$ et ${\delta'}^2=1$. 
\end{itemize}
Alors $\Ga''$ est distingué dans $\Pi''$ et le quotient 
$\Pi''/\Ga''$ est isomorphe à $\mu_2^2$. De plus, si $\Pi''$ est engendré 
par des involutions, alors (ii) est une conséquence de (i).
\end{lemm}

\begin{proof} 
Comme $\Ga''$ n'a pas de 2-torsion, les classes $x\Ga''$ ($x\in K''$) sont
mutuellement distinctes et on a $\Pi'' = \amalg_{x\in K''} \, x \Ga''$.
Posons $K''=\{1,\al,\be,\al\be\}$.  Soit $\de\in \Pi''$ tel que ${\de}^2=1$
et $\al\de \in \Ga''$. Il existe $(x,\ga)\in K''\times \Ga''$ tel que
$\be(\al\de)\be = x\ga$.  L'éventualité $x=\be$ est exclue car
$\be\notin \Ga''$ ; $x=\al$ contredit (i) ; $x= \al\be$ implique
$\be\de=(\de\be)^{-1}\in \Ga''$ qui est absurde. On a donc $x=1$. De
même, on vérifie que $(\al\be)(\al\de)(\al\be) \in \Ga''$. De plus
$\al(\al\de)\al=(\al\de)^{-1} \in \Ga''$. Finalement tous les conjugués
$\al\de$ appartiennent à $\Ga''$.  Grâce à la condition (ii), il en résulte
que $\Ga''$ est normal.  Enfin, le quotient $\Pi''/\Ga''$ est un groupe de
2-torsion d'ordre 4.

\par Supposons que $\Pi''$ est engendré par des involutions et que $\Ga''$
vérifie (i). Tout élément $\ga\in \Ga''$ s'écrit $\ga =
\de_1\de_2\ldots\de_N$, où les $\de_i$ sont des involutions et $N\geq 2$.
Par (i), il existe $x_N \in K''$, $x_N\neq 1$, tel que $x_N\de_N \in
\Ga''$.  On a $\ga(x_N\de_N)^{-1} = x_N (x_N\de_1 x_N) \ldots
(x_N\de_{N-1} x_N)$. Par récurrence, on voit qu'il existe $N$ involutions
$\sig_1,\ldots,\sig_N \in \Ga''$ et $x_1,\ldots,x_N \in K''$ tels que
$\ga(x_N\sig_N)^{-1}(x_{N-1}\sig_{N-1})^{-1}\ldots (x_1\sig_1)^{-1}\in
\Ga''\cap K''$, d'où le résultat.
\end{proof}

\begin{rema} 
\label{rema:indice4}
Si un groupe $\Ga''$ contient un sous-groupe $K''$ isomorphe à $\mu_2^2$,
alors le début de la preuve précédente montre que tout sous-groupe 
$\Ga''$ sans 2-torsion est d'indice au moins~4.
\end{rema}

Nous verrons plus bas que la topologie des quotients de $E^u$ par des
sous-groupes sans torsion maximaux (si $f$ est de type fini) dépend du
groupe $\Isom(f)$. Dans certains cas, il convient en outre de distinguer
plusieurs types de fonctions $f$.  Pour tout $s\in S$, on note $(-1)^s$ le
signe de $f$ sur la composante $s$.

\begin{defi}
\label{defi:types_al_be}
On suppose que  $\Isom(f)$ est infini. Soit $S'$ une partie de  $S$ localement
finie, invariante par $\Isom(f)$ et munie de l’ordre total induit par $\R$. 
\begin{enumerate}
%\item (types $\al$ et $\be$) Soit $S''$ un intervalle de $S'$ qui est un domaine
%fondamental pour l'action de $\Isom(f)$ sur $S'$. On dit que $S'$ est {\em de
%type $\al$} si $(-1)^{s_1} (-1)^{s_2} =-1$  pour tout couple  de composantes
%successives $(s_1,s_2)\in {S''}\null^2$, {\em de type $\be$} sinon.
\item (types pair et impair)  
On dit que la partie $S'$ est {\em de type pair} si le nombre d'éléments de $S'$
par période minimale de $f$ est pair, {\em de type impair} sinon. 
\item 
(types pairs unilatère et bilatère) Lorsque la partie $S'$ est de type pair, on
dit que $S'$  est {\em unilatère}  si $(-1)^{s_1} (-1)^{s_2} =-1$ pour tout
couple de composantes successives $(s_1,s_2)\in {S'}\null^2$, {\em bilatère}
sinon. 
\end{enumerate}
Quand $f$ est de type fini, on dit que $f$ est {\em de type pair unilatère}
(resp. {\em pair bilatère, impair}) si $S$ est de type pair unilatère (resp. 
pair bilatère, impair).
\end{defi}

On spécifiera le type quand il n'est pas déterminé par l'action de $\Isom(f)$
sur $S'$. Le cas (3a) est de type pair bilatère, le cas (3b) impair et le cas
(3c) de type pair unilatère ou bilatère, noté (3c$^{+u}$) ou (3c$^{+b}$). Le
cas (2) peut \^etre d'un type arbitraire : (2$^{+u}$), (2$^{+b}$) ou  (2$^-$).

\begin{lemm}[description géométrique des types]
\label{lemm:interpretation_types}
On suppose que $\Isom(f)$ est infini.
Soit $G'\subset \Isom(E^u)$ associé à une partie $S'\subset S$ localement finie
et invariante par $\Isom(f)$, et soit $\mcD_0\subset E^u$ un pavé fondamental
pour le groupe \og de réflexions \fg\ $G'_\gen$. On note $h_0 \in G'$ un 
générateur du sous-groupe cyclique d'indice minimal (c.-à-d.~1 ou~2) du
stabilisateur de $\mcD_0$ dans $G'$. 
\begin{enumerate} 
\item\label{ass:kr} Si $S'$ est de type pair, alors le champ de Killing $K$ 
est partout
rentrant (ou partout sortant) sur  $\partial \mcD'_0$ si et seulement si $S'$
est unilatère.
\item
\label{ass:direct} Si~$S'$ est de type pair, alors
$h_0\in {G'}^{+,K}$. 
Si~$S'$ est de type impair, alors $h_0$ est indirect et inverse le champ $K$. 
\end{enumerate}
\end{lemm}

\begin{proof} 
On suppose $S'$ de type pair. Soit $R_1$ un ruban maximal de $E^u$ qui coupe
$\mathcal A'$. D'après la preuve de la proposition~\ref{prop:refl_hyp} (dont on
reprend les notations), $\mcA' \cap R_1$ est l'image de $\mcA'\cap R$ par un
élément de $G'_\gen$, où~$R$ désigne le ruban \og de base \fg. Ainsi le
caractère unilatère ou bilatère peut se lire dans un ruban arbitraire coupant
$\mcA'$. Appelons {\em voisins} deux côtés de $\partial \mcD_0$ inclus dans
un même ruban ; chaque côté de $\mcD_0$ admet deux voisins.  Ces
côtés (et plus généralement les axes des réflexions génériques) ont un signe :
temps ou espace, égal au signe de $-f$. On constate que le champ $K$ est
rentrant (ou sortant) sur deux côtés voisins si et seulement si leurs signes sont
opposés, d'où l'assertion ({\ref{ass:kr}}).

\par 
On peut toujours choisir $\mcD_0$
tel que $\mcD_0\cap R \neq \emptyset$. Soit $h \in H$ tel que $h_0h^{-1}
\in G'_\gen$. On rappelle que $H$ et $H_0$ (stabilisateurs $H$ et $H_0$ 
de $R$ et de $\mcD_0$) définissent des sections de $G'\to G'/G'_\gen$.  Si~$N$
désigne le nombre d'axes d'éléments de $\Sig'$ séparant $\mcD_0\cap R$ et
$h(\mcD_0\cap R)$ dans $R$, on a $h(\mcD_0)= \sig_1\ldots\sig_N (\mcD_0)$ avec
$\sig_i\in \Sig'$ ($i=1, \ldots, N$). De plus, la parité de $N$ vaut celle de
$S'$. Par suite $h = \sig_1\ldots\sig_N h_0^{\pm 1}$.  Comme $h\in {G'}^{+,K}$,
on en déduit l'assertion~(\ref{ass:direct}).
\end{proof}

À partir d'ici nous supposons pour simplifier que {\em $f$ est de type
  fini} et nous prenons $S'=S$. Cependant, on gardera à l'esprit que tous
les énoncés suivants valent pour toute partie $S'\subset S$ localement
finie et invariante, voir la remarque~\ref{rema:quotients_Sprime}.  

\par Soit $\nu$
(resp. $\nu^K$) {\em l'indice minimal dans $\Pi$ d'un sous-groupe sans
  torsion (resp. et inclus dans $\Pi^K$).}  D'après la remarque
\ref{rema:indice4}, deux éventualités se présentent, suivant que
$\Isom(E^u)$ contient un sous-groupe isomorphe à $\mu_2^2$ ou non.  Dans ce
dernier cas, on dira que $\Isom(E^u)$, ou $E^u$, {\em n'a pas de produits
  elliptiques}.  Géométriquement, cela revient à dire que les axes de deux
réflexions (génériques ou non) distinctes de $E^u$ sont toujours disjoints,
ou encore que $E^u$ n'admet pas de selles ni de réflexions non génériques :
cas (0), (1a), (2) et (3a) avec $\mfE=\emptyset$ ; nous commençons par
régler ce cas particulier plus simple.

\begin{prop}[indice minimal sans produits elliptiques]
\label{prop:ind_mini_sans_pe} 
Soit $E^u$ non élémentaire, de type fini et 
 sans produits elliptiques. Alors le groupe
$\Pi^K$ est sans torsion et les indices minimaux valent $\nu=\nu^K = 2$.
Si l'on note $\Sigma\smallsetminus F$ le type topologique de $E^u/G^K$,
avec $\Sigma$ surface fermée et $F$ sous-ensemble discret non vide de $\Sigma$, on a
de plus 
\begin{enumerate}
\item si $\Isom(f) \simeq \{1\}$, $\Z/2\Z$ ou $D_\infty$, l'indice $\nu$
  est réalisé par $\Pi^K$ uniquement et la surface $\Sig$ est
  respectivement la sphère, le plan projectif ou la bouteille de Klein,
\item si $\Isom(f) \simeq \Z$, l'indice
$\nu$ est réalisé  par deux sous-groupes  $\Pi^K$ et $\Gamma$.
Si $f$ est de type pair unilatère (resp. pair bilatère ou impair) alors la
surface $\Sig$
est une sphère (resp. un tore) et la surface associée à $\Gamma$ est un plan 
projectif (resp. une bouteille de Klein). 
\end{enumerate}
\end{prop} 

\begin{proof}
On rappelle que $\Pi$ s'identifie à un groupe discret d'isométries de $\HH$ et
que $\Pi_\gen$ est engendré par les réflexions $(\al)_{\al \in S}$ par rapport
aux côtés d'un polygone~$\mcP_0$. À partir de~$\mcP_0$, on construit facilement
un polygone fondamental $\mcP_\Pi$ pour $\Pi$ à l'aide des générateurs du
stabilisateur de $\mcP_0$ dans $\Pi$ (proposition~\ref{prop:refl_hyp}) : des
elliptiques non génériques $s_i$ ($i=1$ pour (1a), $i=1,2$ pour (3a)) et dans le
cas (2) un élément d'ordre infini~$h_0$, direct ou non (cas pair ou impair). 
Si~$\al_0\in S$, le sous-groupe $\Pi^K$ est engendré par les $\al_0\al$, les
$\al_0 s_i$ ($\al \in S$), $h_0$  et $\al_0h_0\al_0$ ou $\alpha_0 h_0$ et
$h_0\al_0$ dans le cas (2).  De plus, en l'absence de produits elliptiques, tout
élément de torsion est conjugué à une réflexion ou à un elliptique non
générique, donc $\Pi^K$ est sans torsion.  Tout sous-groupe sans torsion
d'indice~2 doit contenir les produits de deux involutions arbitraires, donc
coïncide avec $\Pi^K$ dans les cas (0), (1a) et (3a). Dans le cas~(2), on trouve
deux possibilités pour $\Ga$ selon que $h_0\in \Ga$ ou $\al_0h_0 \in \Ga$.
D'après le lemme~\ref{lemm:interpretation_types}, $h_0 \in \Pi^K$ pour le type
pair, tandis que $\al_0h_0 \in \Pi^K$ pour le type impair
(figure~\ref{figu:ss_2beta_m}).  Enfin, il suffit de doubler le polygone
$\mcP_\Pi$ par rapport à un côté $\al_0\in S$ pour obtenir un domaine
fondamental $\mcP_\Ga$ de $\Ga$ dans tous les cas. Les côtés de $\mcP_\Ga$ sont
recollés par des éléments de la forme $\al_0\al$ ($\al\in S$), $\al_0 s_i$ (cas
(1a) et (3a)), $h_0$  et $\al_0h_0\al_0$ ou $\alpha_0 h_0$ et $h_0\al_0$ (cas
(2)), ce qui permet de déterminer facilement la topologie de la surface $\Sig$
(voir figure~\ref{figu:ss_2beta_m}).  L'ensemble $F$, quand $\HH/\Ga$ est de
type topologique fini, sera précisé à la proposition~\ref{prop:top_PiK}.
\end{proof}

\begin{rema} 
Les paires de côtés du  polygone fondamental \smash{$\mcP_\Pi$} sont identifiées
par des générateurs génériques ($\al_0\al$) ou non génériques ($\al_0s_i$,
$h_0$, ...). Plusieurs cas se présentent, suivant qu'il existe ou non des
identifications génériques entrelacées avec les identifications non génériques.
Cela n'influence pas la topologie du quotient, à l'exception des cas (2$^{+u}$)
et (2$^{+b}$). Cette remarque vaut également quand $E^u$ admet des produits
elliptiques, voir plus bas.
\end{rema}

\begin{figure}[h]
\begin{center}
\labellist
\small\hair 2pt
\pinlabel $h_0\notin  \Pi^K$ at 117 205
\pinlabel $\alpha_0$ at 180 140
\pinlabel $\alpha_s$ at 223 170
\pinlabel $\mcP_\Pi$ at 50 155
\pinlabel $\alpha_0h_0$ at 170 105
\pinlabel $h_0\alpha_0$ at 60 105
\endlabellist
\includegraphics[scale=0.5]{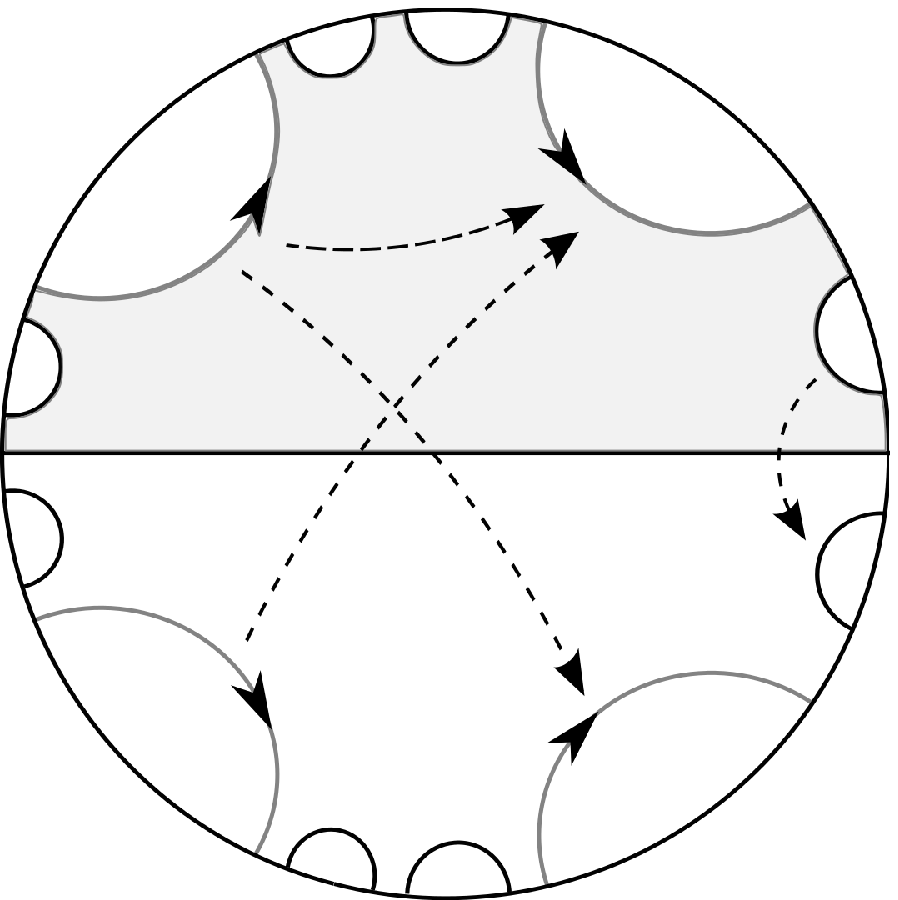}
\hspace{25pt}
\labellist
\small\hair 2pt
\pinlabel $\alpha_0$ at 140 140
\pinlabel $\alpha_s$ at 223 170
\pinlabel $s_0$ at 90 235
\pinlabel $s_1$ at 205 220
\pinlabel $\mcP_\Pi$ at 60 160
\pinlabel $\alpha_0s_0$ at 70 100
\pinlabel $\alpha_0s_1$ at 190 90
\endlabellist
\includegraphics[scale=0.5]{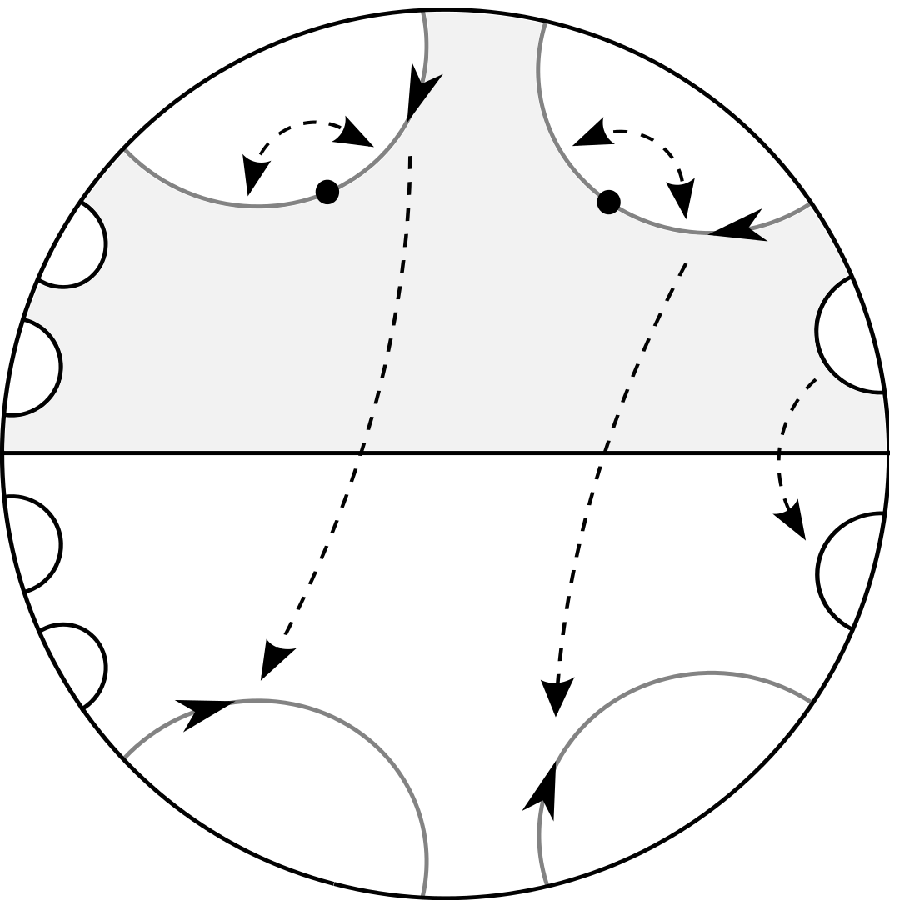}
\caption{$E^u$ sans produits elliptiques, cas (2$^-$)  et (3a)}
\label{figu:ss_2beta_m}
\end{center}
\end{figure}

Quand la surface $E^u$ admet des produits elliptiques, nous allons montrer que
$\nu=\nu^K=4$ dans tous les cas (voir la remarque~\ref{rema:indice4}); à
l'exception du cas (2), le groupe $\Pi$ est engendré par des involutions et les
sous-groupes sans torsion d'indice~4 seront tous distingués,
lemme~\ref{lemm:ind4_normal}. Soit $\mcP_\Pi$ un polygone fondamental pour
l'action de $\Pi$, et soit $(\al_0,\be)$ un couple de réflexions d'axes
orthogonaux bordant $\mcP_\Pi$, avec $\al_0$ générique et $\be$ générique ou
non. Si~$\Ga$ est un sous-groupe sans torsion d'indice~4 de $\Pi$, on a 
$\Pi=\Ga \amalg \Ga\al_0 \amalg \Ga\be \amalg \Ga\al_0\be$. Les polygones 
\begin{equation}
\label{equa:polygoneQ}
\mcP_{\Pi^K} =  \mcP_\Pi\cup \al_0(\mcP_\Pi) 
\esp \mathrm{et} \esp
\mcQ=\mcP_\Pi\cup \al_0(\mcP_\Pi) \cup \be(\mcP_\Pi) \cup \al_0
\be(\mcP_\Pi)
\end{equation}
forment des domaines fondamentaux pour l'action de $\Pi^K$ et de $\Ga$,
respectivement (figure~\ref{figu:as_3bc}).  
Noter que {\em $\mcQ$ est indépendant de $\Ga$}.

\par
Pour commencer, nous étudions la torsion du  groupe générique
$\Pi_\gen$ (qui coïncide avec $\Pi$ si $\Isom(f)$ est trivial). En
remplaçant $\Pi$ par $\Pi_\gen$, on définit des indices $\nu_\gen$ et
$\nu^K_\gen$ analogues à $\nu$ et $\nu_K$. Remarquer que $\Pi_\gen^+ =
\Pi_\gen^K$.  Soit $\pi_0(\mcG_f)$ l'ensemble des composantes connexes du
graphe $\mcG_f= (S,\mfE)$. Quand $\mcG_f$ est fini, on pose $\ell_\gen =
|\pi_0(\mcG_f)|$ et $k_\gen =|\mfE|$. Le résultat suivant donne en
particulier l'indice minimal quand $\Isom(f)$ est trivial.

\begin{prop}[indice minimal générique] 
\label{prop:ind_mini_gen}
Soit $E^u$ non élémentaire, de type fini et admettant des selles ($\mfE
\neq \emptyset$). Alors on a $\nu_\gen = \nu_\gen^K=4$. De plus, les
sous-groupes réalisant $\nu_\gen$ (resp.  $\nu_\gen^K$) sont exactement les
noyaux des morphismes $\om :\Pi_\gen \to \mu_2^2$ (resp. $\carac :
\Pi_\gen^K \to \mu_2$) satisfaisant~\eqref{equa:ker_om_st}
(resp.~\eqref{equa:ker_chi_st}); ces morphismes $\carac$ correspondent
bijectivement à
 $\mu_2 ^{\mcX}$ avec $\mcX=\pi_0(\mcG_f) \smallsetminus \{*\}$.
  Quand $\mcG_f$ est fini,
le groupe $\Pi_\gen$ admet $2^{k_\gen -1} 3^{\ell_\gen -1}$ sous-groupes sans
torsion d'indice~4, tous distingués, dont $2^{\ell_\gen -1}$ sont inclus dans
$\Pi_\gen^K$.
\end{prop}

\begin{proof} On  identifie comme plus haut $\Pi_\gen$ à un groupe de réflexions
du plan hyperbolique $\HH$ (proposition~\ref{prop:refl_hyp}), de polygone
fondamental $\mcP_0$. En considérant l'action sur $\HH$, on voit que tout
élément de torsion de $\Pi_\gen$ est conjugué à un générateur $\ga$
($\ga\in S$) ou a un elliptique $\al\be \in \Pi_\gen^+$ avec
$\{\al,\be\}\in \mfE$.  Par suite, si un morphisme $\om : \Pi_\gen \to
\mu_2^2$ satisfait~\eqref{equa:ker_om_st}, alors $\ker\om$ est sans torsion
et d'indice~4. Inversement, comme $\Pi_\gen$ est engendré par des
involutions, tout sous-groupe sans torsion d'indice~4 est de cette forme,
lemme~\ref{lemm:ind4_normal}.  On construit aisément de tels morphismes
$\om$ (voir ci-dessous). Sachant que $\nu_\gen\geq 4$
(remarque~\ref{rema:indice4}), on conclut que $\nu_\gen = 4$. Enfin, la
correspondance entre sous-groupes réalisant $\nu_{\rm gen}^K$ et caractères
$\carac : \Pi_\gen^+ \to \mu_2$ vérifiant~\eqref{equa:ker_chi_st} est
immédiate puisque \eqref{equa:ker_chi_st} équivaut ici à l'absence de
torsion dans le noyau $\ker\carac$.  

\par
Il reste à donner  une description de tous les morphismes $\om$ et
$\carac$ qui satisfont respectivement \eqref{equa:ker_om_st} et
\eqref{equa:ker_chi_st}.  Fixons une paire $\{\al_0,\be_0\}\in \mfE$. Les
noyaux $\ker\om$ correspondent aux morphismes $\om$ modulo composition par
un automorphisme du groupe $\mu_2^2$. On normalise~$\om$ en fixant les
valeurs (mutuellement distinctes) de $\om(\al_0)$ , $\om(\be_0)$ et
$\om(\al_0\be_0)$.  Si $\Pi_\gen^2$ désigne le sous-groupe engendré par les
carrés, les morphismes $\Pi_\gen \to \mu_2^2$ sont en bijection avec les
morphismes de $\Pi_\gen/ \Pi_\gen ^2 \simeq \Z/2\Z^{(S)}$ sur
$\mu_2^2$. Soit $C_0=(S_0,\mfE_0) \in \pi_0(\mcG_f)$ 
la composante de $\al_0$. Pour tout $C = (S_C,\mfE_C) \in
\pi_0(\mcG_f)\smallsetminus\{C_0\}$, on fixe (arbitrairement) un sommet
$\ga_C \in S_C$. Les morphismes $\om$ sont alors déterminés, par
composante, comme suit : d'abord $\om(\ga_C)$ (3~choix), ensuite $\om(\ga)$
pour $\ga\in S_C$ par récurrence à partir de $\om(\ga_C)$ dans la
composante~$C$ (deux choix pour le sommet voisin, donc $2^{\mfE_C}$ choix),
et de même pour $C_0$ avec seulement $2^{\mfE_0 \smallsetminus\{*\}}$ choix
puisque $\om(\al_0)$ et $\om(\be_0)$ sont déjà fixés.  Il reste à dénombrer
les sous-groupes réalisant l'indice $\nu_\gen^K$.  L'inclusion $\ker \om
\subset \Pi_\gen^K = \Pi_\gen^+$ équivaut à la condition
$$\om(\ga)\in\{\om(\al_0),\om(\be_0)\}\esp (\ga\in S).$$
Dans ce cas la restriction de $\om$ à $\Pi_\gen^+$ est à valeurs dans le
sous-groupe engendré par $\om(\al_0 \be_0)$ et définit un caractère
$\carac: \Pi_\gen^+\to \mu_2$, vérifiant~\eqref{equa:ker_chi_st} si $\om$
vérifie~\eqref{equa:ker_om_st}. Par la condition additionnelle ci-dessus,
les $\om(\ga)$ pour $\ga \in S_C$ ($C\in
\pi_0(\mcG_f)\smallsetminus\{C_0\}$) sont entièrement déterminés par
$\om(\ga_C)$ (avec deux choix pour $\om(\ga_C)$), et les $\om(\ga)$ sont
déterminés pour tout $\ga\in C_0$.  La description topologique des
quotients par $\ker\om \subset \Pi_\gen^+$ sera donnée dans la preuve de la
proposition~\ref{prop:top_quo_Eu}.
\end{proof}

Dans la suite, nous supposerons que le groupe $\Isom(f)$ n'est pas trivial.
Les quotients provenant de l'action simpliciale de $\Isom(f)$ sur le
graphe de contiguïté $\mcG_f=(S,\mcE)$ sont notés $[\mcG_f]_f$, $[S]_f$,
$[\mcE]_f$, $[\pi_0(\mcG_f)]_f$ ; quand $[\mcG_f]_f$ est fini, on pose
\begin{equation}
\label{equa:lk}
\ell=\chi ([\mcG_f]_f) \esp \mathrm{et} \esp k=|[\mcE]_f|,
\end{equation}
qui désignent respectivement la caractéristique d'Euler du quotient
$[\mcG_f]_f$ et le nombre de selles comptées modulo $\Isom(f)$.
Par définition, on a la relation $k+\ell = |[S]_f|$ (nombre 
de composantes de $\{f\neq 0\}$ modulo $\Isom(f)$).
Topologiquement,  $[\mcG_f]_f$ est une union disjointe de
segments, ou un cercle si $\mcG_f$ est connexe avec $\Isom(f)\simeq \Z$.
Hormis ce dernier cas particulier (où $\ell=0$), l'entier $\ell$ 
est aussi le nombre de composantes  de $\mcG_f$ modulo $\Isom(f)$.
\par

On traite maintenant le cas où $\Isom(f)\simeq\Z/2\Z$.  D'après la
proposition~\ref{prop:refl_hyp}, le groupe $\Pi$ est engendré par
$\Pi_\gen$ et un elliptique non générique $s$ pour le cas (1a) ou une
réflexion non générique $\sig$ pour le cas (1b). Cette involution non
générique fixe $\mcP_0$ (polygone fondamental de $\Pi_\gen$).  De plus,
dans le cas (1b), $\Isom(f)$ fixe un unique sommet de $\mcG_f$ noté
$\al_0$.

\begin{prop}[indice minimal, $\Isom(f)\simeq\Z/2\Z$]
\label{prop:ind_mini_sym}
Soit $E^u$ non élémentaire, de type fini et telle que
$\Isom(f)\simeq\Z/2\Z$. Dans le cas (1a), on suppose de plus que $E^u$
admet des selles.  Alors on~a $\nu=\nu^K=4$.  Les sous-groupes réalisant
$\nu$ (resp.  $\nu^K$) sont les noyaux des morphismes $\om :\Pi \to
\mu_2^2$ (resp. $\carac : \Pi^K \to \mu_2$) satisfaisant les
conditions~\eqref{equa:ker_om_st} (resp.~\eqref{equa:ker_chi_st}),
auxquelles il faut ajouter, suivant les cas :
\begin{itemize}
\item[(1a)] $\om(s)\in \De = \mu_2^2\setminus \{ (1,1)\}$,
\item[(1b)] $\om(\{\al_0,\sig,\al_0\sig\}) = \De$
(resp. $\carac(\sig)=-1$).
\end{itemize}
Les morphismes $\carac$ correspondent bijectivement à $\mu_2 ^ \mfX$ avec
$\mfX=[\pi_0(\mcG_f)]_f$, $[\pi_0(\mcG_f)]_f \smallsetminus\{*\}$ dans les cas
  (1a), (1b).  Si $\mcG_f$ est fini, le groupe $\Pi$ comprend $2^{k-1}
  3^{\ell}$ (resp.  $2^{k} 3^{\ell-1}$) sous-groupes sans torsion d'indice
  4, tous distingués, dont $2^{\ell}$ (resp. $2^{\ell-1}$) sont inclus dans
  $\Pi^K$.
\end{prop}

\begin{proof} Dans le cas (1a), on choisit  une géodésique (arbitraire)
$\xi$ passant par le centre de~$s$ et ne rencontrant pas le bord $\partial
  \mcP_0$. Dans le cas (1b), on note $\xi$ l'axe de $\sig$.  La
  géodésique~$\xi$ partage $\mcP_0$ en deux moitiés isométriques formant un
  domaine fondamental pour~$\Pi$ ; on choisit l'une d'elle, notée
  $\mcP_\Pi$. Pour (1a), on fixe une paire $\{\al_0,\be_0\}\in \mfE$ qui
  représente un sommet de $\partial \mcP_\Pi$ (comme graphe, $\partial
  \mcP_0$ est dual à $\mcG_f$). Pour (1b), on rappelle que $\al_0\sig$ est
  elliptique. Soit $\mcQ$ définit par~\eqref{equa:polygoneQ} avec
  $\al=\al_0$ et $\be=\be_0$ pour (1a), $\be=\sig$ pour (1b).
La fin de la preuve est une adaptation immédiate de celle de la
proposition~\ref{prop:ind_mini_gen}. Le polygone $\mcQ$ fournit un domaine
fondamental pour tout sous-groupe sans torsion d'indice~4 et le
lemme~\ref{lemm:ind4_normal} s'applique. Inversement, la torsion de 
$\Pi$ étant connue, les noyaux des morphismes $\om$ et~$\ro$ qui 
satisfont aux conditions de l'énoncé définissent des sous-groupes sans
torsion d'indice~4.  Pour le dénombrement des sous-groupes~$\Gamma$, le
bord $\partial \mcP_\Pi$ ne fait intervenir que les éléments de $\mcG_f$
modulo l'action de $\Isom(f)$ ; dans le cas (1a), il y a une composante de
bord additionnelle $\xi$ (sans sommet), et dans le cas (1b) la paire
$\{\al_0,\sig\}$ joue le rôle de $\{\al_0,\be_0\}$. Noter cependant que si
$\Gamma$ est inclus dans $\Pi^K$, la surface $\HH/\Gamma$ n'est pas
orientable (si $\mfE\neq \emptyset$) car les éléments $\al_0 s$ et
$\al_0\sig\ga$ ($\ga \in S$) sont indirects (voir
proposition~\ref{prop:top_quo_Eu}).
\end{proof}
\begin{figure}[h]
\begin{center}
\labellist
\small\hair 2pt
\pinlabel $\alpha_0$ at 200 140
\pinlabel $\alpha_s$ at 232 157
\pinlabel $\mcP_\Pi$ at 160 170
\pinlabel $\sig_0$ at 115 215
\pinlabel $s_0$ at 209 217
\endlabellist
\includegraphics[scale=0.5]{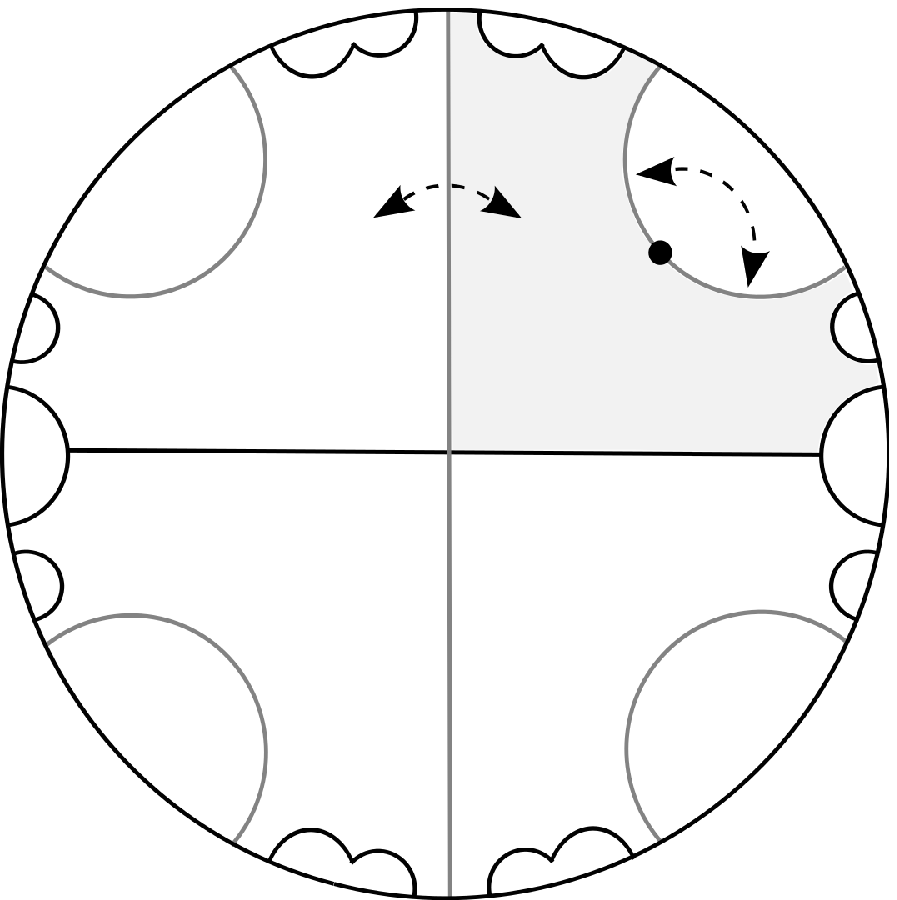}
\hspace{25pt}
\labellist
\small\hair 2pt
\pinlabel $\alpha_0$ at 205 140
\pinlabel $\alpha_1$ at 190 184
\pinlabel $\mcP_\Pi$ at 160 155
\pinlabel $\sig_0$ at 115 203
\pinlabel $\sig_1$ at 177 222
\endlabellist
\includegraphics[scale=0.5]{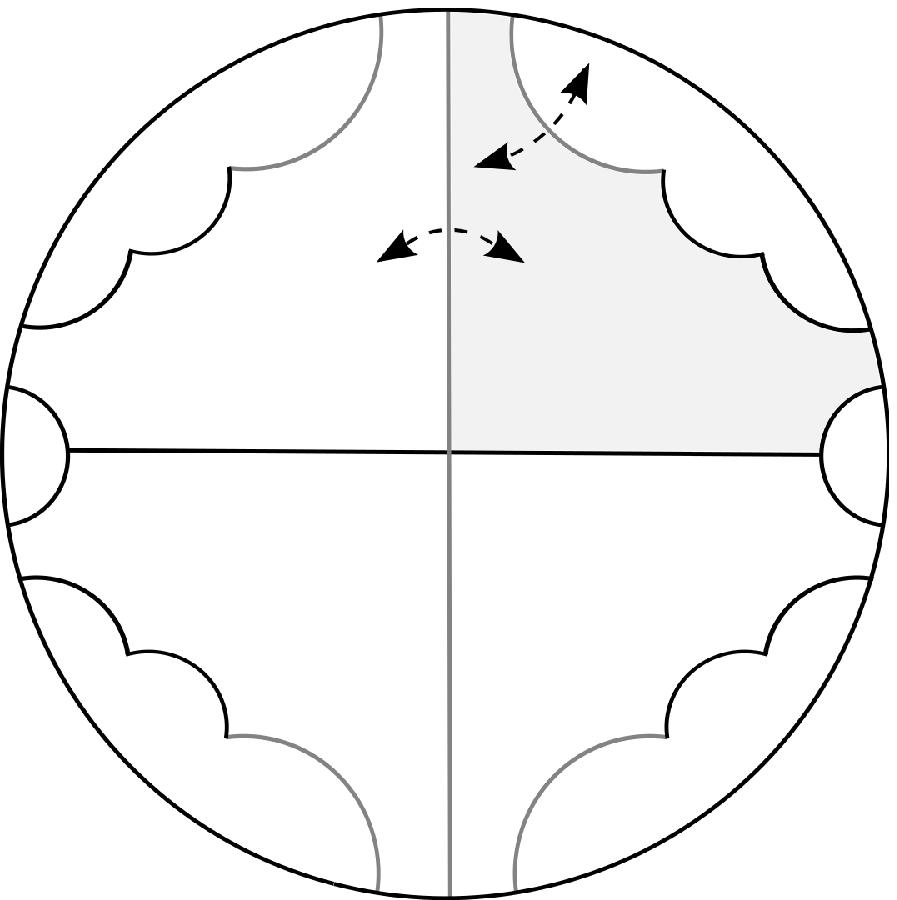}
\caption{$E^u$ avec produits elliptiques, cas (3b) et (3c$^{+u}$)}
\label{figu:as_3bc}
\end{center}
\end{figure}

Quand $\Isom(f)$ est isomorphe à $D_\infty$, les générateurs non génériques
de $\Pi$ fixant $\mcP_0$ seront notés respectivement $s_i$ ($i=0,1$), $s_0$
et $\sig_0$, $\sig_i$ ($i=0,1$) dans les cas (3a), (3b), (3c). Chaque
réflexion non générique $\sig_i$ fixe un unique sommet $\al_i$ de $\mcG_f$.
\begin{prop}[indice minimal, $\Isom(f)\simeq D_\infty$]
\label{prop:ind_mini_diedral}
Soit $E^u$ non élémentaire, de type fini et telle que $\Isom(f)\simeq
D_\infty$. Dans le cas (3a), on suppose de plus que $E^u$ admet des selles.
Alors on~a $\nu=\nu^K=4$.  Les sous-groupes réalisant $\nu$ (resp.
$\nu^K$) sont les noyaux des morphismes $\om :\Pi \to \mu_2^2$
(resp. $\carac : \Pi^K \to \mu_2$) satisfaisant les
conditions~\eqref{equa:ker_om_st} (resp.~\eqref{equa:ker_chi_st}),
auxquelles il faut ajouter, suivant les cas :
\begin{itemize}
\item[(3a)] $\om(s_i)\in \De = \mu_2^2\setminus \{ (1,1)\} $  
pour $i=0,1$,
\item[(3b)] $\om(s_0)\in \De$ et $\om(\{\al_0,\sig_0,\al_0\sig_0\}) = \De$
(resp. $\carac(\sig_0)=-1$),
\item[(3c)] $\om(\{\al_i,\sig_i,\al_i\sig_i\}) = \De$
  (resp. $\carac(\sig_i)=-1$) pour $i=0,1$.
\end{itemize}
L'ensemble $[\mcG_f]_f$ est fini et, dans le cas (3a) (resp. (3b),(3c)),
le groupe $\Pi$ admet $2^{k-1} 3^{\ell+1}$
(resp. $2^{k} 3^{\ell}$, $2^{k+1} 3^{\ell-1}$) sous-groupes sans torsion
d'indice~4, tous distingués, dont $2^{\ell+1}$ (resp. $2^\ell$,
$2^{\ell-1}$) sont inclus dans $\Pi^K$.
\end{prop}

\begin{proof}  
On obtient un domaine fondamental $\mcP_\Pi$ en découpant $\mcP_0$ par deux
géodésiques $\xi,\xi'$ associées aux générateurs non génériques
(involutions elliptiques ou réflexions, suivant les cas, voir
figure~\ref{figu:as_3bc}).  Comme précédemment, les sous-groupes réalisant
$\nu$ et $\nu^K$ sont d'indice~4, tous distingués
(lemme~\ref{lemm:ind4_normal}) et décrits par les morphismes $\om$ et $\ro$
vérifiant les conditions de l'énoncé qui éliminent la torsion dans les
noyaux. La construction de ces morphismes et leur décompte quand
$[\mcG_f]_f$ est fini s'effectuent comme plus haut (preuves des
propositions~\ref{prop:ind_mini_gen} et~\ref{prop:ind_mini_sym}), en
explicitant des générateurs de $\Pi$ (pour $\om$) et de $\Pi^K$
(pour~$\carac$) grâce au théorème de Poincaré.  Soit $S_1$ l'ensemble des
$\al\in S$ dont les axes bordent $\mcP_\Pi$.  Dans le cas (3a), on fixe une
paire contiguë $\{\al_0,\be_0\}$ avec $\al_0, \be_0\in S_1$ ; dans les cas
(3b) et (3c), on remplace $\be_0$ par  $\be=\sig_0$.  Le groupe $\Pi^K$ est
engendré par les $(\al_0\al)_{\al \in S_1}$ et $\de_0,\de_1$ avec
$(\de_0,\de_1) =(\al_0 s_0 , \al_0 s_1)$, $(\sig_0,\al_0 s_1$),
$(\sig_0,\sig_1)$ pour (3a), (3b), (3c).  Si $\al$, $\al'\in S_1$ sont
contiguës, on doit avoir $\ro(\al_0\al)\ro(\al_0\al')=\ro(\al\al')=
-1$. Ainsi, $\ro(\al_0\al)$ est imposé si $\al\in S_1$ appartient à
la composante de $\al_0$, et $\ro$ est déterminé par une valeur pour chacune
des autres composantes génériques de $\partial\mcP_\Pi$. De plus $\ro(\al_0
s_i)$ est arbitraire, d'où l'assertion sur les morphismes~$\ro$. Pour
compter les morphismes $\om$ (normalisés sur $\{\al_0,\be,\al_0\be\}$,
$\be= \be_0$ ou $\sig_0$ suivant les cas), on procède comme pour la preuve
du cas générique (proposition~\ref{prop:ind_mini_gen}). Dans le cas (3c),
noter que $\Pi$ est aussi un groupe de réflexions, les paires
$\{\al_i,\sig_i\}$ ($i=0,1$) jouant le rôle de selles ; il se comporte
(pour $\om$) comme un groupe générique avec $\ell$ composantes et $k+2$
selles, voir proposition~\ref{prop:ind_mini_gen}.
\end{proof}

\begin{prop}[indice minimal, $\Isom(f)\simeq \Z$]
\label{prop:ind_mini_trans}
Soit $E^u$ non élémentaire, de type fini, telle que $\Isom(f)\simeq
\Z$ et admettant des selles. On a alors $\nu= \nu^K = 4$.  Les sous-groupes
réalisant $\nu^K$ correspondent bijectivement aux caractères $\carac :
\Pi^K \to \mu_2$ vérifiant~\eqref{equa:ker_chi_st}.  Plus précisément,
$[\mcG_f]_f$ étant fini,
\begin{enumerate}
\item \label{asse:conn}
si $\mcG_f$ n'est pas connexe,  il y a  $2^{\ell+1}$
sous-groupes $\ker \ro$, dont $2^\ell$ distingués dans $\Pi$ et 
 $2^\ell$ échangés par réflexion générique, 
\item
si $\mcG_f$ est connexe, l'indice $\nu^K$ est réalisé par 2
  sous-groupes, distingués dans $\Pi$. 
\end{enumerate}
\end{prop}

\begin{proof}  Le groupe $\Pi$ est engendré par $\Pi_\gen$ et un
élément $h_0$ d'ordre infini, direct ou non, qui laisse invariant le polygone
fondamental $\mcP_0$ de $\Pi_\gen$ (proposition~\ref{prop:refl_hyp}).  Si
$\mcD$ est un domaine fondamental de $h_0$ bordé par deux géodésiques, alors
$\mcP_\Pi=\mcD\cap\mcP_0$ est un domaine pour $\Pi$ qui pave $\mcP_0$.  On
choisit $\partial \mcD$ disjoint de $\partial \mcP_0$ si $\mcG_f$ n'est pas
connexe et orthogonal à $\partial \mcP_0$ sinon, de sorte que les
composantes de $\mcD \cap \partial \mcP_0$ représentent
$[\pi_0(\mcG_f)]_f$.  Soit $\{\al_0,\be_0\}$ une paire de réflexions
contiguës dont les axes bordent $\mcP_\Pi$. On considère les polygones
$\mcP_{\Pi^K}$ et $\mcQ$ définis par~\eqref{equa:polygoneQ} avec $\be=
\be_0$.  Contrairement aux cas précédents, le lemme~\ref{lemm:ind4_normal}
ne s'applique pas car (ii) n'est pas {\it a priori} vérifiée : $\mcQ$ est
un domaine fondamental pour tout sous-groupe~$\Ga$ sans torsion d'indice~4
et on trouve $h_0$,
$\al_0h_0$, $\be_0h_0$, $\al_0\be_0 h_0$ parmi les générateurs
potentiels de $\Ga$.  Pour simplifier, nous cherchons seulement les
$\Ga\subset \Pi^K= \Pi^+$ {\it via} les caractères $\carac : \Pi^K\to
\mu_2$. D'après le théorème de Poincaré, $\Pi^K$ est engendré par les
identifications des côtés de $\mcP_{\Pi^K}$, à savoir les $\al_0\al$ ($\al \in
S$ bordant $\mcP_\Pi$) ainsi que  $h_0$ et $\al_0 h_0 \al_0$ dans le cas pair, 
ou  $\al_0h_0$ et $h_0\al_0 $ dans le cas impair.
Comme les éléments de
torsion de $\Pi$ sont ceux de $\Pi_\gen$, l'absence de torsion dans $\ker
\carac$ équivaut à~\eqref{equa:ker_chi_st}. De plus
$\carac(\al\al_0)=\carac(\al_0\al$), donc $\ker \carac$ est distingué dans
$\Pi$ si et seulement si $\carac(h_0) = \carac(\al_0h_0\al_0$) pour le cas pair, 
$\carac(\al_0 h_0) = \carac(h_0\al_0)$ pour le cas impair.

\par La construction des caractères $\carac$ est analogue à celle du cas
générique, mais il convient de distinguer deux cas. Si $\mcG_f$ n'est pas
connexe, les valeurs de $\carac(h_0)$ et $\carac(\al_0h_0\al_0)$ (resp.
$\carac(\al_0 h_0)$ et $\carac(h_0\al_0)$) pour le cas pair (resp.
impair) peuvent être arbitraires, d'où l'assertion~(\ref{asse:conn}).
Si $\mcG_f$ est connexe, alors $\ell=0$ et $k$ est pair (les zéros de $f$ sont
isolés, en nombre pair dans chaque période). Le côtés de $\mcP_0$ sont
alternativement de type espace et temps : le nombre de sommets entre les axes de
$\al_0$ et de $h_0\al_0 h_0^{-1}$ est pair. Par suite $\al_0 h_0\al_0 h_0^{-1}$
est le produit d'un nombre pair d'elliptiques génériques et on a $\ro(h_0) = \ro
(\al_0 h_0\al_0)$. Le noyau $\Ga=\ker\ro$ est distingué dans $\Pi$ et on trouve
deux possibilités, suivant que $h_0\in \Ga$ ou que $\al_0\be_0 h_0 \in \Ga$.
\end{proof}

Nous précisons maintenant le nombre et la topologie des quotients lisses
minimaux de $E^u$ portant un champ de Killing, dans le cas où le graphe $\mcG_f$
est fini modulo l'action de $\Isom(f)$. Voici d'abord la description du quotient
$\HH/\Pi^K$ comme surface \og orbifolde\fg. Il s'agit d'une surface compacte
$\Sig$ éventuellement à bord, avec un ensemble $Ell$ de points elliptiques
d'ordre~2, et privée d'un ensemble fini $P$ de points, désignés par \og
pointes\fg\ ; on pose $p_{\rm int} = |P\cap (\Sig\smallsetminus \partial\Sig)|$
et $p_{\partial} = |P\cap \partial\Sig|$ (pointes internes ou pointes au bord).
Géométriquement, les points elliptiques de $E^u/G^K$ sont des \og
demi-selles\fg. On rappelle que les entiers $\ell$ et $k$ sont définis par
\eqref{equa:lk}.

\begin{prop}[structure orbifolde de  $E^u/G^K$]
\label{prop:top_PiK}
Soit $E^u$ non élémentaire. On suppose que $\mcG_f/\Isom(f)$ est fini. La
topologie orbifolde du quotient $E^u/G^K$ est donnée par la 
table~\ref{tabl:top_PiK},
où les surfaces compactes $\SSS$, $\PP$, $\D$, $\T$, $\K$, $\M$ et $\A$
désignent
respectivement la sphère, le plan projectif, le disque, le tore, la bouteille de
Klein, le ruban de Möbius et l'anneau.
\begin{table}[h]
\centering
\renewcommand{\arraystretch}{1.2}
\renewcommand{\tabcolsep}{0.4 em}
\begin{tabular}{llcll|llcll}
\hline
cas  & $\Sigma$ & $|Ell|$  & $p_{\rm int}$ & $p_\partial$ &
cas  & $\Sigma$ & $|Ell|$  & $p_{\rm int}$ & $p_\partial$ \\
\hline
(0)  & $\SSS$ & $k$ & $\ell$ & & %
(2$^-$) & $\T$  & $k$ & $\ell$ \\
(1a)  & $\PP$ & $k$ & $\ell$ & & %
(3a) & $\K$  & $k$ & $\ell$ &  \\
(1b)  & $\D$ & $k$ & $\ell-1$ & 1 & %
(3b) & $\M$  & $k$ & $\ell - 1$ & 1 \\
(2$^{+u}$)   & $\SSS$ & $k$ & $\ell+2$ & &%
(3c$^{+u}$) &  $\D$ & $k$ & $\ell-1$ & 2 \\
(2$^{+b}$)& $\T$  & $k$ & $\ell$ & &
(3c$^{+b}$) & $\A$ & $k$ & $\ell-2$ & 2 \\
\hline 
\end{tabular}
\caption{Structure  orbifolde de $E^u/G^K$}
\label{tabl:top_PiK}
\end{table}
\end{prop}

\begin{proof}
On procède au cas par cas à partir des domaines fondamentaux pour l'action
de $\Pi^K$ sur $\HH$ décrits dans les preuves ci-dessus.  On remarquera que
les cas (2$^{+b}$) et (2$^-$) ont ici le même comportement.
\end{proof}

\begin{rema}  La surface $E^u/G_\gen^K$, privée de ses points
elliptiques, est isométrique à la surface $X'$
construite à la remarque~\ref{rema:variante_Y} (pour un choix convenable
des $\sigma_\alpha$).
\end{rema}

Passons à la description topologique des quotients lisses minimaux de
$E^u$, c'est-à-dire des quotients par les sous-groupes sans torsion
d'indice minimal $\nu^K$ dans $G^K$, à isométrie près.  Il revient au même
de décrire les $\Gamma$ réalisant l'indice minimal $\nu^K$, à conjugaison
près dans le groupe $\Pi$. En effet, comme $\Isom^K(E^u)$ centralise
$\Isom^0(E^u)$, deux sous-groupes de $G^K$ conjugués dans $\Isom(E^u)$ le
sont aussi dans $G^K$.  La topologie de la surface $\HH/\Gamma$ sera codée
par la signature $(g;p)^\pm$ de $\Gamma$, où~$g$ désigne le genre,~$p$ le
nombre de bouts (pointes) et avec $\pm$ pour orientable ou non.

\begin{prop}[topologie des petits quotients de $E^u$]
\label{prop:top_quo_Eu}
Soit $E^u$ non élémentaire.  On suppose que
$\mcG_f/\Isom(f)$ est fini.  Si $\nu^K=2$ ($E^u$ n'a pas de produits
elliptiques), la topologie de $E^u/G^K$ est décrite à la
table~\ref{tabl:top_PiK}, cas (0),(1a),(2) et (3a) avec $k=0$.  Si
$\nu^K=4$, la topologie des sous-groupes sans torsion d'indice~2 dans
$\Pi^K$ est donnée par la table~\ref{tabl:petits_quotients} suivante.  Leur
nombre, à conjugaison près dans $\Pi$ et à topologie fixée, est également
déterminé (colonne~$\#$). Pour le cas (2), on pose $N_j^\ell =
\sum_{i=0}^2\binom{\ell}{j-i}$ ; le cas (2$^{+b}$) est explicité à la
proposition complémentaire~\ref{prop:isomZ_top_fixee}.
\begin{table}[h]
\centering
\renewcommand{\arraystretch}{1.5}
\renewcommand{\tabcolsep}{0.3 em}
{\small 
\begin{tabular}{llllllll}
\hline
$\Isom(f)$ &   &     & $\chi(\Gamma$) & paramètre &  signature 
& \# & total \\
\hline
   trivial      &  (0)& $k>0$ & $4-k-2\ell$ & 
\renewcommand{\arraystretch}{1}
$
\begin{array}[t]{l}
0\leq j \leq \ell \\
k+j~\mathrm{pair}~(*)
\end{array}
$
\renewcommand{\arraystretch}{1.5}
 &  $(\frac{k+j}{2} -1 ; 2\ell-j)^+$ & 
$\binom{\ell}{j}$& $ 2^{\ell-1}$ \\
\hline
 $\Z/2\Z$ & (1a) &  $k>0$  & $2-k-2\ell$ & 
$0\leq j \leq \ell$ et $(*)$  &  $(k+j; 2\ell-j)^-$ & 
$2 \binom{\ell}{j}$& $ 2^{\ell}$ \\

 & (1b) &  $k+j=0$  & $3-2\ell$ & $k=0$ et $j=0$
 & $(0;2\ell - 1)^+$ &  1 &  \\
& & $k+j > 0$ &  $3 - 2\ell -k$
&$0\leq j \leq \ell -1$ & $(k+j ; 2\ell-j-1)^-$ & 
$\binom{\ell-1}{j} $ &   $2^{\ell-1}$\\
\hline

 $\Z, k>0$ & (2$^{+u}$)    & $\ell=0$  & $-k$ & $j=0,2$ 
 &  $(\frac{k+j}{2}-1; 4-j)^+ $ & 
$2 $& $ 2$ \\
  & & $\ell\geq 1$ & $-k-2\ell$ &  $0\leq j \leq \ell+2$   et $(*)$
 &  $(\frac{k+j}{2}-1; 2\ell +4-j)^+ $ & 
$N_j^\ell$ &   $3 \cdot 2^{\ell-1}$ \\

        & (2$^{+b}$)&   & $-k-2\ell$ & $0\leq j \leq \ell$   et $(*)$
 &  $(\frac{k+j}{2}+1; 2\ell -j)^+ $ & 
\ref{prop:isomZ_top_fixee}   &   $3 \cdot 2^{\ell-1}$\\

   & (2$^-$)&   & $-k-2\ell$ & $0\leq j \leq \ell$   et $(*)$
 &  $(\frac{k+j}{2}+1; 2\ell -j)^+ $ & 
$ 3\binom{\ell}{j}$
  &   $3 \cdot 2^{\ell-1}$\\
\hline

  $D_\infty$ & (3a)     & $k>0$   & $-k-2\ell$ & 
$0\leq j \leq \ell$   et $(*)$  &  $(k+j+2; 2\ell -j)^- $ & 
$4 \binom{\ell}{j} $ &   $2^{\ell+1}$ \\

& (3b) &  & $1- 2\ell -k$
&$0\leq j \leq \ell -1$ & $(k+j+2 ; 2\ell-j-1)^-$ & 
$2 \binom{\ell-1}{j} $ &   $2^{\ell}$\\

& (3c$^{+u}$) &   & $2- 2\ell -k$
&$0\leq j \leq \ell -1$ & $(k+j ; 2\ell-j)^-$ & 
$\binom{\ell-1}{j} $ &   $2^{\ell-1}$\\
& (3c$^{+b}$) &   & $2- 2\ell -k$
&$0\leq j \leq \ell -2$ & $(k+j+2 ; 2\ell-j-2)^-$ & 
$2 \binom{\ell-2}{j} $ &   $2^{\ell-1}$\\
\hline
\end{tabular}
}
\caption{Topologie des petits quotients de $E^u$ ($\nu^K = 4$)}
\label{tabl:petits_quotients}
\end{table}
\end{prop}

\begin{proof} Les sous-groupes $\Gamma$ cherchés correspondent évidemment aux 
revêtements doubles orbifoldes $Y$ de $X=\HH/\Pi^K$, dont la structure est
explicitée à la proposition~\ref{prop:top_PiK}. Le paramètre $j$ représente
le nombre de pointes internes à $X$ au-dessus desquelles $Y\to X$ est
ramifié. Quand $j$ croît, le genre de $Y$ croît tandis que le nombre de
pointes de $Y$ décroît, la caractéristique d'Euler restant
constante. Le lemme ci-dessous permet de déterminer le nombre et 
la topologie des sous-groupes $\Ga$ quand $j$ est fixé. On retrouve le
nombre total de sous-groupes, déjà connu grâce aux caractères
 $\carac : \Pi^K\to \mu_2$ 
(propositions~\ref{prop:ind_mini_gen}-\ref{prop:ind_mini_trans}).
Dans le cas (2) avec $\ell \geq 1$, les sous-groupes $\Ga$ réalisant $\nu^K$ ne
sont pas tous distingués. Pour (2$^{+u}$), $E^u/G^K$ comprend $2$ pointes
spéciales associées à $h_0$ et $\al_0 h_0 \al_0$, et $\Ga$ est distingué si et seulement si
elles sont simultanément ramifiées ou non. À topologie fixée, c'est-à-dire $j$ 
fixé ($0\leq j \leq \ell +2$, $k+j$ pair), le nombre de sous-groupes $\Ga$ à
conjugaison près est donc $N_j^\ell= \binom{\ell}{j} + \binom{\ell}{j-1} +
\binom{\ell}{j-2}$, avec la convention $\binom{\ell}{m} = 0$ si $m (\ell-m) <0$;
noter que la somme sur $j$ de chacun des trois termes de $N_j^\ell$ vaut
$2^{\ell-1}$. Dans le cas (2$^{+b}$), la condition pour que $\Ga$ soit distingué
s'écrit $\ro(h_0\al_0) = \ro(\al_0 h_0)$. Soit $\Sig$ la surface de la
table~\ref{tabl:top_PiK}. Le couple $(h_0\al_0, \al_0 h_0)$ induit une base de
l'homologie $H_1(\Sig,\Z)$. On peut fixer arbitrairement $\ro(h_0\al_0)$ et
$\ro(\al_0 h_0)$ (4 choix), et indépendamment $j$ pointes ramifiées avec $k+j$
pair ($\binom{\ell}{j}$ choix).  La proportion de sous-groupes $\Ga$ distingués,
c'est-à-dire 1/2, est donc indépendante de $j$. 
\end{proof}

\begin{lemm}[revêtements doubles orbifoldes]
Soit $X$ une surface orbifolde compacte, de type topologique fini, avec ou
sans bord et possédant $N\geq 1 $ points elliptiques d'ordre~2.  On note
$g_X$ le genre de $X$ et $m\geq 0$ le nombre de composantes de $\partial
X$.  Si $m=0$, on suppose de plus que $N$ est pair.  Soit $q=\kappa g_X +
\max(m-1,0)$ avec $\kappa=2$ si $X$ est orientable, $\kappa=1$ sinon.
Alors il existe exactement $2^q$ revêtements doubles orbifoldes de $X$ par
des surfaces lisses $Y$, fermées et connexes. De plus, les surfaces $Y$
sont toutes homéomorphes et
\begin{enumerate}
\item $Y$ est orientable, de genre $2g_X +N/2 -1$, si $X$ est orientable et
  sans bord,
\item $Y$ est non orientable, de genre $2\kappa g_X +N + 2 m -2$, sinon.
\end{enumerate}
\end{lemm}

\begin{proof}
Soient $x_1,\ldots,x_N$ les points elliptiques.  À partir de la
présentation usuelle du groupe fondamental de $X\smallsetminus
\{x_1,\ldots,x_N\}$, on trouve facilement les revêtements doubles connexes
$Z$ de $X$ ramifiés en tous les points $x_i$. Noter que $Z$ est orientable
si et seulement si~$X$ l'est. Si $\partial X = \emptyset$, on prend $Y=Z$. Sinon, on
définit $Y$ en identifiant les points de $\partial Z$ par l'involution du
revêtement ; vu la nature de celle-ci, cette opération produit
automatiquement une surface non orientable, même si $X$ est
orientable. Enfin, le genre de $Y$ est déterminé par la caractéristique
d'Euler orbifolde de~$X$.
\end{proof} 

\par Nous nous concentrons maintenant sur le cas (2$^{+b}$) (donc $\ell \geq
2$). La surface $\Sig$ de la proposition~\ref{prop:top_PiK} est homéomorphe
au tore. On rappelle que si $\ro :\Pi^K\to \mu_2$ est non trivial, le
caractère distingué dans $\Pi$ de $\Ga=\ker \ro$ se lit sur une paire
d'éléments de $\Pi^K$ (preuve de la proposition~\ref{prop:ind_mini_trans}),
nécessairement hyperboliques.  Plus précisément, pour (2$^{+b}$), il
s'agit de $h_0$ et $\al_0 h_0 \al_0$, dont les géodésiques fermées
associées découpent le tore $\Sig$ en deux cylindres. Ainsi, les pointes et
les points elliptiques de l'orbifolde $\HH/\Pi^K$ se scindent  en
deux sous-ensembles de cardinaux respectifs $\ell_1\geq 1$, $\ell_2 \geq 1$
pour les pointes ($\ell_1+\ell_2= \ell$) et $k_1$, $k_2$ ($k_1+k_2=k)$ pour
les elliptiques. Ces entiers se déduisent de $f$. Par exemple, $k_1$ et
$k_2$ correspondent aux selles sur le bord d'un domaine fondamental $\mcP$
pour $\Pi$, réparties suivant que le champ est rentrant dans $\mcP$ ou
sortant.

\begin{prop}[complément à~\ref{prop:top_quo_Eu}]
\label{prop:isomZ_top_fixee}
Soit $E^u$ non élémentaire, de type fini et telle que $\Isom(f)\simeq
\Z$. On suppose que $E^u$ est de type pair bilatère et admet des selles.  Soit
$j=0,\ldots,\ell$ tel que $k+j$ soit pair. À conjugaison près dans $\Pi$,
le nombre $n_j^\ell$ de sous-groupes d'indice~2 de $\Pi^K$ avec $\HH/\Ga$ homéomorphe
à $(\frac{k+j}{2}+1; 2\ell -j)^+$ est donné par 
$$
n_j^\ell = 2 \binom{\ell}{j} +  2 \sum\raisebox{0.7ex}{**}
\binom{\ell_1}{j_1}\binom{\ell-
  \ell_1}{j-j_1}  
$$ où la somme $\sum^{**}$ est indexée par $0\leq j_1\leq \ell_1$
avec $k_1+ j_1$ pair.
\end{prop}

\begin{proof}
Comme nous l'avons observé plus haut, les courbes simples associées à $h_0$ et
$\al_0h_0\al_0$ découpent le tore $\Sig$ en deux cylindres $C_1$ et $C_2$.
On peut choisir $h'_0\in \Pi^K$ tel que $(h_0,h'_0)$
forme une base de l'homologie de $\Sig$. Avec les notations précédentes, on peut
choisir arbitrairement les valeurs du caractère $\ro$ sur $h_0$ et $h'_0$ (4
choix).  La topologie de $\HH/\Ga$ ne dépend que du nombre de pointes ramifiées
$j$, que l'on répartit en $j=j_1+j_2$ dans $C_1$ et $C_2$. La condition
$\ro(h_0)= \ro(\al_0h_0\al_0)$ équivaut à $k_1+j_1$ est pair, d'où 
$$n_j^\ell = 4 \sum \raisebox{0.7ex}{**}  
\binom{\ell_1}{j_1}\binom{\ell- \ell_1}{j-j_1} + 2  \sum\raisebox{0.7ex}{*}  
\binom{\ell_1}{j_1}\binom{\ell- \ell_1}{j-j_1},$$ 
où la somme $\sum^{**}$ (resp.
 $\sum^*$) est indexée par $0\leq j_1\leq \ell_1$ avec $k_1+ j_1$ pair (resp.
impair). On en déduit facilement l'expression donnée dans l'énoncé.
\end{proof}

\begin{rema}
\label{rema:quotients_Sprime}
Les propositions~\ref{prop:ind_mini_sans_pe}, \ref{prop:ind_mini_gen},
\ref{prop:ind_mini_sym}, \ref{prop:ind_mini_diedral},
\ref{prop:ind_mini_trans}, \ref{prop:top_PiK},
\ref{prop:top_quo_Eu} et \ref{prop:isomZ_top_fixee} sont
valables, avec une preuve identique, pour toute partie $S'\subset S$
localement finie et invariante.  Il suffit de remplacer $\nu$, $\nu^K$ par
des indices analogues relativement à $G'$, et $k$, $\ell$ par $k'$, $\ell'$
définis comme dans~\eqref{equa:lk} par l'action de $\Isom(f)$ sur le graphe
de contiguïté $\mcG'$ associé à $S'$ ($\ell'= 0$ n'est possible que si $f$
est de type fini et avec $S'=S$). On dira que $E^u$ n'a pas de produits
elliptiques {\em relativement à $S'$} si $G'$ ne contient pas de
sous-groupe isomorphe à $\mu_2^2$. Enfin, les types (pair, impair, etc)
sont déjà définis relativement à $S'$
(définition~\ref{defi:types_al_be}). Ainsi on pourra construire de
nombreuses familles de quotients lisses, que $f$ soit de type fini ou
non; on pourra prendre des parties strictes $S'$ localement finies, arbitrairement grande si $f$ n'est de type fini. 
\end{rema}

\begin{prop}[types topologiques à géométrie locale fixée]
\label{prop:top_classe_fixee}
Soit $f$ une fonction $\mcC^\infty$. On suppose que l'ensemble
$\{f\neq 0\}$ possède au moins~4 composantes connexes. Alors pour toute surface
$\Sigma$ non compacte, orientable et de type topologique fini, la surface $E^u$
admet un quotient difféomorphe à $\Sigma$.
\end{prop}

\begin{proof}
Il suffit de réaliser les types topologiques $(0,2)^+$, $(0,3)^+$ et $(1,1)^+$.
En effet, on obtient évidemment les types $(0,d+2)^+$ et $(1,d)^+$ à partir de
$(0,3)^+$ et $(1,1)^+$, grâce à des revêtements cycliques de degré $d\geq 1$
de $(0,2)^+$ et $(1,0)^+$, respectivement. De plus, pour tout $(g,p)\in
\N\times\N^*$ avec $g \geq 2$,  il existe un revêtement du tore par la surface
fermée de genre $g$, de degré $d=2g+p-2 \geq 1$, ayant une unique fibre ramifiée
de cardinal $p$,  voir \cite[proposition~3.3]{EKS1984} ; on obtient donc le type
$(g,p)^+$ comme revêtement de $(1,1)^+$.

\par
Soit $S'\subset S$ un sous-ensemble formé de 4 composantes connexes. D'après la
proposition~\ref{prop:refl_hyp}, l'action de $G'_\gen \subset G_\gen$ sur $E^u$
est conjuguée à celle d'un groupe de réflexions associé à un quadrilatère
hyperbolique (l'invariance par $\Isom(f)$ n'est pas nécessaire pour ce point).
Ce groupe est engendré par 4 réflexions $x_i$ (numérotées cycliquement), avec~5
possibilités pour les ordres de $x_ix_{i+1}$, en fonction du nombre et de la
position des selles : $(2,2,2,\infty)$,   $(2,2,\infty,\infty)$,
$(2,\infty,2,\infty)$, $(2,\infty,\infty,\infty)$ et
$(\infty,\infty,\infty,\infty)$. Pour réaliser $(0,2)^+$, il suffit d'avoir deux
axes de réflexions disjoints, ce qui évidemment toujours le cas. Ensuite, on
pourra réaliser $(1,1)^+$ (resp. $(0,3)^+$) si l'on trouve un quadrilatère bordé
par 4 axes disjoints (aucun ne séparant l'ensemble des 3 autres) $uvu'v'$ (resp.
$uu'vv'$) de sorte que les paires $\{u,u'\}$ et $\{v,v'\}$ soient identifiées
par des éléments directs sans points fixes de $G'_\gen$. On vérifie aisément
l'existence de tels quadrilatères dans chaque cas (illustration pour
$(2,\infty,2,\infty)$ sur la figure~\ref{figu:2i2i}).
\end{proof}

\begin{figure}[h]
\begin{center}
\labellist
\small\hair 2pt
\pinlabel $u$ at 175 160
\pinlabel $v$ at 75 160
\pinlabel $u'$ at 42 60
\pinlabel $v'$ at 200 60
\pinlabel $u$ at 555 180
\pinlabel $u'$ at 395 180
\pinlabel $v$ at 555 75
\pinlabel $v'$ at 395 75
\endlabellist
\includegraphics[scale=0.4]{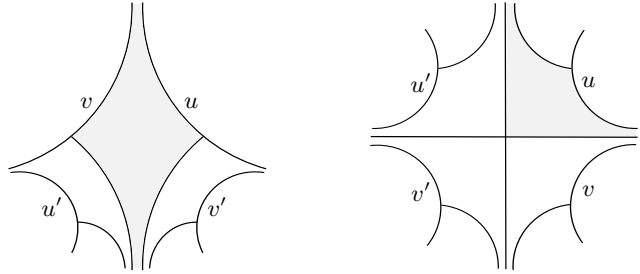}
\caption{Réalisation de $(1,1)^+$ et de $(0,3)^+$ pour $(2,\infty,2,\infty)$}
\label{figu:2i2i}
\end{center}
\end{figure}

\begin{rema}
Soit $f$ comme dans la proposition~\ref{prop:top_classe_fixee}. Si de plus $f$
possède
une symétrie, alors on vérifie que $E^u$ admet un quotient de type topologique
$(2,1)^-$ (bouteille de Klein épointé). On prend cette fois $S'$ invariante par
$\Isom(f)$, dont le cardinal dépend du type de symétrie : $\card S' = 4$ pour
(1a), $5$ pour (1b). Toujours d'après \cite[proposition~3.3]{EKS1984}, la
surface $E^u$ admet un quotient non orientable de type $(g,p)^-$ pour tout
$g\geq 2$ {\em pair} et tout $p\geq 1$. C'est le cas par exemple pour la
géométrie de Clifton-Pohl.
\end{rema}

\subsection{Ouverts de discontinuité}
\label{subs:ouv_dis} 
Soit $\Ga$ un sous-groupe de $\Isom(E^u)$. Un ouvert $U$ de $E^u$ stable par
$\Ga$ est {\em un ouvert de discontinuité pour $\Ga$} si l'action de $\Ga$ sur
$U$ est propre et discontinue. On dira que $\Ga$ est transverse au flot si
$\Ga\cap \Isom^0(E^u)$ est trivial ; il est clair que $\Gamma$ est alors un
sous-groupe discret de $\Isom(E^u)$. Rappelons que $E^u=E^u_f$ est de type fini
(définition~\ref{defi:type_fini}) si l'ensemble des composantes connexes de
$\{f\neq 0\}$ est localement fini.

\begin{prop}[actions propres et discontinues sur $E^u$]
\label{prop:action_propre_et_disc}
Soit $E^u$ de type fini (par exemple analytique) à courbure non constante et
soit $\Gamma$ un sous-groupe de $\Isom(E^u)$. Si $\Gamma$ est transverse au
flot, alors $\Gamma$ agit proprement et discontinûment sur $E^u$. De plus,
l'action de $\Gamma$ est libre si et seulement si $\Gamma$ est sans torsion.
\end{prop}

\begin{proof}
D'après la la preuve de la proposition~\ref{prop:action_Gprime}, lorsque
l'action d'un groupe transverse au flot n'est pas propre, alors les axes des
réflexions génériques de $E^u$ s'accumulent quelque part. La fonction $f$ étant
supposée de type fini, un tel comportement n'est pas possible. De plus, tout
point a un stabilisateur fini, donc trivial si $\Gamma$ est sans torsion.
\end{proof}

\begin{prop}[ouverts maximaux de discontinuité]
\label{prop:ouv_disc}
Soit $E^u$ de type fini, soit  $\Gamma$ un sous-groupe de $\Isom(E^u)$ et
soit~$U$ un ouvert de discontinuité pour $\Gamma$, connexe et  maximal.
Si $E^u$ n'est pas à courbure constante et si $U\neq \emptyset$, on a
l'alternative suivante : 
\begin{enumerate}
 \item ou bien $\Ga$ est transverse au flot et  $U=E^u$,
 \item ou bien $\Ga\cap\Isom^0(E^u)\simeq \Z$ et $U/\Gamma$ est 
 homéomorphe au cylindre, au ruban de Möbius, au tore ou à la bouteille de
 Klein.
 De plus $U$ est simplement connexe (en particulier $\Ga\simeq \pi_1(U/\Ga)$)
 et l'espace de ses feuilles est une géodésique maximale de $\mcE^u$. En particulier, 
 si $E^u$ n'est pas élémentaire, alors $U\neq E^u$.
\end{enumerate}
\end{prop}

\begin{proof}
Le groupe $\Gamma\cap \Isom^0$ est un groupe cyclique distingué dans $\Ga$, par
conséquent s'il n'est pas trivial, alors les feuilles du champ $K$ (induit par
$K^u$ sur $U/\Gamma$) sont toutes fermées.  La surface $U/\Gamma$ admet donc un
revêtement d'ordre au plus $2$ qui est un fibré en cercle et l'espace des
feuilles de $K$ est séparé. Or cet espace est revêtu par l'espace des feuilles
de la restriction de $K^u$ à $U$ et donc $U$ est contenu dans la préimage d'une
géodésique de $\mcE^u$. Si $\Gamma\cap \Isom^0$ est trivial, alors $\Gamma$
est transverse au flot et $U=E^u$ d'après la 
proposition~\ref{prop:action_propre_et_disc}.
\end{proof}

\begin{coro}%
\label{coro:cas_analytique}
Soit $(X,K)$ une surface lorentzienne connexe, analytique, satur\'ee, maximale
et sans selles \`a l'infini. Si flot de $K$ n'est pas périodique alors $X$ est
un quotient d'une surface $E^u$ associ\'ee \`a une certaine fonction analytique
inextensible.
\end{coro}

\begin{proof} Une surface $(X,K)$ analytique sans selles à l'infini est
uniformisée par un ouvert~$U$ d'une surface $E^u$ de type fini. Comme $X$ est
maximale, $U$ est un ouvert maximal de discontinuité pour l'action du groupe
d'holonomie.
\end{proof}

Quand $E^u$ n'est pas de type fini, il peut exister des groupes
d'isométries transverses au flot ayant des ouverts de discontinuité
connexes et maximaux différents de~$E^u$. 

\begin{exem}
\label{exam:type_infini}
Soit $f\in\mcC^\infty(\R,\R)$. On suppose que les composantes de $\{f\neq
0\}\cap \R^+$ s'accumulent uniquement sur $0$, de sorte que $g=f_{|]0,+\infty[}$
est de type fini.  Si $\Gamma$ est l'image d'une section de $\Isom(E^u_g)\to
\Isom(E^u_g)/\Isom^0(E^u_g)$, alors d'après la proposition
\ref{prop:action_Gprime} (ou \ref{prop:action_propre_et_disc}), $\Gamma$ agit
proprement et discontinûment sur $E^u_g$. Cependant, $E^u_g$ s'identifie à un
ouvert de $E^u_f$ et $\Isom(E^u_g)=\Isom_\gen(E^u_g)$ s'identifie à un
sous-groupe de $\Isom(E^u_f)$, lemme~\ref{lemm:Euf_reflexive}. Le groupe
$\Gamma$ contient alors des réflexions génériques dont les axes s'accumulent
dans $E^u_f$ ; il n'agit donc pas proprement sur $E^u_f$. En fait, $\Gamma$
n'agit proprement en aucun point du bord de $E^u_g$ dans $E^u_f$ : $E^u_g$ est
un ouvert de discontinuité maximal de $E^u_f$. On remarquera toutefois que la
vari\'et\'e quotient est de classe $[g]$ et non de classe $[f]$.
\end{exem}

\begin{prop}\label{prop:tores_loc_model}
Les tores localement modelés sur $E^u$ sont uniformisables par des ouverts
saturés simplement connexes de $E^u$ dont la projection sur $\mcE^u$ est
une géodésique. Ils possèdent un champ de Killing.  Il existe des tores
localement modelés sur $E^u$ si et seulement si $f$ est périodique.
\end{prop}

\begin{proof} 
Soit $T$ un tore localement modelé sur $E^u$ et soit $\mathcal D$ une
développante associée.  Si~$T$ est plat le r\'esultat est clair, on suppose donc
$T$ non plat.  Le champ de Killing $K^u$ se relève par $\mathcal D$ en un champ
de Killing $\widetilde K$.  Le feuilletage {\it a priori} singulier associé à
$\widetilde K$ est invariant par isom\'etrie, il induit donc un feuilletage
$\mathcal K$ sur $T$.  Par le théorème de Poincaré-Hopf, on voit que ce
feuilletage ne peut pas être singulier puisque ses seules singularités possibles
sont des selles. D'après la proposition~\ref{prop:struct_transv_killing}, ce
feuilletage est transversalement riemannien. Les niveaux de la courbure étant
saturés par ce feuilletage, on en déduit que toutes ses feuilles sont fermées.
Il est donc orientable et $\widetilde K$ passe donc au quotient. Par
cons\'equent, $f$ est p\'eriodique, $\mathcal D$ est injective (th\'eor\`eme
\ref{theo:uni_tore}) et son image se projette sur une g\'eod\'esique de
$\mcE^u$, proposition~\ref{prop:ouv_disc}. Inversement, si $f$ est
périodique, le ruban $R_f$ rev\^et un tore.
\end{proof}

Nous verrons (\S\ref{subs:class_tores}) que la géodésique de $\mcE^u$ associée à
l'uniformisation d'un tore $T$ modelé sur $E^u$ contient un invariant de nature
combinatoire : la suite périodique des bandes de type~II de $T$. Soit $f$
périodique de type fini et soit $G$ l'image d'une section de la suite
exacte~\eqref{equa:pi0_isom}. Le groupe engendré par $G^{K,+}$ et $\Phi_K^1$
agit discontin\^ument et librement sur $E^u$ privé de ses points selles. La
variété quotient $X$, non séparée sauf si $E^u$ est élémentaire, possède une
propriété remarquable : {\em  pour toute suite finie de bandes de type II
extraite de $R_f$, il existe un tore $T$ immergé dans $X$ ({\it via}
l'uniformisation) et admettant cette combinatoire prescrite.}

Nous étudions enfin l'espace des déformations des surfaces obtenues comme
quotient d'un ouvert de discontinuité de $E^u$.  Pour tout groupe $\Gamma$,
notons $\Hom(\Ga,\Isom(E^u))$ l'ensemble des représentations de $\Ga$ dans
$\Isom(E^u)$, muni de la topologie induite par $\Isom(E^u)^\Ga$. Comme
$\Isom(E^u)$ est localement connexe, le groupe quotient $\pizeroIsom(E^u) =
\Isom(E^u)/\Isom^0(E^u)$ est discret. On rappelle que $\pizeroIsom(E^u)$ est
isomorphe à $\Pi$ dans le cas non élémentaire et à $\Z/2\Z \times \Isom(f)$ dans
le cas élémentaire. Par ailleurs, la projection de $\Isom(E^u)$ sur
$\Isom(E^u)/\Isom^K(E^u)\simeq \mu_2$, où $\mu_2=\{\pm 1\}$, induit un caractère
$\overline{\chi} : \pizeroIsom(E^u) \to \mu_2$. Soit $p:\Isom(E^u) \to
\pizeroIsom(E^u)$. Pour tout $\ro\in \Hom(\Ga,\Isom(E^u))$, on pose 
$$
\overline{\ro} = p\circ \ro\in \Hom(\Ga,\pizeroIsom(E^u))
\esp \mathrm{et}\esp
\chi_\ro=\overline{\chi}\circ  \overline{\ro}\in \Hom(\Ga,\mu_2).
$$
Par continuité des projections de $\Hom(\Ga,\Isom(E^u))\subset \Isom(E^u)^\Ga$ 
sur les facteurs $\Isom(E^u)$, les morphismes $\overline{\ro}$ 
et $\chi_\ro$ ne dépendent que de la composante connexe $C_\ro$ de $\ro$
dans $\Hom(\Ga,\Isom(E^u))$. En particulier, le caractère $\chi_\ro$ définit une 
structure de $\Ga$-module sur $\R$ associée à $C_\ro$, que l'on notera simplement
$\R_\ro$.

\begin{prop}[composantes et espaces de déformations]\label{prop:composantes_deformations}
Soit $E^u$ à courbure non constante et soit~$\Ga$ un groupe. Soit $\ro\in \Hom(\Ga,\Isom(E^u))$
et soit $C_\ro$ la composante connexe de~$\ro$ dans $\Hom(\Ga,\Isom(E^u))$.
\begin{enumerate}
          \item\label{ass:bij} La composition par la projection de $\Isom(E^u)$ sur $\pizeroIsom(E^u)$
          induit une bijection de l'ensemble des composantes connexes de 
          $\Hom(\Ga,\Isom(E^u))$ sur $\Hom(\Ga,\pizeroIsom(E^u))$.
          \item\label{ass:homeo} L'espace des dérivations
          $\Der(\Ga,\R_\ro)$ agit simplement et transitivement sur la composante
          $C_\ro$. Celle-ci est homéomorphe à $\Der(\Ga,\R_\ro)$.
          \item \label{ass:h1} L'action des conjugaisons de $\Isom^0(E^u)$ sur $C_\ro$ 
          (déformations triviales) correspond à l'espace des dérivations principales
          $\Der^0(\Ga,\R_\ro)$. L'espace des déformations de $\ro$ (modulo conjugaison par
          $\Isom^0(E^u)$) est donc homéomorphe au $\R$-espace vectoriel $H^1(\Ga,\R_\ro)$.
          \item\label{ass:fidel} On suppose que $E^u$ est de type fini (par exemple analytique). 
          Si $\overline{\ro}$ est fidèle, alors toute représentation
          $\ro'\in C_\ro$ est fidèle et l'action de $\ro'(\Ga)$ sur $E^u$ est propre et discontinue ;
          de plus si $\overline{\ro}(\Ga)$ est sans torsion, alors $\ro'(\Ga)$ agit librement 
          sur $E^u$. 
          \end{enumerate}
\end{prop}

\begin{proof}
Nous avons vu que la composition par $p$ induit une application bien définie $\theta$ de
l'ensemble des composantes de $\Hom(\Ga,\Isom(E^u))$ dans $\Hom(\Ga,\pizeroIsom(E^u))$.
Comme $p$ est scindée (théorème~\ref{theo:groupe_isom_Eu} et proposition
\ref{prop:isom_Eu_elem}), l'application $\theta$ est surjective. Notons $\Phi^t$
le flot du champ~$K^u$, qui identifie $\Isom^0(E^u)$ avec $\R$.
Soit $\ro\in \Hom(\Ga,\Isom(E^u))$ une représentation. Tout élément 
$d\in \Der(\Ga,\R_\ro)$, c'est-à-dire vérifiant
$d(\ga \ga')= d(\ga) + \chi_\ro(\ga) d(\ga')$ pour $\ga,\ga' \in \Ga$, 
définit une représentation $\ro'\in C_\ro$ par 
\begin{equation}
\label{equa:derivation}
\ro'(\ga)=\Phi^{d(\ga)}\ro(\ga) \esp (\ga\in \Ga).
\end{equation}
Inversement, si $p\circ \ro'= p\circ \ro$, alors l'application $d:\Ga\to \R$ 
donnée par \eqref{equa:derivation} est l'unique dérivation qui envoie $\ro$ sur $\ro'$
(en particulier $\ro'\in C_\ro$ et $\theta$ est injective), d'où ({\ref{ass:bij}}) et ({\ref{ass:homeo}}). 
Le choix de $\ro\in C_\ro$ comme point base définit un homéomorphisme entre 
$C_\ro$ et $\Der(\Ga,\R_\ro)$. L'assertion ({\ref{ass:h1}}) résulte de la relation
$$\Phi^s \ro(\ga)\Phi^{-s}\ro(\ga)^{-1} = \Phi^{s-\chi_\ro(\ga) s}
\esp
(s\in \R, \ga\in \Ga).$$
Supposons que $\overline{\ro}$ est fidèle. Si $\ro'\in C_\ro$, on a $\overline{\ro}'=
\overline{\ro}$ : il est clair que $\ro'$ est fidèle et que $\ro'(\Ga)$ est transverse
au flot. Par suite (proposition~\ref{prop:action_propre_et_disc}), l'action
de $\ro'(\Ga)$ sur $E^u$ est propre et discontinue. Si $\overline{\ro}(\Ga)  =
\overline{\ro}'(\Ga)$ est sans torsion, alors $\ro'(\Ga)$ est également sans torsion,
son action sur $E^u$ est donc libre  (proposition~\ref{prop:action_propre_et_disc}).
\end{proof}

\begin{exem}
\label{exem:quot_Eu}
L'assertion ({\ref{ass:fidel}}) s'applique à tous les sous-groupes propres et
discontinus mis en évidence au \S\ref{subs:quotients_mini}. Notamment, comme
sous-produit de la proposition~\ref{prop:top_quo_Eu}, on obtient la
classification des composantes des sous-groupes $\Ga\subset \Isom^K(E^u)$
transverses au flot, sans torsion et maximal pour ces propriétés. Pour chaque
composante, la détermination de l'espace des déformations, c'est-à-dire ici
$H^1(\Ga,\R)$, est immédiate puisque la topologie de $E^u/\Ga$ est connue
par~\ref{prop:top_quo_Eu}. De plus, celle-ci est fixée dans chaque composante.
On voit donc qu'il existe de nombreuses surfaces portant un champ de Killing et
revêtues par $E^u$. Bien sûr, pour décrire les classes d'isométries, il faudrait
encore prendre en compte l'action à droite du \og groupe modulaire\fg\
$\Aut(\Ga)$. 
\end{exem}

\begin{exem}
 Soit $U$ un ouvert de discontinuité pour un sous-groupe $\Ga$ de $\Isom(E^u)$ comme
 dans la proposition~\ref{prop:ouv_disc}-(2). Si $U/\Ga$ est un tore, alors
 l'espace de déformations est $H^1(\Z^2,\R)=\R^2$ ; noter qu'il existe un 
 unipotent de $\GL_2(\Z) = \Aut(\Z^2)$ dont l'action stabilise cet espace.
 Contrairement à l'exemple~\ref{exem:quot_Eu}, la topologie des quotients n'est pas fixe 
 par déformations (on trouve aussi des anneaux). Dans le cas où $U/\Ga$ est une bouteille de Klein,
 on a $\Der(\Ga,\R_\ro)\simeq \R^2$ et $H^1(\Ga,\R_\ro)\simeq \R$ si $\Ga$ ne préserve
 pas le champ et  $\Der(\Ga,\R_\ro)=H^1((\Ga,\R_\ro)\simeq \R$ sinon ; 
 à nouveau, la topologie n'est pas fixe par déformation.
\end{exem}

\section{Surfaces compactes : classifications, points conjugués}
\label{sect:tores}
\subsection{Tores possédant un champ de Killing}
\label{subs:class_tores}

Soit $(T,K)$ un tore lorentzien non plat muni d'un champ de Killing $K$ non
trivial. On a vu dans la preuve du théorème~\ref{theo:uni_tore} que le flot de
$K$ est périodique, ce qui entraîne (existence d'une transversale fermée) que
l'espace des feuilles du champ $K$ est isométrique à $(\R/\msfm \Z, dx^2)$ pour
un certain $\msfm >0$, appel\'e masse de $K$. Afin d'identifier l'espace des
feuilles à $\R/\msfm \Z$, on choisit  un point  $\zeta$ et une  orientation
$\nu$ de cet espace. Le triplet  $(T,\nu,\zeta)$ sera désigné simplement par
\og{\em tore marqué}\fg\ (il est sous-entendu que $T$ possède un champ de
Killing). La fonction $\langle K, K\rangle$ se factorise par l'espace des
feuilles et le marquage définit des applications $f:\R\to \R$ et $\bar f :
\R/\msfm \Z\to \R$. On a vu aussi qu'il existe une développante $\mathcal D:
\widetilde T\rightarrow E^u_f$ (la proposition~\ref{prop:tores_loc_model} donne
la réciproque), que l'on peut construire grâce au marquage. \emph{Par abus de
langage nous appellerons ruban ou bande les parties de~$T$ dont un relevé au
revêtement universel est un ruban ou une bande.}

\par
Dans le cas riemannien, la mesure  du paramètre de twist (ou décalage) repose
sur l'existence de transversales globales lisses et naturelles : les géodésiques
orthogonales
au champ~$K$. Ici, il n'existe pas en général de telles transversales
\emph{lisses}. Notamment, si la fonction $\langle K, K\rangle$ s'annule,
aucune géodésique perpendiculaire à $K$ ne coupe toutes les feuilles de $K$ ; de
plus, si $T$ contient  des bandes de type II, aucune feuille de lumière maximale
n'est partout transverse à $K$. Pour pallier ce manque, nous commençons par
construire un feuilletage transverse $\mcK^\pitchfork$ naturel (du point de vue géométrique)
dont les feuilles ne sont que lisses par morceaux dès que  $T$ contient des
bandes de type II. 
\par

Si $\langle K, K\rangle$ ne s'annule pas, on choisit pour feuilletage
transverse $\mcK^\pitchfork$ le feuilletage orthogonal à $K$ ; noter que les deux
feuilletages de lumière sont transverses au champ, mais aucun n'est
privilégié. Si $\langle K, K\rangle$ s'annule et si $T$ ne contient que 
des bandes de type I, alors $\widetilde T$ est un ruban et on prend pour
feuilletage transverse $\mcK^\pitchfork$ l'unique feuilletage de lumière partout
transverse à~$K$.  
\par

Supposons maintenant que $T$ n'est pas un ruban. Le champ $K$ possédant une
transversale ferm\'ee, on voit que $T$ ne peut pas contenir de bande de
type~III. Il contient donc forc\'ement des bandes de type~II qui ne peuvent
s'accumuler dans $T$ car leurs  bords appartiennent à des feuilletages de
lumière différents.  Elles sont donc en nombre fini, nécessairement pair. On
note $\zeta_i$, $i$ modulo $2k$, les feuilles du champ $K$ situées au milieu des
bandes de type II (voir remarque~\ref{rema:feuille_milieu}) et $C_i$ les
cylindres fermés obtenus par découpage de $T$ suivant les $\zeta_i$.  Chaque
$C_i$ est un ruban à bord contenant au moins un zéro de $\langle K, K\rangle$,
donc admet un unique feuilletage de lumière transverse à $K$. Ces feuilletages
se recollent en un feuilletage de $T$ noté $\mcK^\pitchfork$, transverse à $K$ et dont les
feuilles sont des géodésiques de lumière brisées. Par construction, le lieu de
bifurcation des feuilles de $\mcK^\pitchfork$ est la réunion des $\zeta_i$, chaque feuille
étant formée de segments appartenant alternativement à l'un des deux
feuilletages de lumière de~$T$.

Si $T$ n'est plat, {\em le feuilletage transverse $\mcK^\pitchfork$ est
invariant par isométrie}. C'est immédiat si $T$ n'a pas de bandes de type II.
Sinon, cela résulte de l'invariance par isométrie de la décomposition en
cylindres par les milieux de ces bandes. Pour tout $p\in T$, on note $\ell_p$
l'unique feuille de $\mcK^\pitchfork$ passant par le point~$p$.  On paramètre
$\ell_p$ de fa\c con lisse par morceaux, directe (c.-à-d. compatible avec
l’orientation $\nu$ de l'espace des feuilles de $K$) et de sorte que
$\ell_p(0)=p$.

\begin{defi}
Soit $K$ un champ de Killing de $T$ de période $t_0>0$. On définit le
twist de $(T,\nu,\zeta)$ relativement à $K$ comme étant l'élément
$\tau$ de $\R/t_0\Z$ tel que $\Phi_K^\tau(p')=p$, où $p$ est un point de
$\zeta$ et $p'$ est le point de premier retour $\ell_p$ sur $\zeta$. Ce
nombre ne dépend pas du choix de $\zeta$, ni de celui de $p\in \zeta$.
\end{defi}

Le paramètre de twist $\tau$ est nul si et seulement si toute feuille de $\mcK^\pitchfork$ est fermée
et coupe chaque feuille de $K$ en un seul point. Plus généralement, le
rapport $\tau/t_0$ est rationnel si et seulement si toutes les feuilles de $\mcK^\pitchfork$ sont
fermées.

On dira qu'un tore muni d'un champ de Killing est \emph{élémentaire} s'il est
plat ou si son champ de Killing ne possède pas de trajectoires de lumière. Les
tores élémentaires sont en tout point semblables à des tores riemanniens, par
exemple dans le cas non plat ils sont déterminés par la fonction $\langle
K,K\rangle$, la période du flot et un paramètre de twist. {\em Dans la suite du
paragraphe \S\ref{subs:class_tores} nous n'étudierons  que les tores non
élémentaires}.
 
Sur tout tore marqué
$(T,\nu,\zeta)$ non élémentaire il existe un unique champ de Killing
$K_0$ de masse $1$ vérifiant $\langle K_0(p),\dot \ell_p(0)\rangle>0$ avec
$p\in\zeta$ (si $p$ correspond à un point singulier de $\ell_p$, alors
$\dot\ell_p(0)$ désigne la dérivée à droite).  Ce champ sera appelé
\emph{champ de Killing privilégié} de $(T,\nu,\zeta)$. Il dépend de
la position de $\zeta$ par rapport aux feuilles $\zeta_i$ (voir preuve du
lemme~\ref{prop:bouge_marques}).
 
 Soit $\bar f$ une fonction (lisse) de $\R/\Z$ dans $\R$. On appelle
 \emph{marquage pair} de $\bar f$ la donnée d'un sous-ensemble fini
 (éventuellement vide) de cardinal pair de l'ensemble des milieux des
 intervalles de $\R/\Z\smallsetminus \bar f^{-1}(0)$. On note $\mathcal
 C^{P}_0$ l'ensemble des fonctions lisses de $\R/\Z$ dans $\R$ non
 constantes, s'annulant et munies d'un marquage pair.  Deux triplets
 $(T,\nu,\zeta)$, $(T',\nu',\zeta')$ sont équivalents s'il existe une
 isométrie entre $T$ et $T'$ qui envoie $(\nu,\zeta)$ sur
 $(\nu',\zeta')$. On note $[T,\nu,\zeta]$ la classe d'équivalence 
(ou d'isométrie) du tore marqué $(T,\nu,\zeta)$.

\begin{prop}\label{prop:class_pointee} 
Il existe une bijection $\Theta$ entre les classes d'isométrie  de tores 
marqués $[T,\nu,\zeta]$ non élémentaires 
et $\R_+^*\times \R/\Z\times \mathcal  C^{P}_0$. 
 Plus précisément, $\Theta$ est définie par 
$\Theta([T,\nu,p])= \big(t_0,\tau/t_0,(\bar  f,\{x_1,\dots x_{2k}\})\big)$ où:
\begin{itemize}
\item $t_0$ est  la période (positive) du flot du champ de Killing
  privilégié $K_0$, 
\item $\tau$ est le paramètre de twist de $(T,\nu,p)$ relativement à
  $K_0$,
\item $\bar f$ est la fonction induite par $\langle K_0,K_0\rangle$ sur $\R/\Z$ via le marquage $(\nu,\zeta)$  et 
 $\{x_1,\dots, x_{2k}\}$ est l'ensemble des coordonnées dans $\R/\Z$ des milieux des bandes de type II.
\end{itemize}
\end{prop}

\begin{proof} 
 L'application $\Theta$ est bien définie car $(t_0,\tau/t_0,(\bar
 f,\{x_1,\dots x_{2k}\}))$ ne dépend que de la classe d'isométrie du
 triplet $(T,\nu,p)$ (une isométrie entre triplets envoie transversale sur
 transversale).
On commence par montrer que $\Theta$ est surjective.
On se donne donc
$\big(t_0,\tau/t_0,(\bar f,\{x_1,\dots x_{2k}\})\big)\in \R_+^*\times
\R/\Z\times\mathcal C^{0}_{P}$.  Soit $f$ le relevé à $\R$
de $\bar f$ et soit $E^u$ la surface associée. 
On oriente $\mathcal
E^u$, l'espace des feuilles non triviales du champ de Killing, et on le
pointe en $\zeta^u$ de coordonnée transverse $0$.
On choisit $q\in \zeta^u$ et un vecteur $v$ de type lumière en $q$ se
projetant sur un vecteur unitaire de $\mathcal
E^u$ orienté positivement. Il existe alors un unique champ de Killing $K^u$ induisant $f$ sur
$\mathcal E^u$ et tel que $\langle K^u,v\rangle =1$.

Soit $c:[0,1]\to E^u$ la géodésique de lumière brisée issue de $q$
vérifiant $\dot c(0)=v$,  se projetant sur une géodésique  unitaire et
directe de $\mathcal E^u$
et bifurquant aux temps $x_1, \dots, x_{2k}$.  On note $c_0$ le segment
$c([0,x_1])$, $c_i$ le segment $c([x_i,x_{i+1}])$ et $c_{2k}$ le segment
$c([x_{2k},1])$. Soient $\sigma_1,\dots, \sigma_{2k}$ les réflexions
génériques vérifiant $\sigma_i(c(x_i))=c(x_i)$ pour tout $i\in
\{1,\dots,2k\}$. On voit que $\hat
c=c_0.\sigma_1(c_1).(\sigma_2\circ\sigma_1)(c_2).\,\ldots\,
. (\sigma_{2k}\circ\cdots \circ \sigma_1)(c_{2k})$ est contenu dans un
ruban $R$. Le chemin $\hat c$ étant de longueur $1$, il existe donc une
isométrie directe $\varrho_R$ qui vérifie $\varrho_{R}(q)=\hat
c(1)$. Par conséquent, il existe $\Phi\in \Isom^{+,K}(E^u)$ tel que
$\Phi(q)=c(1)$. Cet élément est unique d'après la proposition
\ref{prop:class_isom_rubans}.

Soit $\gamma$ le chemin obtenu en mettant bout à bout les translatés de
$c$ sous l'action du groupe engendré par $\Phi$. Ce chemin est un relevé
d'une géodésique maximale de $\mcE^u$. 
On note~$U$ le saturé de $\gamma$
sous l'action du flot de $K^u$. L'ouvert $U$ est invariant sous l'action du
groupe engendré par le flot de $K^u$ et $\Phi$. On désigne par
$\Lambda_\tau$ le groupe engendré par $\Phi_{K^u}^{t_0}$ et
$\Phi_{K^u}^\tau\circ \Phi$.  Le quotient de $(U,\nu^u, \zeta^u)$ par
$\Lambda_\tau$ est bien un tore. Le chemin $\gamma$ se projette sur une
transversale $\ell_p$, grâce à quoi on vérifie que 
$\Theta([(U, \nu^u,  \zeta^u)/\Lambda_\tau])=(t_0,\tau/t_0,(\bar f,\{x_1,\dots
x_{2k}\}))$.

Il nous reste à voir que si $(t_0,\tau/t_0,(\bar f,\{x_1,\dots
x_{2k}\}))=\Theta([T,\nu,\zeta])$, alors le tore construit ci-dessus est
isométrique $(T,\nu,\zeta)$.  Soit $\mathcal D : \smash{\widetilde{T}}\to E^u_f$
une développante envoyant un relevé d'un point $p\in \zeta$ sur $q$, $K_0$
sur $K^u$ et préservant les orientations respectives des espaces des
feuilles. L'image par $\mathcal D$ de $\widetilde \ell_p$, le relevé au
revêtement universel de $\ell_p$, est forcément $\gamma$ et donc l'image
de $\mathcal D$ est $U$. Le groupe d'holonomie de $T$ est contenu dans
$\Isom^0(E^u)\times G$, il est engendré par $\Phi_K^{t_0}$ et un élément
qui laisse $U$ invariant et envoie $q$ sur $\Phi_{K^u}^{\tau}
\circ\gamma(1)$. Cet élément ne peut être que $\Phi_{K^u}^\tau\circ
\Phi$. Ce qui montre que le tore $(T,\nu, \zeta)$ est isométrique à $(U,
\nu^u,\zeta^u)/\Lambda_\tau$ et termine la preuve.
\end{proof}

\begin{rema}\label{rema:masse=1}
On peut choisir n'importe quel réel positif $\msfm$ comme masse du champ
privilégié. La proposition~\ref{prop:class_pointee} s'adapte alors sans
difficultés.
\end{rema}

Pour obtenir la classification des tores non élémentaires il suffit de voir ce
que deviennent les paramètres associés à $T$ lorsque $\nu$ est changé en $-\nu$
et lorsque $\zeta$ varie.

\begin{lemm}\label{prop:bouge_marques}
Soit $(T,\zeta,\nu)$ un tore marqué non élémentaire.  Si on pose
$\Theta ([T,\zeta,\nu]) = 
\big(t_0,\tau/t_0,( \bar f, \{x_1,\dots,x_{2k}\}))$,
alors $t_0$ ne dépend pas du choix de $(\zeta,\nu)$ et
\begin{enumerate}
 \item lorsque $\nu$ est remplacé par $-\nu$,   
  $\tau$ devient $-\tau$,
  $\bar f$ devient $\bar f^\vee$ et 
   $\{x_1,\dots,x_{2k}\}$ devient $\{1 -x_{2k},\dots, 1-x_1\}$,
 \item lorsque $\zeta$ est remplacé par $\zeta'$ de coordonnée $y$ pour
  le marquage $(\zeta,\nu)$,
  la fonction $\bar f$ devient $\bar f(\cdot+y)$,
  $\{x_1,\dots,x_{2k}\}$ devient $\{x_1-y,\dots,x_{2k}-y\}$ et 
  $\tau$ devient $(-1)^i\tau$ 
  où $i\in \{1,\dots,2k\}$ vérifie $y\in [x_i,x_{i+1}[$
($\R/\Z$ étant cycliquement orienté).
\end{enumerate}
 \end{lemm}
 
 \begin{proof}  
Le seul point non évident est le comportement du twist lorsque $\zeta$ est
déplacé. Il s'explique simplement par le fait que lorsque $\zeta$
franchit un point de bifurcation de $\ell$ alors $K_0$ est transformé
en $-K_0$, ce qui inverse le twist.
 \end{proof}

\begin{coro}\label{coro:class_surf}
Les classes d'isométrie de tores lisses possédant un champ de Killing et
non élémentaires sont en bijection avec le quotient de $\R_+^*\times
\R/\Z\times \mathcal C^{P}_0$ par la relation d'équivalence définie au
lemme~\ref{prop:bouge_marques}.
\end{coro}

\begin{rema}
On voit que contrairement au cas riemannien, le sens du twist importe.
Soit $\Gamma_{\epsilon}$ ($\epsilon=\pm 1$) le sous-groupe des isométries 
de $R= (\R^2, 2dxdy+\cos(x)dy^2)$ engendré par 
$\gamma_0(x,y)=(x,y+1)$ et $\gamma_\epsilon(x,y)=(x+2\pi, y+\epsilon/4)$.
Le lecteur pourra vérifier que $\Gamma_1$ et $\Gamma_{-1}$ sont 
distingués dans $\Isom(R)$ (voir proposition~\ref{prop:class_isom_rubans});
les tores associés ne sont pas isométriques.
\end{rema}

\subsection{Bouteilles de Klein  possédant un feuilletage de Killing}
\label{subs:bout_killing}

On s’intéresse maintenant aux bouteilles de Klein lorentziennes $B$ non plates
dont le revêtement universel possède un champ de Killing invariant $\widetilde
K$. Celui-ci est  préservé ou retourné par isométrie.  Par conséquent,
$\widetilde K$ ne passe pas forcément au quotient, mais il induit toujours un
\og feuilletage de Killing\fg\ $\mathcal K$.  On est donc en présence de deux
types de bouteilles : celles n'admettant pas de  champ de Killing mais seulement
un feuilletage $\mcK$  (type~1) et celles admettant un champ de Killing (type 2,
$\mcK$ orientable). De plus, dans les deux cas, la fonction $\langle \widetilde
K, \widetilde K\rangle$ passe au quotient et induit une fonction constante le
long des feuilles de $\mcK$. 

Le revêtement d'orientation d'une bouteille $B$ de type~1 ou~2 est un tore
lorentzien $T$ admettant un champ de Killing (voir preuve de la proposition
\ref{prop:tores_loc_model}) et une isométrie indirecte involutive~$\delta$ telle
que $T/\langle \delta\rangle=B$. Ainsi toutes les feuilles de $\mathcal K$ sont
fermées.  On reprend les conventions et les notations du~\S
\ref{subs:class_tores} concernant les tores marqués. On dira que $B$ est non
élémentaire si~$T$ l'est.

\begin{prop}
\label{prop:isom_indirecte}
Soit $(T,\nu,\zeta)$ un tore marqué non plat,  muni d'un champ de Killing~$K$
(non trivial) de masse $2\msfm$ et possédant une isométrie involutive indirecte
$\delta$ sans points fixes. Alors le paramètre de twist de $(T,\nu,\zeta)$ est
nul. De plus 
\begin{enumerate}
\item si $\delta_*K=-K$ (type~1), alors $\delta$ induit une rotation de longueur
$\msfm$ de l'espace des feuilles de~$K$ et, si $T$ n'est pas élémentaire, le
nombre  de bandes de type II de $T$ est congru à 2 modulo 4,
\item si $\delta_*K=K$ (type 2), alors $\delta$ induit une réflexion sur
  l'espace des feuilles de $K$, dont les deux points fixes  $\zeta_0$ et
  $\zeta_1$ sont au  milieu d'une bande de type II dès que 
$T$ est non élémentaire.
\end{enumerate}
\end{prop}

Autrement dit, pour le type~1, la fonction $\bar f$ définie par le marquage
est $\msfm$-périodique;  si $\bar f$ s'annule, il existe $h\in N$
tel que   les milieux des bandes de
type II ont pour
coordonnées $x_1,\dots,x_{2h+1},\msfm+x_1,\dots, \msfm+x_{2h+1}$.
Pour le type~2, en choisissant $\zeta= \zeta_0$, on a $\bar f=\bar f^\vee$,
$\zeta_0$ et $\zeta_1$ ont pour coordonnées $0$ et $\msfm$, et si $\bar f$
s'annule, les milieux des bandes de type II ont pour coordonnées
$0,x_2,\dots, x_{k-1},\msfm,2\msfm-x_{k-1},\dots,2\msfm-x_2$ ($f$ est de
type (3b) ou (3c) dans la terminologie du~\S\ref{subs:quotients_mini}).

\begin{proof} Observons d'abord  que $\delta$ doit permuter les feuilletages de
lumière de $T$ tout en laissant invariant le feuilletage transverse
$\mcK^\pitchfork$. Par
suite, ou bien $T$ est élémentaire, ou bien~$T$ admet des bandes de type II. 
Dans ce dernier cas, on note $K_0$ le champ privilégié de $(T,\nu,\zeta)$ et
$K_1$ celui de $(T,\delta^*\nu,\delta(\zeta))$ ; comme $\delta^*\nu=\nu$ si et
seulement si $\delta_*K=-K$, on a donc toujours $K_0=-K_1$.

Si $\delta_*K=-K$, alors $\delta^*\nu=\nu$ et $\delta$ induit une isométrie
directe sur l'espace des feuilles, c'est-à-dire une rotation. Comme
$\delta^2=\Id$, il s'agit d'une rotation de longueur $d=0$ ou $\msfm$.  Si
$d=0$, alors $\delta$ induit sur chaque feuille de $K$ une isométrie
indirecte et donc fixe des points.  D'autre part, il existe $p\in\zeta$ tel
que $\delta(p)\in \ell_p$ (car $p$ et $\delta(p)$ varient en sens inverse
par rapport à $K$). On en déduit que $\ell_p$ est invariante par $\delta$,
fermée et qu'elle ne coupe qu'une fois $\zeta$ : le twist est nul. La
fonction $\bar f$ définie par le marquage est $\msfm$-périodique. Si $\bar
f$ s'annule, les coordonnées des milieux des bandes de type II s'écrivent
$x_1,\dots,x_{k},\msfm+x_1,\dots, \msfm+x_{k}$. L'invariance de
$\mcK^\pitchfork$ par $\delta$ (ou $K_1=(-1)^k K_0$) impose $k$ impair.

Si $\delta_*K=K$, alors $\delta$ ne préserve pas $\nu$ et induit une
isométrie indirecte sur l'espace des feuilles, c'est-à-dire une réflexion
laissant fixe deux feuilles $\zeta_0$ et $\zeta_1$ à distance $\msfm$.
Dans le cas non élémentaire, ces feuilles sont nécessairement au milieu de
bandes de types II car $\delta(\mcK^\pitchfork)=\mcK^\pitchfork$ (inversion
des feuilletages de
lumière par $\delta$).  Le long de $\zeta_0$ et $\zeta_1$, $\delta$ induit
une rotation de longueur $t_0/2$ et donc $\sig= \phi_K^{t_0/2}\circ \delta$
est une réflexion de $T$ fixant point par point $\zeta_0$ et $\zeta_1$.
Chaque feuille de $\mcK^\pitchfork$ est donc fermée, constituée d'un arc de $\zeta_0$ à
$\zeta_1$ et de son image par~$\sig$, ce qui force le twist à être nul.
\end{proof}

Il est maintenant très facile de décrire les bouteilles de Klein
élémentaires. On peut supposer que la masse de l'espace des feuilles
est égale à $1$.  Leur revêtement universel est donc isométrique
à $(\R^2, \pm dx^2 +f(x) dy^2)$ où $f$ est $2$-périodique et ne
s'annule pas. On note $\gamma_1$ et $\gamma_2$ les générateurs
usuels du groupe fondamental.  Si la bouteille est de type~1 alors $f$
est $1$-périodique et il existe $t_0>0$ tel que $\gamma_1(x,y)=(x+1,-y)$
et $\gamma_2(x,y)=(x,y+t_0)$. Si la bouteille et de type $2$ alors on peut
supposer que $f$ est paire et il existe $t_0>0$ tel que
$\gamma_1(x,y)=(-x,y+t_0/2)$ et $\gamma_2(x,y)=(x+2,y)$.

Soit $B$ une bouteille de type~1 non élémentaire et soit $\mcK$ son \og
feuilletage de Killing\fg. On déduit de la proposition~\ref{prop:isom_indirecte}
que l'espace des feuilles de $\mathcal K$ muni de la métrique riemannienne
habituelle est isométrique $\R/\msfm\Z$.  Pour réaliser cette identification, on
munit cet espace d'une orientation et d'une origine, notées encore $\nu$ et
$\zeta$, et on étudie les {\em bouteilles marquées} $(B,\nu,\zeta)$.  On dira
que le tore marqué $(T,\nu',\zeta')$ est un revêtement d'orientation marqué de
$(B,\nu,\zeta)$ si $T$ est un revêtement à $2$ feuillets de $B$ et que ce
revêtement envoie $(\nu',\zeta')$ sur $(\nu,\zeta)$.

Soit $\bar f$ une fonction lisse de $\R/\Z$ dans $\R$ qui s'annule. On appelle
\emph{marquage impair} de $\bar f$ la donnée d'un sous-ensemble fini de cardinal
impair de l'ensemble des milieux des intervalles de $\R/\Z\smallsetminus \bar
f^{-1}(0)$. On note $\mathcal C^{I}_0$ l'ensemble des fonctions lisses non
constantes de $\R/\Z$ dans~$\R$, s'annulant et munies d'un marquage impair.

\begin{prop}\label{prop:bout1_marquees} 
Il existe une bijection $\Xi_1$ entre les classes d'isométrie de bouteilles de
Klein de type~1, non élémentaires et marquées $(B,\nu,\zeta)$ et $\R_+^*\times
\mathcal C^{I}_0 $. Plus précisément $\Xi_1([B, \nu, \zeta])=(t_0,(\bar f,
\{x_1,\dots, x_{2h+1}\}))$  où
\begin{itemize}
\item
le réel $t_0$ est  la période du  champ privilégié d'un revêtement d'orientation
marqué de~$B$, 
 \item la fonction  $\bar f$ est induite par  $\langle \pm K_0,\pm K_0\rangle$
 sur $\R/\Z$ via le marquage et $\{x_1,\dots, x_{2h+1}\}$ est l'ensemble
 des coordonnées des milieux des bandes de type II. 
 \end{itemize}
\vspace{.2cm}
\end{prop}

\begin{proof} 
D'après la proposition~\ref{prop:isom_indirecte}, si
$\Xi_1[B,\nu,\zeta]=(t_0,(\bar f, \{x_1,\dots, x_{2h+1}\}))$, alors tout 
revêtement d'orientation marqué de $(B,\nu,\zeta)$ appartient à
$\Theta^{-1}(t_0,0,(\hat f, A))$ où où $\hat f$ est la fonction sur $\R/2\Z$
induite par $f$ et $A= \{x_1,\dots, x_{2h+1},1+x_1, \dots,1+x_{2h+1} \}$ (on
prend $\Theta$ définie à partir des champs privilégiés de masse $2$, voir
remarque \ref{rema:masse=1}).  Pour montrer la proposition, il suffit d'établir
qu'un tel tore est le revêtement d'une bouteille de Klein de type~1, unique à
isométrie près. On se place à nouveau dans  $E^u_f$ transversalement orienté et
pointé en $\zeta^u$ de coordonnées $0$ dans les rubans qui le contiennent et on
choisit un champ de Killing $K^u$. 

Soit $c:[0,2]\to E^u$ une géodésique de lumière brisée issue de $q\in \zeta^u$,
se projetant sur une géodésique unitaire directe de $\mathcal E^u$ et bifurquant
aux temps $x_1, \dots, x_{2h+1},1+x_1, \dots,1+x_{2h+1}$. La fonction $f$ étant
$1$-périodique, on voit, par un raisonnement proche de celui utilisé pour
montrer la proposition~\ref{prop:class_pointee}, qu'il existe une unique
isométrie $\Psi$ ne préservant ni $K^u$, ni l'orientation qui  envoie $c([0,1])$
sur $c([1,2])$.  Soit $\gamma$ le chemin  obtenu en mettant bout à bout les
translatés de $c$ sous l'action du groupe engendré par $\Psi^2$ et soit $U$ le
saturé de $\gamma$ sous l'action du flot de $K^u$. Le quotient de
$(U,\nu^u,\zeta^u)$ par le groupe engendré par $\Phi^{t_0}_{K^u}$ et $\Psi^2$
est un tore appartenant à  $\Theta^{-1}(t_0,0,(\smash{\hat f},A))$. 
L'application $\Psi$ induit sur $T$ une involution isométrique indirecte 
$\delta$ sans points fixes et  ne préservant pas le champ. 

Soient $\delta$ et $\delta'$  deux isométries involutives indirectes sans
points fixes  de $(T,\nu',\zeta')$ envoyant~$K$ sur $-K$. Il est clair que
$\delta=\delta'$ si et seulement si elles coïncident en un point. Par conséquent
il existe $t\in \R$ tel que $\delta'=\Phi_K^t\circ \delta$ et donc  $\delta$ et
$\delta'$ sont conjuguées par \smash{$\Phi^{t/2}_K$}. Ainsi
$(T,\nu',\zeta')/\langle \delta\rangle$ et $(T,\nu',\zeta')/\langle
\delta'\rangle$ sont isométriques.
\end{proof}

La classification est déduite de la proposition~\ref{prop:bout1_marquees} par le
même raisonnement que pour les tores.

\begin{lemm}\label{prop:bouge_marques_bout1}
 Soient $(t_0, (\bar f,  \{x_1,\dots,x_{2h+1}\}))$ les données associées à la bouteille de type~1 transversalement orientée et pointée $(B, \nu, \zeta)$. 
\begin{enumerate}
\item Lorsque $\nu$ est remplacé par $-\nu$ alors  
  $\bar f$ est remplacé par $\bar f^\vee$ et  
  $\{x_1,\dots,x_{2h+1}\}$ devient $\{1 -x_{2h+1},\dots,1-x_1\}$.
\item Lorsque $\zeta$ est remplacé par $\zeta'$ de coordonnée $y$, 
  $\bar f$ est remplacé par $x\mapsto \bar f(x+y)$ et
  $\{x_1,\dots,x_{2h+1}\}$ devient $\{x_1-y,\dots,x_{2h+1}-y\}$.
\end{enumerate}
 \end{lemm}
 \begin{coro}\label{coro:class_bout1}
  Les classes d'isométrie de bouteilles de type~1 non élémentaires 
  sont en bijection avec le quotient de $\R_+^*\times \mathcal C^{I}_0$ 
  par l'action de $\R/\Z\rtimes \Z/2\Z$ définie à la
proposition~\ref{prop:bouge_marques_bout1}.
 \end{coro}

 On regarde maintenant les bouteilles de type 2 non élémentaires.   D'après la
proposition~\ref{prop:isom_indirecte}, l'espace des feuilles de $\mathcal K$ 
d'une  telle bouteille est isométrique à un intervalle $[0,\msfm]$. Pour
réaliser cette identification, il suffit de choisir une des deux feuilles
courtes de $K$ que l'on note  $\zeta_0$. On travaille donc avec les paires  $(B,
\zeta_0)$. Il existe à nouveau un champ de Killing privilégié $K_0$ tel
que $\msfm=1$ et dont les orbites au bord de  la bande associée à
$\zeta_0$ sont dirigées vers un sommet de type selle. On dira que
$(T,\nu,\zeta)$ est un revêtement d'orientation marqué de $(B,\zeta_0)$ si c'est
un revêtement à deux feuillets envoyant~$\zeta$ sur $\zeta_0$.
 
 Soit $\bar f$ une fonction (lisse) \emph{paire} de $\R/2\Z$ dans $\R$
s'annulant. On appelle \emph{marquage symétrique} de $\bar f$ la donnée d'un
sous-ensemble fini de l'ensemble des milieux des intervalles de
$\R/2\Z\smallsetminus \bar f^{-1}(0)$ qui est invariant par la symétrie
$x\mapsto -x$ et qui contient $0$ et $1$. On note $\mathcal C^{S}_0$ l'ensemble
des fonctions lisses de $\R/2\Z$ dans $\R$ qui sont paires, s'annulent,
vérifient $\bar f(\bar 0)\bar f(\bar 1)\neq 0$  et sont munies d'un marquage
symétrique. 
 
 \begin{prop}\label{prop:bout2_marquees}
 Il existe une bijection $\Xi_2$ entre les classes d'isométrie de bouteilles de
Klein de type 2 marquées, non élémentaires  et $\R_+^*\times  C^{S}_0$.

Plus précisément $\Xi_2(B,  \zeta_0)=(t_0,(\bar f, \{0,x_2, \dots, x_{j-1},1
\}))$ où
 \begin{itemize}
  \item $t_0>0$ est la période du flot de $K_0$, 
  \item   $\bar f$ est la fonction (paire) induite par  $\langle  K_0,
K_0\rangle$ sur le revêtement d'orientation de $B$ et 
$\{0,x_2,\dots,x_{j-1},1\}$ sont les coordonnées des  milieux des bandes de type
II de $B$.
 \end{itemize}
\end{prop}

\begin{proof} 
D'après la proposition~\ref{prop:isom_indirecte}, si
$\Xi_2[B,\zeta_0]=(t_0,(\bar f, \{0,x_2,\dots,x_{j-1}, 1\}))$, alors tout
revêtement d'orientation marqué de $(B,\zeta_0)$ appartient à
$\Theta^{-1}(t_0,0,(\bar f, A))$ avec
$A= \{0, x_2,\dots,1,\dots, 2-x_2\}.$
Pour montrer la proposition, il suffit de montrer qu'un tel tore  est le
revêtement d'orientation marqué  d'une unique bouteille de Klein de type 2
marquée. On se place dans l'espace $E^u$ transversalement pointé en $\zeta^u$ de
coordonnées $0$ dans les rubans qui le contiennent (on rappelle que $\zeta$ est
au milieu d'une bande). 

On choisit $q\in \zeta^u$ et un champ de Killing $K^u$.  On marque les deux
côtés du carré contenant $\zeta^u$ qui contiennent le sommet de type
source.  Soit $\Psi$ la réflexion non générique fixant $\zeta^u$. On
remarque que $\Psi$ permute les 2 côtés marqués précédemment.  Il existe
une unique géodésique de lumière brisée $c:[-1,1]\to E^u$ se projetant
sur une courbe unitaire directe de $\mathcal E^u$, vérifiant $c(0)=q$,
intersectant les côtés marqués et bifurquant aux temps $-x_{j-1}, \dots,
0,\dots,x_{j-1}$.  Elle est invariante par $\Psi$. Soit $\Phi$ l'unique
élément de $\Isom^{+,K}(E^u)$ envoyant $c(-1)$ sur $c(1)$.

On note $\gamma$ le chemin obtenu en mettant bout à bout les translatés de $c$
sous l'action de~$\Phi$, il bifurque forcément en $c(\pm 1)$. Soit $U$ le saturé
de $\gamma$ par $\Phi_{K^u}$. Le quotient de $U$ par le groupe engendré par
$\Phi^{t_0}_{K^u}$ et $\Phi$ est un tore appartenant à 
$\Theta^{-1}(t_0,0,(\bar f, A))$. L'application
$\Phi^{t_0/2}_{K^u}\circ \Psi$ induit sur ce tore une isométrie involutive (son
carré vaut car $\Phi^{t_0}_{K^u}$), indirecte, sans points fixes et préservant
$K$.

Soit $(T,\nu',\zeta')$ un revêtement d'orientation marqué d'une bouteille
$(B,\zeta_0)$ de type 2  et soit~$\delta$ telle que $B=T/\langle \delta\rangle$.
L'application $\delta$ est l'unique isométrie vérifiant $\delta(\zeta')=\zeta'$,
$\delta_*K=K$ et induisant une symétrie sur l'espace des feuilles de $K$. Par
conséquent, $(B,\zeta_0)$ est uniquement déterminée.
\end{proof}

La classification est déduite de la proposition~\ref{prop:bout2_marquees} 
comme  précédemment.

\begin{lemm}\label{prop:bouge_marques_bout2}
Soient $(t_0, (\bar f, \{0,\dots,1\}))$ les données associées à la bouteille de
type 2 marquée $(B, \zeta_0)$.
Lorsque $ \zeta_0$ est remplacé par la feuille $\zeta_1$ de coordonnée $1$, 
$\bar f$ est remplacée par $x\mapsto \bar f(1-x)$ et 
$\{0,x_2,\dots,x_{j-1},1\}$ devient $\{0,1 - x_{j-1},\dots,1 - x_{2},1\}$.
 \end{lemm}

\begin{coro}\label{coro:class_bout2}
  Les classes d'isométrie de bouteilles de type 2  non élémentaires 
   sont en bijection avec le quotient de $\R_+^*\times\mathcal C^{S}_0$ 
   par l'action de $\Z/2\Z$ définie à la proposition
   \ref{prop:bouge_marques_bout2}.
\end{coro}

 \subsection{Composantes de l'espace des métriques}\label{sub:compos}
On s’intéresse ici aux composantes connexes de l'espace des tores et des
bouteilles de Klein contenant des surfaces modelées sur un espace $E^u_f$, avec
$f$ donnée. De façon générale, on sait que les composantes de l'espace des
métriques lorentziennes sur une variété correspondent bijectivement, {\it via}
le choix d'un champ de droites de type temps par m\'etrique, aux classes
d'homotopie de sections du fibré projectif tangent, 
\cite[\textsection~40]{Steenrod1951}.
\par 

Pour les surfaces compactes (tores et bouteilles de Klein), les composantes
de l'espace des métriques sont caractérisées par le comportement des cônes
(positifs ou négatifs) le long de certaines courbes.  Il s'agira ici de
feuilles de $\mcK$ (feuilletage associé au champ de Killing) ou de feuilles
du feuilletage transverse $\mcK^\pitchfork$. Au-dessus d'une courbe d'un
type donné (temps, espace ou lumière, par exemple une feuille de $\mcK$),
le comportement des cônes est trivial puisqu'ils ne peuvent pas changer de
côté par rapport à la courbe. Il n'en est pas de même au-dessus d'une
géodésique de lumière brisée $\ell$, par exemple une feuille de
$\mcK^\pitchfork$ en présence de bandes de type II. On convient qu'une
telle géodésique $\ell$ est toujours localement injective, ce qui signifie
qu'en chaque point de bifurcation $p\in \ell$, le feuilletage de lumière
local qui porte~$\ell$ change (pour éviter les aller-retours).  Cela permet
d'attribuer un signe à $p$ : on dira que {\em le point de bifurcation $p\in
  \ell$ est positif (resp. négatif) si $\ell$ traverse une direction 
  positive (resp. négative) en $p$}; les directions de signe opposé ne sont
pas traversées par $\ell$.

\begin{defi}
\label{defi:suite_reduite}
  Soit $\sig=((-1)^{s_1},\ldots,(-1)^{s_N})$  une suite
finie de signes ($s_j=0$ ou~1). On  appelle {\em suite réduite associée à
  $\sig$} la suite de signes alternés obtenue en effaçant le premier 
couple de signes successifs égaux, puis en itérant cette opération.
\end{defi}

\begin{lemm}[enroulement des cônes au-dessus d'une courbe de lumière 
brisée]
\label{lemm:nb_enroul_transv}
Soit $X$ une surface lorentzienne et soit $\ga:[0,1]\to X$ une géodésique
de lumière brisée, localement injective. On suppose que $\ga([0,1])$ admet
$N\geq 1$ points de bifurcation, appartenant à $\ga(]0,1[)$.  Soient enfin
      deux sections continues $\xi^+$ et $\xi^-$ du fibré trivial (orienté)
      $\ga^*(TX)$, respectivement positive et négative pour la
      métrique. Alors, au signe près, le nombre d'enroulement (compté en
      tours) du repère $(\xi^+,\xi^-)$ le long de $\ga$ vaut
\begin{equation}
\label{equa:nb_enroul_transv}
\frac{1}{4}\sum_{j=1}^N (-1)^{j+s_j} = \pm \frac{M}{4},
\end{equation}
où $(-1)^{s_1},\ldots,(-1)^{s_N}$ est la suite des signes des points de
bifurcation, dans l'ordre du paramétrage, et $M$ désigne la longueur de la
suite réduite associée.
\end{lemm}

\begin{figure}[h!]
\labellist
\small\hair 2pt
\pinlabel $1$ at 105 147
\pinlabel $2$ at 105 247
\pinlabel $3$ at 365 247
\pinlabel $4$ at 365 147
\pinlabel $5$ at 235 147
\pinlabel $6$ at 235 277
\pinlabel $7$ at 490 277
\pinlabel $8$ at 490 182
\pinlabel $1$ at 64 25
\pinlabel $-$ at 64 54
\pinlabel $2$ at 128 25
\pinlabel $-$ at 128 54
\pinlabel $3$ at 192 25
\pinlabel $+$ at 192 54
\pinlabel $4$ at 256 25
\pinlabel $-$ at 256 54
\pinlabel $5$ at 320 25
\pinlabel $+$ at 320 54
\pinlabel $6$ at 384 25
\pinlabel $-$ at 384 54
\pinlabel $7$ at 448 25
\pinlabel $+$ at 448 54
\pinlabel $8$ at 512 25
\pinlabel $+$ at 512 54
\pinlabel $-$ at 20 163
\pinlabel $+$ at 48 163
\pinlabel $-$ at 20 68
\endlabellist
\begin{center}
\includegraphics[scale=0.62]{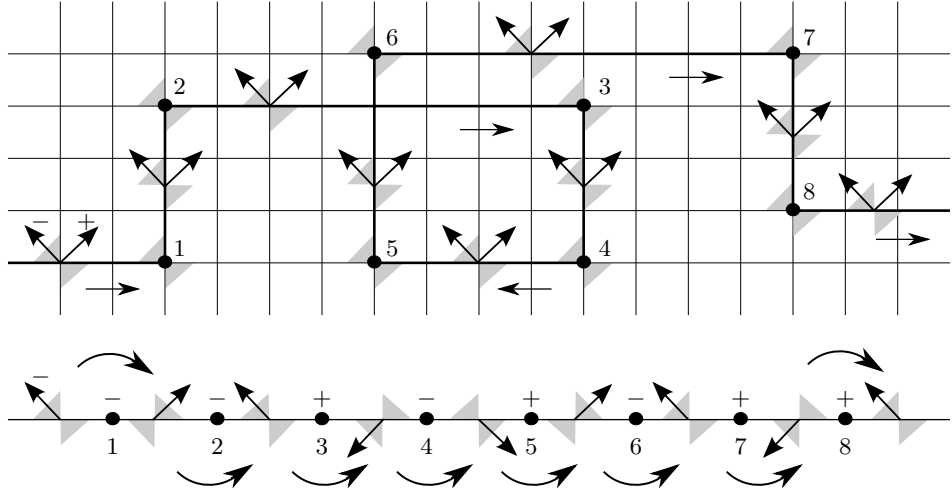}
\caption{Enroulement des cônes}
\label{figu:enroul_transv}
\end{center}
\end{figure} 

\begin{proof}
Notons $p_1,\ldots,p_N$ la suite ordonnée des points de bifurcations le
long de $\ga$. À chaque passage d'un point $p_j$, le cône du signe opposé à
$p_j$ change de côté par rapport à la feuille (plus précisément la section
du signe opposé à $p_j$ change de côté et l'autre non), le repère
$(\xi^+,\xi^-)$ fait donc un quart de tour. De plus, entre $p_j$ et
$p_{j+1}$, le sens de ce quart de tour change si $p_j$ et $p_{j+1}$ sont de
même signe, et ne change pas sinon. Ce comportement est illustré à la
figure~\ref{figu:enroul_transv} (cône négatif en gris).  Par une récurrence
immédiate, on voit que le sens $\eps_j$ du quart de tour au passage de $p_j$,
en orientant $\ga^*(TX)$ pour que $\eps_1=1$, vaut $\eps_j =
(-1)^{s_1+s_j+j-1}$. Par suite, le nombre d'enroulement cherché est donné
par $\frac{1}{4}\sum_{j=1}^N\eps_j = \frac{1}{4} (-1)^{s_1+1} \sum_{j=1}^N
(-1)^{j+s_j}$. Enfin, l'expression de l'enroulement en fonction de~$M$
découle directement du comportement des cônes décrit ci-dessus.
\end{proof}

Sur le tore, tout champ de droite de type temps est homotope à chacun des champs
de lumière.  L'appartenance à une composante connexe est donc déterminée par la
classe d'homotopie des feuilletages de lumière, c'est-à-dire par le nombre et le
sens de leurs composantes de Reeb ainsi que par la classe d'homotopie de leurs
feuilles compactes (voir par exemple \cite{BM} pour plus de détails). Soit
$\mcL$ un feuilletage de lumière d'un tore $(T,K)$. Toute orientation de la
classe d'isotopie des feuilles compactes de $\mcL$ induit une orientation de
$\mcL$ au voisinage de ces feuilles, ce qui permet de distinguer les composantes
de Reeb positives ou négatives suivant que leurs feuilles, parcourues dans le
sens direct donné par l'orientation locale, s'accumulent au bord ou non. Ces
composantes sont stables par le flot de $K$, donc bordées par des orbites de
celui-ci; s'il en existe, on convient d'orienter la classe des feuilles
compactes de $\mcL$ dans le sens du champ $K$.  Si~$r^+(\mcL)$ (resp. 
$r^-(\mcL)$) désigne le nombre de composantes de Reeb positives (resp.
négatives) de $\mcL$, l'entier naturel $r(T)=|r^+(\mcL) - r^-(\mcL)|$
caractérise la composante de la métrique modulo l'action des difféomorphismes. 
\par

À chaque composante de Reeb d'un feuilletage de lumière de $(T,K)$ correspond un
ruban maximal $R$ de $(T,K)$, nécessairement bordé par des bandes de type II,
mais la réciproque n'est pas vraie. Cela dépend de la position relative de $K$
et des bandes bordant $R$. Les termes {\em ruban}  et {\em bande}  de $(T,K)$
sont ici utilisés avec la même convention qu'au~\S~\ref{subs:class_tores}. On
rappelle qu'une bande de type II est une composante de Reeb du feuilletage
orthogonal à $K$, voir lemme ~\ref{lemm:types_bandes}.

\begin{defi}
\label{defi:signe_bandeII}
On dira qu'une bande de type II munie d'un champ $K$ est {\em positive} si c'est
une composante de Reeb positive du feuilletage orthogonal à $K$ (au sens
précédent, le bord étant orienté par $K$), et qu'elle est {\em négative} sinon.
\end{defi}

Les orbites du champ qui bordent une bande de type II positive (resp. négative)
sont dirigées vers un sommet de type \og selle \fg\ (resp. de type puits). Sur
la figure~\ref{figu:comp_Reeb_lum}, les bandes $S_1$, $S_2$ et $S_4$ sont
positives tandis que $S_3$ est négative (représentation du revêtement universel,
feuilletage orthogonal à $K$  en pointillés).

\begin{lemm}
\label{lemm:comp_reeb}
Soit $R$ un ruban maximal de $(T,K)$ bordé par deux bandes $S_1$ et $S_2$
de type II, et soit $\overline{R}$ l'adhérence  de $R$. 
Les propriétés suivantes sont équivalentes :
\begin{enumerate}
\item $\overline{R}$ est une composante de Reeb de l'un des
  feuilletages de lumière de~$T$,
\item $S_1$ et $S_2$ ont le même signe (qui est alors celui de 
$\overline{R}$, avec les conventions ci-dessus),
\item la fonction $\langle K,K\rangle$ est de signe différent sur $S_1$ et
  $S_2$.
\end{enumerate}
\end{lemm}

\begin{proof}
Les géodésiques de lumière d'une bande de type II d'un tore s'accumulent
du côté de la selle, ou si l'on préfère une feuille de lumière
bordant une bande de type II a une holonomie (en tant que feuille d'un des
feuilletages de lumière) \og contractante\fg\ lorsqu'elle est parcourue
en se dirigeant vers la selle.
 Par conséquent $\overline{R}$ est une composante
de Reeb de l'un des feuilletages de lumière de $T$ si et seulement si
$S_1$ et $S_2$ sont de même signe.  Le signe de cette composante
$\overline{R}$ (le bord étant orienté par $K$) est alors celui des $S_i$.
\par

\begin{figure}[h!]
\labellist
\small\hair 2pt
\pinlabel $\widetilde{S}_1$ at -15 220
\pinlabel $\widetilde{S}_2$ at 350 220
\pinlabel $\widetilde{S}_3$ at 210 50
\pinlabel $\widetilde{S}_4$ at 575 50
\endlabellist
\begin{center}
\includegraphics[scale=0.55]{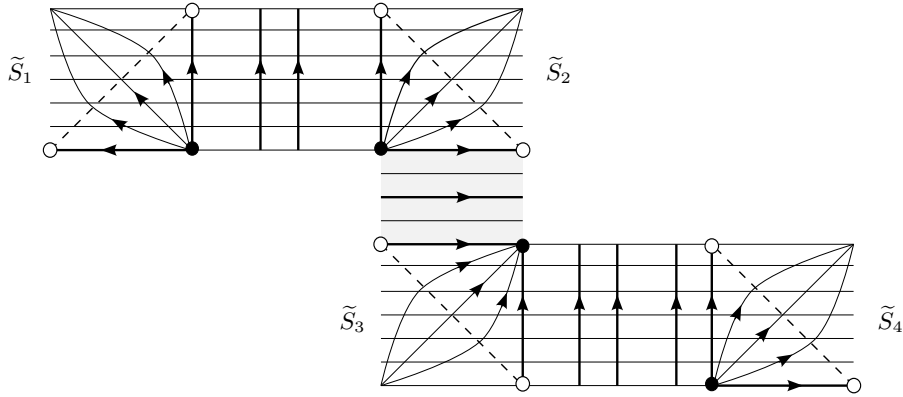}
\caption{Composante de Reeb et de suspension de $\mathcal L$}
\label{figu:comp_Reeb_lum}
\end{center}
\end{figure} 

Comme les champs de lumière sont homotopes aux champs de droites négatif
(ou positifs), la condition (1) est vérifiée
 si et seulement si les cônes font un demi-tour en traversant
$\overline{R}$ le long d'une feuille de $\mcK^\pitchfork$. D'après le
lemme~\ref{lemm:nb_enroul_transv}, cela revient à dire que les signes des 2
bifurcations, c'est-à-dire les signes de $\langle K,K\rangle$ sur $S_1$ et
$S_2$, sont opposés.
\end{proof}

\begin{prop}
 Soit $(T,K)$ un tore lorentzien muni d'un
champ de Killing. On note $k^+$ (resp. $k^-$) le nombre de bandes de type
II positives (resp. négatives) au sens de la définition
\ref{defi:signe_bandeII}. Par ailleurs, soit
$((-1)^{s_1},\ldots,(-1)^{s_{2k}})$ la suite cycliquement ordonnée des
signes de $\langle K,K\rangle$ sur les milieux des bandes de type II et
soit $L$ la longueur de la suite réduite associée. L'entier 
$r(T) \in \N$ caractérisant la composante
de la métrique modulo l'action des difféomorphismes est alors 
donné par 
$$r(T) =
\frac{L}{2} =
\frac{1}{2} \left| \sum_{j=1}^{2k} (-1)^{j+s_j} \right|
= \frac{1}{2} |k^+ - k^-|.$$
\end{prop}

\begin{proof} L'entier $r(T)$ est égal au 
nombre de demi-tours effectués par les cônes de la métrique le long d'une
feuille de $\mcK^\pitchfork$ qui fait un tour dans l'espace des feuilles de
$\mcK$. Si $(T,K)$ n'admet  pas de bandes de type II, les feuilles de
$\mcK^\pitchfork$ sont d'un type donné et $r(T) = 0$. Sinon, le nombre
de demi-tours est donné par le lemme~\ref{lemm:nb_enroul_transv}.
\par

Notons $S_1,\ldots,S_{2k}$ les bandes de type II de $T$, 
dans l'ordre cyclique, en supposant que \mbox{$k\geq 1$.} Les couples
successifs $(S_1,S_2), \ldots, (S_{2k-1},S_{2k})$ bordent $k$ rubans maximaux
$R_1,\ldots,R_k$, feuilletés par l'un des feuilletages de lumière
$\mcL$. Le complémentaire de ces rubans est constitué de composantes de
suspension de $\mcL$ (en gris sur la figure~\ref{figu:comp_Reeb_lum}), 
éventuellement réduites à une feuille fermée.
D'après le
lemme~\ref{lemm:comp_reeb}, le nombre $r^+(\mcL)$ de composantes de Reeb
positives de $\mcL$ est égal au nombre de couples $(S_{2i-1},S_{2i})$ de
signe $(+,+)$, et de même pour $r^-(\mcL)$.  Par suite, on a $|k^+ - k^-| =
2 |r^+(\mcL) - r^-(\mcL)|$.
\end{proof}

\begin{rema} 
Soit $\mcK^\perp$ le feuilletage orthogonal à $K$. Par
  définition du signe des bandes de type II, on a évidemment
$k^+=r^+(\mcK^\perp)$ et $k^-=r^-(\mcK^\perp)$.
\end{rema}

\begin{prop}\label{prop:espace_des_metriques_tore}
 Soit $f$ une fonction périodique. Si $f$ change de signe alors il existe des
tores modelés sur $E^u_f$ dans toutes les composantes connexes de l'espace des
métriques lorentziennes du tore. Sinon les tores modelés sur $E^u_f$
appartiennent tous à la composante connexe des métriques plates.
\end{prop}

\begin{proof}
On déduit du lemme~\ref{lemm:comp_reeb} que si $f$ ne change pas de signe, alors
les feuilletages de lumières de $T$ n'ont pas de composantes de Reeb et donc $T$
appartient à la composante contenant les métriques plates. D'après \cite[\S
2.3]{BM}, pour terminer la preuve il suffit de montrer que lorsque $f$ s'annule
il existe un tore modelé sur $E^u$ dont les feuilletages de lumière ont une
seule composante de Reeb et qu'il existe un tore modelé sur $E^u$ dont l'un des
feuilletages de lumière n'a aucune composante de Reeb.
\par

La fonction $f$ changeant de signe, il existe $ x_1<x_2 $ milieux d'intervalles
de $\R\smallsetminus f^{-1}(0)$ appartenant à une même période et tels que
$f(x_1)f(x_2)<0$. Soit $T$ le tore modelé sur $E^u_f$ associé aux invariants
$(1,0,(\bar f, \{\bar{x}_1,\bar{x}_2\}))$. Ce tore ne contient que deux rubans
maximaux (un pour chaque feuilletage de lumière), bordés par des bandes de
type~II de même signe (lemme~\ref{lemm:comp_reeb}). Chaque feuilletage de
lumière de $T$ a donc une unique composante de Reeb.  Le quotient d'un ruban
maximal de $E^u_f$ par une période de $f$ et un élément du flot (tore associé
aux invariants $(1,0,(\bar f,\emptyset))$) fournit l'autre exemple.
\end{proof}

Pour la bouteille de Klein, les composantes de l'espace des métriques
lorentziennes correspondent bijectivement à $\Z\times \Z/2\Z$. Rappelons
rapidement comment réaliser une telle bijection (voir \cite{BM} pour plus
de détails).  La bouteille de Klein étant considérée comme une fibration en
cercles, on se donne un méridien $\beta$ (c.-à-d. une fibre de la
fibration) et une âme $\alpha$ (c.-à-d. une section de la
fibration). Soit $\xi$ un champ de droites de type temps, vu comme section
du fibré projectif tangent $P$.  Les restrictions $P_{|\beta}$ et
$P_{|\alpha}$ sont des fibrations en cercles sur le cercle, respectivement
tore et bouteille de Klein ; à homotopie près, cette dernière fibration
n'admet que deux sections. Les classes d'homotopie de $\xi_{|\beta}$ et de
$\xi_{|\alpha}$ sont donc déterminées par un degré $n \in \Z$ ($n=0$ pour
$\xi$ tangent à $\beta$) et un élément $\overline{m}\in \Z/2\Z$, et on peut
vérifier que le couple $(n,\overline{m})$ caractérise la classe d'homotopie
de~$\xi$ (\cite[Appendice]{BM}). Plus précisément, on convient de prendre
$\overline{m}=\overline{0}$ si $\xi{|_\alpha}$ est homotope à la section
tautologique (champ tangent à $\alpha$), c.-à-d. $\xi{|_\alpha}$
orientable, et $\overline{m}=\overline{1}$ dans l'autre cas, c'est-à-dire
si $\xi{|_\alpha}$ est homotopiquement transverse à la section
tautologique. Si $n$ est impair, la métrique n'est orientable ni en temps,
ni en espace. Si $n$ est pair avec $\overline{m}=\overline{0}$
(resp. $\overline{m}=\overline{1}$), la métrique est orientable
temporellement mais pas spatialement (resp. spatialement mais pas
temporellement).

\begin{rema}
\label{rema:isotopie}
Les méridiens sont tous isotopes: on peut à tout moment choisir $\beta$
sans modifier $|n|$ et $\overline{m}$.
Il existe exactement deux classes d'isotopie d'âmes; on peut changer le
représentant de $\alpha$ sans modifier $n$ et $\overline{m}$, mais changer
de classe d'isotopie transforme $\overline{m}$ en
$\overline{m}+\overline{n}$.
\end{rema}

Modulo l'action des difféomorphismes, les composantes sont caractérisées
par $n_{\rm abs}=|n|$ pour~$n$ impair et par  $(n_{\rm abs},\overline{m})$ pour 
$n$ pair ; dans ce dernier cas, les difféomorphismes agissent trivialement
sur $\overline{m}$. 
 En particulier, il existe deux composantes contenant les métriques 
plates, indexées avec nos conventions par  $(0,\overline{0})$ et 
$(0,\overline{1})$ ; elles sont fixes par l'action des difféomorphismes
mais permutées par le changement de signe de la métrique.  
\par

Soit $B$ une bouteille de Klein dont le revêtement d'orientation $B^{\rm
  or}$ admet un champ de Killing invariant~$K$. On rappelle que $K$ induit un
feuilletage $\mcK$ sur $B$ ainsi qu'une fonction notée abusivement $\langle
K,K\rangle$, constante sur les feuilles de $\mcK$. En particulier chaque
feuille de~$\mcK$ hérite du signe (-1, 0 ou 1) de $\langle K,K\rangle$.  Vu que
$B$ n'admet pas de feuilletages de lumière (mais seulement un 2-tissu de
lumière, voir~\cite{BM}), la fonction $\langle K,K\rangle$ ne peut être
identiquement nulle. Par ailleurs, le feuilletage transverse de $B^{\rm
  or}$ induit un feuilletage de $B$, notée $\mcK^\pitchfork$, dont toutes
les feuilles sont également fermées, voir~\S\ref{subs:bout_killing}.  On
rappelle que dans le cas non élémentaire, les feuilles de $\mcK^\pitchfork$
sont des géodésiques de lumière, nécessairement brisées puisque $B$ n'a pas
de feuilletage de lumière lisse; le lieu de bifurcation de
$\mcK^\pitchfork$ est une réunion finie de feuilles de~$\mcK$ (milieux des
bandes de type~II), de signe $\pm 1$.  \par

On distinguera comme plus haut deux types de bouteilles $B$. Quand $B$ est
de type~1, le feuilletage $\mcK$ est un feuilletage par méridiens (donc
non orientable) et l'espace de ses feuilles  est un cercle
; le feuilletage transverse $\mcK^\pitchfork$ est orientable, l'espace de
ses feuilles est un segment dont les extrémités sont les âmes de $B$.
Pour le type~2, les rôles de $\mcK$ et $\mcK^\pitchfork$ sont inversés : 
$\mcK$ est orientable, l'espace de ses feuilles est un segment dont les
extrémités sont les âmes de $B$ et $\mcK^\pitchfork$ est un feuilletage
par méridiens, non orientable, dont l'espace des feuilles est un cercle.

\begin{prop}
\label{prop:comp_bout_elem}
Toute bouteille de Klein $(B,\mcK)$ élémentaire appartient à l'une des deux
composantes connexes des métriques plates.  De plus, $(B,\mcK)$ est
temporellement orientable si et seulement si $\mcK$ est du signe de $(-1)^{\tau+1}$ pour la
métrique (c.-à-d. \mbox{$(-1)^\tau \langle K,K\rangle < 0$)}, où
$\tau\in\{1,2\}$ désigne le type de $B$.
\end{prop}

\begin{proof}
Le feuilletage $\mcK^\pitchfork$ est ici lisse, formé des géodésiques
orthogonales à $\mcK$.  Soit $\xi$ un champ de droites de type temps de $B$. On
choisit comme méridien $\beta$ une feuille de $\mathcal{K}$ si $\tau=1$ ou une
feuille de $\mcK^\pitchfork=\mcK^\perp$ si $\tau=2$ (voir la remarque
\ref{rema:isotopie}). Ce méridien est d'un type donné (espace ou temps), donc
$\xi|_\beta$ est de degré $0$ et $B$ appartient à une composante contenant des
métriques plates. Les signes de $\mcK$ et $\mcK^\pitchfork$ sont évidemment
opposés et $B$ est temporellement orientable si et seulement si $\mcK^\pitchfork$ (resp.
$\mcK$) est négatif pour  $\tau =1$ (resp.  pour $\tau =2$).
\end{proof}

\begin{prop}
\label{prop:comp_bout_NE}
Soit $(B,\mcK)$ une bouteille de Klein non élémentaire de type~$\tau\in \{1,2\}$
et soit $\zeta_1\cup \ldots \cup \zeta_k$ le lieu de bifurcation de
$\mcK^\pitchfork$. Si $\tau=2$, on suppose que les $\zeta_i$ sont indexées
dans l'ordre croissant, l'espace des feuilles de $\mcK$ étant orienté (en
particulier $\zeta_1$ et $\zeta_k$ sont les âmes de $B$).
\begin{enumerate}
\item Si $\tau = 1$, alors  $B$ appartient à l'une des deux
composantes connexes des métriques plates. De plus,   $B$ est temporellement
orientable si et seulement si le nombre de  $\zeta_i$ négatives 
est pair.
\item Si $\tau=2$, on note  $(-1)^{s_1},\ldots,(-1)^{s_k}$ 
la suite des signes des $(\zeta_j)$. 
L'invariant entier $n_{\rm abs}(B) \in \N$ de la composante
de $B$ modulo les difféomorphismes est donné par
$$
n_{\rm abs}(B) =  \left\{ 
\begin{array}{lll}
  L_{]1,k[} & \mathrm{si}  &  s_1+s_k+ k ~\mathrm{est ~pair}, \\
  L_{[1,k[} & \mathrm{si}  &  s_1+s_k+ k ~\mathrm{est ~impair}, 
\end{array}
\right.
$$ où $L_{]1,k[}$ (resp. $L_{[1,k[}$) désigne la longueur de la suite
    réduite associée à $(-1^{s_j})_{1<j<k}$ (resp. $(-1^{s_j})_{1\leq
      j<k}$).  De plus, $B$ est temporellement (resp. spatialement)
    orientable si et seulement si $n_{\rm abs}(B)$ est pair et si les âmes
    $\zeta_1$ et
    $\zeta_k$ sont négatives (resp. positives).
\end{enumerate}
\end{prop}

\begin{proof}
Pour le type~1, la vérification de la première assertion est identique à
celle de la proposition~\ref{prop:comp_bout_elem}, à ceci près que le
méridien $\be$ peut ici être de lumière. Il reste à étudier l'orientabilité
temporelle de $B$.  Deux des feuilles de $\mcK^\pitchfork$ sont des âmes
de~$B$. Soit $\al$ l'une d'entre elles et soit $\xi^-$ un champ négatif
au-dessus d'un paramétrage de $\al$, comme dans le
lemme~\ref{lemm:nb_enroul_transv}. D'après la preuve de celui-ci, $\xi^-$
ne change de côté de $\al$ qu'au passage de chaque bifurcation positive,
d'où le résultat (on rappelle que $k$ est impair).  \par

Supposons que $\tau=2$. On prend comme méridien $\beta$ une feuille de 
$\mcK^\pitchfork$.  L'entier $n_{\rm abs}(B)$ est le nombre (non
orienté) de demi-tours que font les cônes au-dessus de $\beta$.
Quand on parcourt~$\beta$, on rencontre une fois les âmes $\zeta_1$ et 
$\zeta_k$ et deux fois les autres $(\zeta_i)$, dans l'ordre inverse. 
Par le lemme~\ref{lemm:nb_enroul_transv}, on a donc
\begin{eqnarray}\nonumber
 n_{\rm abs}(B) &  = & \frac{1}{2} \left| (-1)^{1+s_1} +
 \sum_{1<j<k} (-1)^{j+s_j} + (-1)^{k+s_k} + 
\sum_{1<j<k} (-1)^{2k- j+s_j} \right| \\ \nonumber
 & = & \left| \frac{1}{2} \left[
  (-1)^{k+s_k}-(-1)^{s_1}\right] 
+ \sum_{1<j<k} (-1)^{j+s_j} \right|,
\end{eqnarray}
d'où l'on déduit aisément la formule cherchée.  La dernière assertion est
claire. 
\end{proof}

\begin{rema}
\label{rema:n_s_1_s_2}
Les longueurs $L_{]1,k[}$ et $L_{[1,k[}$ ayant respectivement la parité de
    $k$ et $k+1$, on voit que $n_{\rm abs}(B)$ est congru à $s_1+s_2$
    modulo~2. On retrouve le fait que les âmes ont le même signe quand
    $n_{\rm abs}(B)$ est pair.
\end{rema}

\begin{coro}
\label{coro:comp_type1_f}
 Soit $f$ une fonction périodique et soit 
$\mathfrak{C}_1(f)\subset \Z\times \Z/2\Z$ l'ensemble indexant
les composantes des bouteilles de Klein de type 1 modelées sur $E^u_f$.
\begin{enumerate} 
\item Si $f\geq 0$ (resp. $f\leq 0$), alors $\mathfrak{C}_1(f) = 
\{(0,\overline{0})\}$ (resp. $\mathfrak{C}_1(f) =  \{(0,\overline{1})\})$.
\item Si $f$ change de signe, alors 
$\mathfrak{C}_1(f) = \{(0,\overline{0}),(0,\overline{1})\}$.
\end{enumerate}
\end{coro}

\begin{proof} Si $f\geq 0$, alors le champ $\mcK$ (non orientable) 
est de type espace ou lumière : la métrique est nécessairement
temporellement orientable. À l'inverse, si $f\leq 0$ alors $B$ est
spatialement orientable.  Si $f$ change de signe, on choisit $t_0>0$ et
$x_1$ au milieu d'un intervalle de $\R\smallsetminus f^{-1}(0)$.  D'après
la proposition~\ref{prop:comp_bout_NE}, la bouteille associée à
$(t_0,(\bar f,\{\bar{x}_1\}))$ est temporellement (resp. spatialement)
orientable si $f(x_1)>0$ (resp. $f(x_1)<0$).
\end{proof}

\begin{coro}
\label{coro:comp_type2_f}
 Soit $f$ une fonction paire, périodique de plus petite période $2\msfm>0$
 et vérifiant $f(0)f(\msfm)\neq 0$, et soit 
$\mathfrak{C}_2(f)\subset \Z\times \Z/2\Z$ l'ensemble indexant
les composantes des bouteilles de Klein de type 2 modelées sur $E^u_f$.
\begin{enumerate} 
\item Si $f\geq 0$ (resp. $f\leq 0$), alors $\mathfrak{C}_2(f) = 
\{(0,\overline{1})\}$ (resp. $\mathfrak{C}_2(f) =  \{(0,\overline{0})\}$.
\item Si $f$ change de signe et vérifie $f(0)f(\msfm)>0$, alors 
$\mathfrak{C}_2(f) = 2\Z\times \{\overline{0}\}$ si $f(0)<0$ et 
$\mathfrak{C}_2(f) = 2\Z\times \{\overline{1}\}$ si $f(0)>0$.
\item Si  $f(0)f(\msfm)<0$, alors 
$\mathfrak{C}_2(f)= \Z\times \Z/2\Z$ (toute
  composante de l'espace des métriques contient des bouteilles de
  type 2 modelées sur $E^u_f$).
\end{enumerate}
\end{coro}

\begin{proof} Noter que les âmes des bouteilles de type 2  modelées sur
  $E^u_f$ doivent correspondre à des points fixes des involutions de
  $\Isom(f)$, c'est-à-dire à des éléments de $\msfm\Z$.  Par ailleurs, si
  l'on peut réaliser une valeur de $n_{\rm abs}$, on pourra toujours
  réaliser les valeurs $n=\pm n_{\rm abs}$ (sans changer $\overline{m}$)
  grâce à l'action d'un difféomorphisme homotope à une symétrie par
  rapport à un méridien, \cite[p.\,489]{BM}.  \par

Si  $f$ est de signe constant, alors la suite des signes des $(\zeta_j)$ est
  constante et les suites réduites de $(-1^{s_j})_{1<j<k}$ (pour $k$ pair)
  et de $(-1^{s_j})_{1\leq j<k}$ (pour $k$ impair) sont vides. On a donc
  $n_{\rm abs}(B)=0$ et le signe des âmes est celui de $f$, d'où l'assertion
  (1).  \par

Si $f$ change de signe et vérifie $f(0)f(\msfm)>0$, les âmes ont à nouveau
le même signe. Par conséquent  $n_{\rm abs}(B)$ est pair
(remarque~\ref{rema:n_s_1_s_2}) et le champ de cônes du signe de $f(0)$ est
orientable. Soit $t_0>0$ et soit  $x_1\in ]0,\msfm[$ tel que
$f(0)f(x_1) <0$, situé au
milieu d'un intervalle de $\R\smallsetminus f^{-1}(0)$. La bouteille 
associée à $(t_0,(\bar f,\{\bar{0},\bar{\msfm}\}))$
satisfait
$n_{\rm abs} = 0$ et la bouteille~$B$ associée à $(t_0,(\bar
f,\{\bar{0},\bar{x_1},\bar{\msfm}\}))$ satisfait $n_{\rm abs}(B) = 2$,
proposition~\ref{prop:comp_bout_NE}-(2).  Par suite, le revêtement cyclique
$B_d$ de degré $d\in \N^*$ de $B$ dans le sens du méridien vérifie $n_{\rm
  abs}(B_d) = 2d$.    Toutes les composantes possibles sont donc atteintes.
\par

Si $f(0)f(\msfm)<0$, alors la bouteille $B'$ de paramètre $(t_0,(\bar
f,\{\bar{0},\bar{\msfm}\}))$ avec $t_0>0$ vérifie $n_{\rm abs}(B')=1$,
proposition~\ref{prop:comp_bout_NE}-(2). Pour un revêtement cyclique $B'_d$
comme ci-dessus, on~a $n_{\rm abs}(B'_d)=d$ ($d\in \N^*$). Soit $\Phi$ un
difféomorphisme qui permute les âmes à isotopie près.  Si $d$ est impair,
l'action de $\Phi$ sur $B'_d$ modifie $\overline{m}$ par changement de
l'âme de référence, \mbox{\cite[remarque~A.1]{BM}}. On atteint ainsi toutes
les composantes avec $n_{\rm abs}$ impair. 
Considérons maintenant $f$ 
comme une fonction $4\msfm$-périodique v\'erifiant $f(0)f(2\msfm)>0$. 
Quitte \`a remplacer $f(x)$ par $f(x+\msfm)$, on peut choisir le signe
de $f(0)$ : l'assertion (2) montre donc 
que toutes les composantes avec $n_{\rm abs}$ pair sont atteintes.
\end{proof}

\subsection{Points conjugués}\label{sect:points_conj}

\begin{lemm}
\label{lemm:pts_conj}
Soit $I$ un intervalle ouvert de $\R$ et soit $f\in \mcC^\infty(I,\R)$.  On
suppose qu'il existe $C>0$ et $\eps=\pm 1$ tels que $\{\epsilon f<C^2\}$
disconnecte $I$ (c'est le cas par exemple si $f$ 
change de signe au moins deux fois).
Alors la métrique lorentzienne $2dxdy + f(x)dy^2$
définie sur $I\times \R$ a 
des points conjugués. En particulier, tout ruban ouvert contenant deux 
demi-bandes (voir remarque~\ref{rema:ouv_sat_carre})
de même signe, ou une bande et deux demi-bandes, admet des points conjugués.
\end{lemm}

\begin{proof}
Le champ de Killing $\partial_y$ est un champ de Jacobi le long de toute
géodésique.  Nous allons montrer l'existence d'une géodésique tangente à
$\partial_y$ en deux points distincts et distincte d'une orbite de
$\partial_y$. On peut supposer que $\epsilon C^2$ est une valeur régulière de $f$.
Soit $]a,b[$ une composante connexe relativement compacte
    de $\{\epsilon f<C^2\}$. On a forcément $f(a)=f(b)=\epsilon C^2$ et
    $f'(a)f'(b)\neq 0$.  Soit $\gamma$ la géodésique telle que
    $\gamma(0)=(a,0)$ et $\gamma'(0)=\epsilon \partial_y/C$.  L'écriture
    des deux intégrales premières (l'énergie et Clairaut) donne :
$$\left\{
\begin{array}{ll}
f(x)y'+x'=C,\\
f(x){y'}^2+2x'y'=\epsilon.
\end{array}
\right.$$ On en déduit que ${x'}^2=C^2-\epsilon f(x)$ d'où
$\epsilon f(x)\leq C^2$. Comme $f'(a)\neq 0$, $x'$ n'est pas identiquement
nulle.  Pour $t>0$ et tant que $x(t)$ reste inférieur à $b$, on a
$x'(t)>0$ et donc $y'=\frac{1}{f}(C- \sqrt{C^2-\epsilon f})$.  On en
déduit que $y'$ reste borné (même si $f$ s'annule) tant que $x$ ne
prend pas la valeur $b$, puis  que $x(t)$ tend vers $b$. Pour voir
que la valeur $b$ est atteinte, on considère la géodésique
$\gamma_1=(x_1,y_1)$ vérifiant $\gamma_1(0)=(b,0)$ et $\gamma'_1(0)=
\epsilon\partial_y /C$. Il existe $t_0>0$ et $t_1<0$ tels que
$x(t_0)>x_1(t_1)$. Par conséquent il existe $s\in \R$ tel que $\gamma$ et
$\Phi_{\partial_y}^s\circ \gamma_1$ se croisent. Par suite les géodésiques
$\gamma$ et $\gamma_1$ coïncident au paramétrage près et $\gamma$
est bien deux fois tangente à $\partial_y$.
\end{proof}

\begin{prop}
\label{prop:pts_conjugues_Eu}
 Soit $(T,K)$ un tore lorentzien muni d'un champ de Killing. Si $T$ n'est pas
plat, alors l'extension réflexive $E^u_f$ associée à $(T,K)$ a des points
conjugués.
\end{prop}

\begin{proof}
 L'extension $E^u_f$ contient des rubans isométriques à  $(\R^2,2dxdy +
f(x)dy^2)$
avec~$f$ périodique. Si $f$ n'est pas constante (c.-à-d. si $T$ n'est pas plat),
alors ou bien $f$ change une infinité de fois de signe, ou bien $f$ ou $-f$
admet un minimum positif ou nul, distinct de son maximum. Dans les deux cas, le
lemme~\ref{lemm:pts_conj} permet de conclure que $E^u$ possède des points
conjugués.
 \end{proof}

\begin{theo}
\label{theo:comme_CP}
 Soit $f\in \mcC^\infty(\R,\R)$ une fonction
périodique non constante et soit $T$ un tore modelé sur $E^u_f$. Si $T$
est sans points conjugués alors
 \begin{enumerate}
  \item l'ensemble  de composantes connexes de $\{f\neq 0\}$
    est localement fini,
  \item toute composante définit une bande de type II de~$T$,
  \item $f$ est de signe différent sur deux composantes consécutives,
  \item $f'$ ne change de signe qu'une fois sur chaque composante.
 \end{enumerate}
En particulier, $T$ n'est pas homotope à un tore  plat.
\end{theo}

\begin{proof}
D'après le lemme~\ref{lemm:pts_conj}, la fonction $f$ doit changer de signe,
sans quoi $f$ ou $-f$ admettrait un minimum positif ou nul ($f$ étant
périodique). Rappelons que les composantes de $\{f\neq 0\}$ (modulo une certaine
période de $f$) correspondent aux bandes de~$T$. Soit $R$ un ruban maximal de
$T$. Toujours grâce au lemme~\ref{lemm:pts_conj}, on voit que $R$ ne peut
contenir plus de deux demi-bandes, d'où (1). En particulier $R$ ne peut contenir
de bandes de type I, d'où~(2). Finalement, l'adhérence de tout ruban maximal de
$T$ doit être formée de deux bandes de types II séparées par une composante
connexe de $\{f=0\}$ et vérifiant (3) et (4) (encore le
lemme~\ref{lemm:pts_conj}). Par conséquent, la suite des signes de $f$ sur les
bandes de type~II est forcément réduite (définition~\ref{defi:suite_reduite}) et
$T$ n'est pas homotope à un tore plat.
\end{proof}

\begin{rema}
La fonction sinus est l'une des fonctions les plus naturelles satisfaisant
les assertions (1), (3) et (4) du théorème~\ref{theo:comme_CP}.  Les tores
model\'es sur $E^u_{\sin}$ satisfaisant (2) sont (à homothétie près) les
tores de Clifton-Pohl.  Nous avons montré qu'un tore possédant un champ de
Killing, sans points conjugués et non plat doit ressembler à un tore de
Clifton-Pohl.  Rappelons que ceux-ci n'ont pas de points
conjugués~\cite{BM}.
\end{rema}

\begin{prop}\label{coro_adherence}
 Soit $T$ un tore  lorentzien sans points conjugués et ayant un champ de Killing
non trivial. Alors il existe une suite de tores lorentziens possédant des points
conjugués et convergeant vers $T$ pour la topologie $C^\infty$. 
\end{prop}

\begin{proof} Si $T$ n'est pas plat et n'a pas de points conjugués, alors,
d'après le lemme~\ref{lemm:pts_conj}, la fonction $\langle K,K\rangle$ change de
signe à chaque fois qu'elle s'annule. Si la fonction $\langle K,K \rangle$
s'annule sur un ouvert, alors on peut la perturber de fa\c con à faire
apparaître un minimum local positif et donc des points conjugués. Sinon tous les
zéros sont isolés et, quitte à perturber la fonction, on peut supposer qu'il
existe une bande bordée par deux orbites de $K$ le long desquelles la dérivée de
$\langle K,K \rangle$ ne s'annule pas. Cette bande  se complète alors en un
carré ayant deux vrais points selles. Toutes les géodésiques perpendiculaires à
$K$ de ce carré passent alors  par ces points selles. On obtient ainsi deux
points conjugués sur le bord de $\widetilde T$. On répète alors la preuve de la
proposition 2.3 de \cite{BM} qui consiste à perturber la métrique de façon à
rapprocher ces points conjugués. 

Enfin, si $T$ est plat, il possède un champ de Killing périodique $K$. On
construit par perturbation conforme une suite de tores non plats invariants par
le flot de $K$ qui converge vers $T$.  Par le théorème~\ref{theo:comme_CP}, ces
tores ont des points conjugués.
\end{proof}

\bibliographystyle{plain}
\bibliography{extension_2016}

\begin{thebibliography}{10}

\bibitem{BM}
Christophe Bavard and Pierre Mounoud.
\newblock Sur les surfaces lorentziennes compactes sans points conjugu\'es.
\newblock {\em Geom. Topol.}, 17(1):469--492, 2013.

\bibitem{Bourbaki_LieIV}
N.~Bourbaki.
\newblock {\em \'{E}l\'ements de math\'ematique. {F}asc. {XXXIV}. {G}roupes et
  alg\`ebres de {L}ie. {C}hapitre {IV}: {G}roupes de {C}oxeter et syst\`emes de
  {T}its.}
\newblock Actualit\'es Scientifiques et Industrielles, No. 1337. Hermann,
  Paris, 1968.

\bibitem{EKS1984}
Allan~L. Edmonds, Ravi~S. Kulkarni, and Robert~E. Stong.
\newblock Realizability of branched coverings of surfaces.
\newblock {\em Trans. Amer. Math. Soc.}, 282(2):773--790, 1984.

\bibitem{Ferrand}
Jacqueline Ferrand.
\newblock The action of conformal transformations on a {R}iemannian manifold.
\newblock {\em Math. Ann.}, 304(2):277--291, 1996.

\bibitem{Goldman}
William~M. Goldman.
\newblock Geometric structures on manifolds and varieties of representations.
\newblock In {\em Geometry of group representations ({B}oulder, {CO}, 1987)},
  volume~74 of {\em Contemp. Math.}, pages 169--198. Amer. Math. Soc.,
  Providence, RI, 1988.

\bibitem{GPR}
Manuel Guti{\'e}rrez, Francisco~J. Palomo, and Alfonso Romero.
\newblock Lorentzian manifolds with no null conjugate points.
\newblock {\em Math. Proc. Cambridge Philos. Soc.}, 137(2):363--375, 2004.

\bibitem{HR1957}
Andr{\'e} Haefliger and Georges Reeb.
\newblock Vari\'et\'es (non s\'epar\'ees) \`a une dimension et structures
  feuillet\'ees du plan.
\newblock {\em Enseignement Math. (2)}, 3:107--125, 1957.

\bibitem{Matveev}
Vladimir~S. Matveev.
\newblock Pseudo-{R}iemannian metrics on closed surfaces whose geodesics flows
  admit nontrivial integrals quadratic in momenta, and proof of the projective
  {O}bata conjecture for two-dimensional pseudo-{R}iemannian metrics.
\newblock {\em J. Math. Soc. Japan}, 64(1):107--152, 2012.

\bibitem{Monclair}
Daniel Monclair.
\newblock Isometries of {L}orentz surfaces and convergence groups.
\newblock {\em Math. Ann.}, 363(1-2):101--141, 2015.

\bibitem{MS}
Pierre Mounoud and Stefan Suhr.
\newblock Spacelike {Z}oll surfaces with symmetries.
\newblock J. Diff. Geom, à paraître (arXiv: 1402.5377).

\bibitem{Nomizu}
Katsumi Nomizu.
\newblock On local and global existence of {K}illing vector fields.
\newblock {\em Ann. of Math. (2)}, 72:105--120, 1960.

\bibitem{Oneill}
Barrett O'Neill.
\newblock {\em Semi-{R}iemannian geometry. With applications to relativity},
  volume 103 of {\em Pure and Applied Mathematics}.
\newblock Academic Press, Inc. [Harcourt Brace Jovanovich, Publishers], New
  York, 1983.

\bibitem{Palais1961}
Richard~S. Palais.
\newblock On the existence of slices for actions of non-compact {L}ie groups.
\newblock {\em Ann. of Math. (2)}, 73:295--323, 1961.

\bibitem{Penrose}
Roger Penrose.
\newblock Gravitational collapse and space-time singularities.
\newblock {\em Phys. Rev. Lett.}, 14:57--59, 1965.

\bibitem{Piccione-Zeghib}
Paolo Piccione and Abdelghani Zeghib.
\newblock Actions of discrete groups on stationary {L}orentz manifolds.
\newblock {\em Ergodic Theory Dynam. Systems}, 34(5):1640--1673, 2014.

\bibitem{Sanchez}
Miguel S{\'a}nchez.
\newblock Structure of {L}orentzian tori with a {K}illing vector field.
\newblock {\em Trans. Amer. Math. Soc.}, 349(3):1063--1080, 1997.

\bibitem{Steenrod1951}
N.~Steenrod.
\newblock {\em The {T}opology of {F}ibre {B}undles}.
\newblock Princeton Mathematical Series, vol. 14. Princeton University Press,
  Princeton, N. J., 1951.

\bibitem{wolf}
Joseph~A. Wolf.
\newblock {\em Spaces of constant curvature}.
\newblock AMS Chelsea Publishing, Providence, RI, sixth edition, 2011.

\end{thebibliography}

\bigskip
\begin{tabular}{ll}
 Adresse: & Univ. Bordeaux, IMB, UMR 5251, F-33400 Talence, France\\
& CNRS, IMB, UMR 5251, F-33400 Talence, France\\

E-mails:& {\tt christophe.bavard@math.u-bordeaux1.fr}\\
&{\tt pierre.mounoud@math.u-bordeaux1.fr}
\end{tabular}

\end{document}